%% file: Untitled.tex
\documentclass[10pt,a4paper]{amsart}
\usepackage{fullpage}
\usepackage{mathrsfs}
\usepackage[left=1.0in, right=1.0in, top=1in, bottom=1.3in, includefoot, headheight=13.6pt]{geometry}
\usepackage[colorlinks=true,hyperindex,citecolor=blue,linkcolor=black]{hyperref}
\input{xy}
\usepackage{cite}
\usepackage{tikz-cd}
\usepackage[capitalize]{cleveref}
\usepackage{eucal,times,amsmath,amsthm,amssymb,mathrsfs,stmaryrd,color,enumerate,accents}
\usepackage{tensor}

\usepackage{textcomp}
\usepackage{mathtools,bm,eucal} 
\usepackage[T1]{fontenc} 
\usepackage{enumerate,comment,braket,xspace,csquotes} 

 \numberwithin{equation}{subsection}

\newtheorem{theorem}{Theorem}[subsection]
\newtheorem{lemma}[theorem]{Lemma}
\newtheorem{corollary}[theorem]{Corollary}
\newtheorem{proposition}[theorem]{Proposition}

\newtheorem{thm-intro}{Theorem}

\newtheorem{ap-theorem}{Theorem}
\newtheorem{ap-corollary}{Corollary}
\newtheorem{ap-lemma}{Lemma}
\newtheorem{ap-proposition}{Proposition}
\newtheorem{ap-definition}{Definition}

\theoremstyle{definition}
\newtheorem{remark}[theorem]{Remark}
\newtheorem{example}[theorem]{Example}
\newtheorem{notation}[theorem]{Notation}
\newtheorem{definition}[theorem]{Definition}
\newtheorem{construction}[theorem]{Construction}
\newtheorem{warning}[theorem]{Warning}

\newcommand{\cXysquare}[8]{
\[\cXymatrix{
#1 \ar@{#5}[r] \ar@{#6}[d] & #2 \ar@{#7}[d]\\
#3 \ar@{#8}[r] & #4
}\]
}

\def\lim{\mathrm{lim}}

\DeclareRobustCommand{\SkipTocEntry}[5]{}

\input{newcommands}

\author{Jorge Ant\'onio}
\date{}
\newcommand{\adress}{{
  \bigskip
  \footnotesize

\textsc{Jorge Ant\'onio,  IRMA, UMR 7501
 7 rue René-Descartes
 67084 Strasbourg Cedex}\par\nopagebreak
  \textit{E-mail address}:  \texttt{antonio@math.unistra.fr}

}}

\title{Moduli of $p$-adic representations of a profinite group} 
\begin{document}

\maketitle

\begin{abstract} Let $X$ be a smooth and proper scheme over an algebraically closed field.
The purpose of the current text is twofold. First, we construct the moduli stack parametrizing rank $n$ continuous $p$-adic representations of the \'etale fundamental group $\pi_1^\et(X)$. Our construction realizes such object as a $\bQ_p$-analytic stack, denoted $\LocSys(X)$.
Secondly, we prove that $\LocSys(X)$ admits a canonical derived structure. This derived structure allow us to intrinsically recover the deformation theory of continuous $p$-adic representations, studied in \cite{galatius2018derived}.
Our proof of geometricity of $\LocSys(X)$ uses in an essential way the $\bQ_p$-analytic analogue of Lurie-Artin representability,
proved in \cite{porta2017representability}.
\end{abstract}

\tableofcontents

\section{Introduction}
\subsection{Main results} The study of continuous $p$-adic representations of \'etale fundamental groups has been a central subject in both number theory and algebraic geometry. In \cite{mazur1989deforming}, B. Mazur first studied the
obstruction theory of continuous $p$-adic representations of absolute Galois groups, bringing the techniques of deformation theory into the realm of number theory. More recently, S. Galatius and A. Venkatesh generalized much of Mazur's work by
studying the (derived) deformation theory of continuous representations of absolute Galois groups, see \cite{galatius2018derived}. Nonetheless, a construction of a global moduli space parameterizing continuous representations of \'etale fundamental
groups was still lacking. The existence of such object and its role in counting arithmetic $\ell$-adic sheaves on smooth and proper schemes, in positive characteristic, was conjectured by P. Deligne, in \cite{deligne2015comptage}. 

Let $X$ be a smooth and proper scheme over an algebraically closed field. The goal of the present text is to construct the moduli space, envisioned by P. Deligne, parameterizing $p$-adic \'etale lisse sheaves on $X$ and study its corresponding deformation theory. More precisely,
we will construct a functor
	\[
		\LocSys(X) \colon \Afd_{\bQ_p}^\op \to \mathrm{Grpd},
	\]
where $\Afd_{\bQ_p}$ denotes the category of $\bQ_p$-affinoid spaces and $\mathrm{Grpd}$ the category of groupoids, given by the formula
	\[
		A \in \Afd_k^\op \mapsto  \LocSys(X)(A) \in \mathrm{Grpd}.
	\]
Here $\LocSys(X)(A)$ denotes the groupoid of conjugation classes of continuous group homomorphisms
	\[
		\rho \colon\pi_1^\et(X) \to \GLn(A).
	\]
We endow $\GLn(A)$ with the topology induced by the non-archimedean topology on $A \in \Afd_k^\op$. Our first main result is the following:

\begin{theorem}[\cref{thm:Theorem_1}, \cref{scholze&bhatt}, \cref{loc_vs_rep}]
The moduli functor 
	\[
		\LocSys(X) \colon \Afd_{\bQ_p}^\op \to \mathrm{Grpd},
	\]
is representable by a $\bQ_p$-analytic stack. More precisely, there exists a $k$-analytic space $\LocSysfr(X) \in \An_k$ together with a canonical smooth map
	\[
		\pi \colon \LocSysfr(X) \to \LocSys(X),
	\]
which exhibits $\LocSysfr(X) $ as a smooth atlas for $\LocSys(X)$. More concretely, $\pi$ exhibits $\LocSys(X)$ as the $\bQ_p$-analytic quotient stack of $\LocSysfr(X)$ under the conjugation action of $\anGLn$ on $\LocSysfr(X)$. Moreover,
the groupoid
	\[\LocSys(X)(A) \in \mathrm{Grpd},\]
can be naturally identified with the groupoid of rank $n$ pro-\'etale $A$-local systems on $X$.
\end{theorem}

The $\bQ_p$-analytic space $\LocSysfr(X)$ is first constructed via its functor of points: 
	\begin{equation} \label{first}
		\LocSysfr(X) \colon A \in \Afd_{\bQ_p}^\op \mapsto \Hom_{\cont} \big( \pi_1^\et(X), \GLn(A) \big) \in \mathrm{Set}.
	\end{equation}
The proof that the above formula defines a functor which is representable by a $\bQ_p$-analytic space will occupy the first part of \S 2. Our proof is inspired by the complex case, see \cite{Simpson_Non_abelian_HT_II}.

Nonetheless, the fact that we have to keep track of both the profinite and $\bQ_p$-analytic topologies on $\pi_1^\et(X)$ and $\GLn(A)$, respectively, renders our arguments considerably more involved. 
We will exploit, at length, the fact that the $
\bQ_p$-analytic topology on the group $\anGLn(A)$, $A \in \Afd_{\bQ_p}$, admits a natural description in terms of ind-pro objects. Indeed, the pro-structure
arises from the choice of a formal model $A_0$ for $A$. The ind-structure is induced by the formula,
	\begin{align*}
		A & \simeq A_0 \otimes_{\bZ_p} \bQ_p \\
		   & \simeq \colim_{\textrm{mult by p}} A_0.
	\end{align*}
The second half of \S 2 is devoted to present the theory of $\bQ_p$-analytic stacks and show that $\LocSys(X)$ can be identified with the $\bQ_p$-analytic stack quotient of $\LocSysfr(X) $, under the conjugation action of $\anGLn$. 

We will further use the
the results of \cite{bhatt2013pro}, to naturally identify families of points of $\LocSys(X)$ with rank $n$ pro-\'etale local systems on $X$.

Having constructed the global moduli space $\LocSys(X)$, we then proceed to study the corresponding deformation theory. In order to do so, we will show that $\LocSys(X)$ admits a canonical derived enhancement. More precisely, we prove:

\begin{theorem} \label{intro_thm:derived_enhancement} Let $\dAfd_{\mathbb{Q}_p}$ denote the \infcat of derived $\mathbb{Q}_p$-affinoid spaces, introduced in \cite[7.3]{porta2016derived}. We have the following:
	\begin{enumerate}
	\item (\cref{t1}) The $\bQ_p$-analytic stack 
		\[
			\LocSys(X) \colon \Afd_{\bQ_p}^\op \to \mathrm{Grpd}
		\]
	admits a natural derived enhancement
		\[
			\dLocSys(X) \colon \dAfd_{\bQ_p}^\op \to \cS,
		\]
	where $\cS$ denotes the \infcat of spaces. 
	\item (\cref{computation} and \cref{perfect_etale}) The derived moduli stack $\dLocSys(X)$ admits a global analytic cotangent complex, denoted $\bL^\an_{\dLocSys(X)}$. Given a continuous representation
		\[
			\rho \colon  \pi_1^\emph{\et}(X) \to \GLn(\overline{\bQ}_p),
		\]
	we can naturally identify the analytic cotangent complex at $\rho$ with
		\[
			\bL^\an_{\dLocSys(X), \rho } \simeq C^*_\emphet \big(X, \Ad(\rho ) \big)^\vee[-1] \in \Mod_{\overline{\bQ}_p},
		\]
	where $\Mod_{\overline{\bQ}_p}$ denotes the stable \infcat of $\overline{\bQ}_p$-modules, and $C^*_\emphet \big(X, \Ad(\rho ) \big)^\vee$ denotes the dual of the complex of global \'etale cochains on the associated adjoint representation
		\[
			\Ad( \rho ) \coloneqq \rho \otimes \rho^\vee,
		\]
	regarded naturallly as a pro-\'etale local system on $X$.
	\item (\cref{thm:final_thm}) Moreover, the functor 
	\[
		\dLocSys(X) \colon \dAfd_{\bQ_p}^\op \to \cS
	\]
is representable by a derived (geometric) $\bQ_p$-analytic stack whose $0$-truncation canonically agrees with $\LocSys(X)$.
	\end{enumerate}
\end{theorem}

We observe that the computation of the analytic cotangent complex, $\bL^\an_{\dLocSys(X)}$, provided in \cref{intro_thm:derived_enhancement} is not surprising and matches the complex analytic case, see \cite[Derived mapping stacks, p.41]{toen2006higher}.

A main difficulty in proving \cref{intro_thm:derived_enhancement} is to extend the definition of continuity to derived coefficients. We first observe that given $Z \in \dAfd_{\bQ_p}$, a derived $\bQ_p$-affinoid space,
the derived ring of global sections
	\[
		A \coloneqq \rmR \Gamma(Z, \cO_Z) \in \CAlg_{\bQ_p},	
	\]
admits higher homotopy groups. Since the pro-\'etale fundamental group, $\pi_1^\et(X)$ only classifies pro-\'etale local systems on $X$ with discrete coefficients we will need to replace it with
the pro-\'etale homotopy type, which we denote by $\Sh^\et(X)$. The latter is naturally an object in the \infcat of profinite spaces, denoted $\pro^{\mathrm{fc}}(\cS)$, cf. \cite[Theorem 3.6.5]{2009derived}.

A more serious problem is the fact that the object $A \in \CAlg_{\bQ_p}$ does not come equipped with any meaningful additional \emph{topological structure}. Therefore, we have to define a reasonable notion of a \emph{derived
continuous representation}
	\[
		\rho \colon \Sh^\et(X) \to \rmB \GLn(A).
	\]	
For this purpose, we will consider the object $\rmB \GLn(A)$ roughly as an object naturally enriched in the \infcat of ind-pro-spaces, $\ind(\pro(\cS))$. We then define a derived continuous representation
	\[
		\rho \colon  \Sh^\et(X) \to \rmB \GLn(A),	
	\]
as a morphism of monoid-like objects enriched in $\ind(\pro(\cS))$. 

In order to make the above heuristics into valid statements we need to study in detail the theory of $\ind(\pro(\cS))$-enriched \infcats. This is the main bulk of the second part of the paper.
In \S 
3, we study the theory of enriched \infcats. Our main example of application corresponds to the \infcat of perfect complexes over $A$. After having constructed the required enrichment on $\Perf(A)$, we continue with our study of derived continuous representations in \S 4. We will prove a
main ingredient, of independent interest, namely a derived analogue of the well-known result that a continuous representation
	\[
		\rho \colon \pi_1^\et(X) \to \GLn(E),	
	\]
where $E$ is a finite field extension of $\bQ_p$, can be lifted (up to conjugation) to an integral continuous representation,
	\[
		\rho' \colon \pi_1^\et(X) \to \GLn(\cO_E).
	\]
The reader can find more details in \cref{homotopy1}. Another crucial ingredient for us is the derived Raynaud localization theorem, cf. \cite[Theorem 4.4.10]{antonio2018p}.
Finally, in \S 5, we recall the non-archimedean analogue of the Artin-Lurie representability theorem, proved in \cite[Theorem 7.1]{porta2017representability}. We will then proceed to study in the detail the deformation
theory of $\LocSys(X)$. We will then deduce, out of combining these ingredients together, the representability statement stated in \cref{intro_thm:derived_enhancement} (iii).

\subsection{Notations and Conventions} In the body of the text we will need to enlarge a starting Grothendieck universe, and we often omit such procedure.
Fortunately, this is innocuous for us.

\begin{enumerate} 
\item Let $k$ denote either a finite field extension or a complete algebraically closed field extension of $\bQ_p$. We denote by $\Ok$ its ring of integers and $p \in \Ok$ a uniformizer for $k$. Given $n \ge 1$, we shall denote by $\cO_{k, n}$ the reduction mod $p^n$ of the valuation ring $\Ok$.

\item We denote $\An_k$ the category of
strict $k$-analytic spaces and $\Afd_k$ the full subcategory spanned by strict $k$-affinoid spaces, see \cref{def:analytic_space}.

\item We refer the reader to  \cite[\S 2.6]{berkovich1993etale} for a definition of the Berkovich's analytification functor. We denote $\anGLn$ denotes the analytification of the usual general linear group scheme, $\GLn$, over $k$.

\item In this paper we extensively use the language of \infcats. Most of the times, we reason model independently. However, whenever one needs to argue by means of an explicit model, we will use the theory of quasi-categories developped in
\cite{lurie2009higher}. We use caligraphic letters $\cC, \ \mathcal{D}$ to denote \infcats. We denote $\Cat$ the \infcat of (small) \infcats. We further denote by $\Cat^{\mathrm{perf}}$ the \infcat of idempotent complete \infcats whose equivalences correspond
to Morita equivalences. We will denote by $\cS$ the \infcat of spaces, $\cS^\fc$ the \infcat of finite constructible spaces, cf
\cite[\S 3.1]{2009derived}.

\item Given $\cC \in \Cat$, we denote by $\ind(\cC)$ and $\pro(\cC)$ the corresponding \infcats of ind-objects and pro-objects on $\cC$, respectively. When $\cC = \cS^\fc$, the \infcat $\pro(\cS^\fc)$ is referred as the \infcat of \emph{profinite spaces}.

\item Let $\cC^\otimes$ denote a symmetric monoidal \infcat. We shall denote by $\mathbf 1_\cC$ the unit for the symmetric monoidal structure on $\cC$. When $\cC$ is clear from the context, we simply denote $\mathbf 1_\cC$ by $\mathbf 1$.

\item Let $\cV^\otimes$ denote a \emph{presentably symmetric monoidal \infcat}. We shall denote the \infcat of $\cV^\otimes$-enriched \infcats by $\Cat(\cV^\otimes)$. We will mainly follow \cite{gepner2015enriched} as the foundational framework for the theory of enriched \infcats. 

\item Let $R$ be a simplicial commutative ring, which we also refer to as \emph{derived ring}. We will denote by $\CAlg_R$ the \infcat of simplicial $R$-algebras. The latter can be realized as the \infcat associated to the usual category of simplicial $R$-algebras equipped with its usual model structure. For simplicity, we shall refer to objects of $\CAlg_R$ as \emph{derived $R$-algebras}.

\item We shall denote by $\CAlg_{\Ok}^\ad$ the \infcat of \emph{topologically almost of finite presentation} \emph{derived $\Ok$-adic algebras}. We refer the reader to \cite[Definition 3.1.5]{antonio2018p} for the definition of derived $\Ok$-adic algebra and \cite[Definition 3.2.1]{antonio2018p}
for the notion of topologically almost of finite presentation derived $\Ok$-adic algebra.

\item Let $R$ denote a derived ring. We denote $\Mod_R$ the derived $\infty$-category of $R$-modules, $\Coh(R) \subseteq \Mod_R$ the full subcategory spanned by almost perfect $R$-modules and $\Perf(R)$ the full subcategory of perfect $A$-modules.
Let $M , \ N \in \Mod_R$, we denote the (derived) tensor product by $M \otimes_R N \in \Mod_R$.

\item Given $A \in \adCAlg$, we denote by $A_n \coloneqq A \otimes_{\Ok} \cO_{k, n} \in \CAlg_{\cO_{k, n}}$. More generally, given $M \in \Mod_A$ and $n \ge 1$, we denote by $M_n \coloneqq M \otimes_A A_n$ the base change of $M$ along the canonical morphism
$\kc \to \cO_{k, n}$.

\item Let $A \in \adCAlg$. Given $M, N \in \Coh(A \otimes_{\kc} k)$ we shall usually use gothic symbols $\mathfrak M , \mathfrak N \in \Coh(A)$ to denote formal models for $M$ and $N$, respectively. More specifically, $\mathfrak M \in \Coh(A)$
satisfies $ \mathfrak M \otimes_{\kc} k \simeq M$ and similarly for $N$ and $\mathfrak N$.

\item Let $R \in \CAlg$. The \infcat of $R$-modules, $\Mod_R$, is naturally $R$-linear. For this reason, it admits a natural enrichment over $\Mod_R$, itself. Given $M, \ N \in \Mod_R$ we shall denote by 
	\[
		\underline{\Map}_{\Mod_R}(M, N) \in \Mod_R,
	\]
the \emph{mapping $R$-module}. For each $i \in \bZ$, its underlying spectrum has homotopy groups given by
	\[
		\pi_i(\underline{\Map}_{\Mod_R}(M, N) ) \simeq \pi_0(\underline{\Map}_{\Mod_R}(M, N[-i]) ).
	\]

\item In this paper we will make fundamental use of the theory of derived $k$-analytic geometry, developed in \cite{porta2016derived, porta2017representability}. We shall denote by $\dAfd_k$ the \infcat of \emph{derived $k$-afffinoid spaces}, introduced in
\cite[Definition 7.3]{porta2016derived}. Given $Z \in \dAfd_k$ a derived $k$-affinoid space we shall denote by
	\begin{align*}
		\rmR \Gamma(Z, \cO_Z)   &\coloneqq \Map_{\dAn_k}(Z, (\bA^1_k)^\an ) \\
							& \in \CAlg_k,
	\end{align*}
its derived ring of global sections, see \cite[]{porta2016derived}.
\end{enumerate}

\subsection{Acknowledgments} I would like to express my deep gratitude to my advisor B. Toen for all his advice and by sharing many ideas and suggestions concerning the contents of this text. I am grateful to M. Porta from who I learned derived non-
archimedean geometry and for useful mathematical suggestions along these last years.
I would also like to thank to M. Robalo, A. Vezzani, B. Hennion, D. Gepner, M. d'Addezio, V. Daruvar, C. Simpson, V. Melani, P. Scholze, B. Bhatt, B. Conrad, B. Stroh,  and J. Tapia for useful 
commentaries and suggestions during the elaboration of the text. The author also expresses his gratitude to the math departments of Universit\'e Paris 13 and Universit\'e de Strasbourg, where I could discuss and share many of the ideas presented in this text.

\section{Representability of the space of morphisms}
Let $G$ be a profinite group topologically of finite generation. For example, one can consider $G$ to be the \'etale fundamental group of a proper and smooth scheme over an algebraically closed field, see \cite[Thm 2.9, Expos\'e 10]{grothendieck224revetements}. We define the functor
	\[
		\LocSysfr(G) \colon \Afd_{\bQ_p}^\op \to \mathrm{Set},
	\] 
given on objects by the formula
	\[
		\Sp(A) \in \Afd_{k} \mapsto \Hom_{\cont} \big(G, \GLn(A) \big) \in \mathrm{Set}.
	\]
We denote by $\Hom_{\cont} \big(G, \GLn(A) \big)$ the set of continuous group homomorphisms
	\[
		G \to \GLn(A),
	\]
where $\GLn(A)$ is endowed with the $p$-adic topology induced by the one on the $k$-affinoid algebra $A$.
We will prove that $\LocSysfr(G)$ is representable by a $k$-analytic space, i.e. 
	\[
		\LocSysfr(X) \in \An_{k}.
	\]
The proof of representability is established first when $G$ is a free
profinite group. This is the main result of \S 2.2. The case where $G$ is merely a topologically finitely generated profinite group follows by the previous result, see \cref{pr:Hom_G}.

We will further show in \cref{thm:Theorem_1}, that $\LocSysfr(G)$ defines a smooth atlas for the moduli of continuous representations of $G$, $\LocSys(G)$.
Consequently, it follows that $\LocSys(G)$ is representable by a geometric stack with respect to the \emph{$k$-analytic geometric context}. See \cref{geometric_stack_definition} for the definition of the latter term.

\subsection{Preliminaries} We shall first make a brief review of certain basic notions of $k$-analytic geometry, that will be useful for us. 
\begin{definition}
Let $n \geq 1$ be an integer. The \emph{Tate $k$-algebra on $n$ generators with radius $(r_1, \dots, r_n)$} is defined as
	\[
		k \langle r_1^{-1} T_1, \dots, r_n^{-1} T_n \rangle := \{ 
		\sum_{ i_1, \dots, i_n } a_{i_1, \dots, i_n} T_1^{i_1} \dots T_n^{i_n} \in k [[T_1, \dots T_n ]]| \ a_{i_1, \dots, i_n} r_1^{i_1} \dots r_n^{i_n} \to 0 
		\},
	\]
whose multiplicative structure is induced by the multiplicative structure on the formal power series ring $k [[T_1, \dots T_n ]]$.
\end{definition}

\begin{definition}
A \emph{$k$-affinoid algebra} is a quotient of a Tate algebra $k \langle r_1^{-1} T_1, \dots, r_n^{-1} T_n \rangle $ by a finitely generated ideal $I$. 
\end{definition}

\begin{definition}
Let $A$ be a $k$-affinoid algebra we say that $A$ is a \emph{strict $k$-affinoid algebra}
if we can choose such a presentation for $A$ as before with all the $r_i = 1$. We denote by $\cC \Afd_k$ the category of strict $k$-affinoid algebras together with continuous $k$-algebra homomorphisms
between them.
\end{definition}

\begin{remark}The $k$-algebra $k \langle r_1^{-1} T_1, \dots, r_n^{-1} 
T_n \rangle$ admits a canonical $k$-Banach structure induced by the 
usual Gauss norm. Moreover, any finitely generated ideal $I \subset 
k \langle r_1^{-1} T_1, \dots, r_n^{-1} T_n \rangle$ is closed. As a consequence, any $k$-affinoid algebra $A$ admits a $k$-Banach
structure, depending on the choice of a presentation of $A$. 
Nonetheless it is possible to show that any two such $k$-Banach 
structures, on $A$, are equivalent. Therefore, the latter inherits a 
canonical topology, induced from the canonical topology on $k \langle r_1^{-1} 
T_1, \dots, r_n^{-1} T_n \rangle$.
\end{remark}

Strict $k$-affinoid algebras correspond to the affine objects in (rigid) $k$-analytic geometry: 

\begin{definition}The category of \emph{$k$-affinoid spaces} is defined as the opposite category
	\[
		\Afd_k \coloneqq ( \cC \Afd_k)^\op.
	\]
\end{definition} 

\begin{remark}
Let $A \in \cC \Afd_k$ denote a $k$-affinoid algebra. Given a presentation 
for $A$
	\[
		A \cong k \langle T_1, \dots, T_m \rangle / I
	\]	
one is able to construct a formal model for $A$. More precisely, we consider the $p$-adic complete $\Ok$-adic algebra
topological of finite presentation 
	\[
		A_0 \coloneqq \Ok \langle T_1, \dots T_m \rangle / I \cap \Ok  \langle T_1, \dots T_m \rangle.
	\]
It is then clear that $A_0$ satisfies 
	\[
		A \simeq A_0 \otimes_{\Ok} k.
	\]	
\end{remark}

\begin{notation} We will often denote by $A_0$ a formal model for $A$, i.e., a ($p$-adically complete) $\Ok$-algebra of topological
	finite 
	presentation such that we have an isomorphism
		\begin{align*}
			A & \simeq A_0 \otimes_{\Ok} k \\
			    & \simeq  \colim_{\textrm{mult by }p} A_0 .
		\end{align*}
\end{notation}

\begin{remark}
	The ring $A_0$, above, is an open subring of $A$. 
	For this reason, the topology of $A$ can be thought  as an ind-pro topology, in which the pro-structure comes from the fact that formal models are $p$-adically complete and the ind-structure arises after 
	localizing at $p$.
\end{remark}

\begin{definition}
Given a $k$-affinoid algebra $A$, we let $\rmM(A)$ denote
the set of semi-multiplicative seminorms on $A$. Given $x \in \rmM(A)$ we can associate to it a (closed) prime ideal of $A$. Namely, it corresponds to the kernel of the induced morphism
	\[
		x \colon A \to \bR.
	\]
The fact that it defines
a prime ideal of $A$ follows from the multiplicativity of $x \in \rmM(A)$. The elements $x \in \rmM(A)$ are usually referred to as \emph{closed points of $\Sp(A)$}.
\end{definition}

\begin{notation}
We shall denote by $\rmH(x)$ the completion of the residue field $\mathrm{Frac}( A / \mathfrak{p})$, where
$\mathfrak{p} \coloneqq \ker(x)$. The field $\rmH(x)$ possesses a canonical valuation, denoted $|\bullet|_x$, induced by the one on $A$. Given $a \in A$ we denote by $|a|_x \in 
\mathbb{R}$ the evaluation of $| \bullet |_x$ on the image of $a$ in $\rmH(x)$.
\end{notation}

\begin{definition} Let $\phi \colon A \to A'$ denote a bounded morphism of $k$-affinoid algebras. Let $x \in \rmM(A)$ denote a closed point of $\Sp(A)$. We say that $x$ is
\emph{inner with respect to $A$} if there exists a continuous surjective map 
	\[
		\varphi \colon A \langle r_1^{-1} T_1, \dots, r_n^{-1} T_n \rangle  \to A',
	\]
of $k$-affinoid algebras such that the norm on $A'$, induced by $\varphi$, agrees with the usual one and $ \vert T_i \vert_{x'} < r_i$, for each $i$. 
\end{definition}

In Berkovich's non-archimedean geometry it is possible to define the notion of relative interior, which will be very useful for us:

\begin{definition} Let $
		\phi: A \to A',
	$
denote a bounded morphism of $k$-affinoid algebras. The \emph{relative interior} of $\phi$, denoted by $\mathrm{Int}( \rmM(A')/ \rmM(A))$, is by 
definition the set of points
	\[
		\mathrm{Int} \big( \rmM(A')/ \rmM(A) \big) \coloneqq \{ x' \in \rmM(A') \vert \ A' \to \rmH(x') \text{ is inner with respect to }A \}.
	\]
\end{definition}

\begin{construction}
Let $A \in \cC \Afd_k$. Fix an $\Ok$-adic algebra, $A_0$, such that
	\[
		A_0 \otimes_{\Ok} k \simeq A.
	\]
The topology on $A_0$ admits the family $\{ p^n A_0 \}_{n \geq 1}$ as a fundamental family of open 
neighborhoods around $0 \in A_0$.
Consequently, for $m \geq 0$, we have a fundamental family of normal open subgroups 
	\[
		\mathrm{Id} + p^{m+1} \cdot \mathrm{M}_n( A_0) \unlhd \GLn(A_0),
	\]
which induce the $p$-adic topology on $\GLn(A_0)$. We have further canonical isomorphisms 
	\[
		\GLn( A_0) / \big( \mathrm{Id} + p^m \cdot M_n( A_0) \big) \simeq \GLn( A_0/p^m A_0).
	\]
Indeed, the canonical group homomorphism $g_m \colon \GLn(A_0) \to \GLn(A_0/p^m A_0)$ sends the subgroup 
	\[
		 \mathrm{Id} + p^m \cdot M_n( A_0),
	\]
to the identity on $\GLn(A_0/p^m A_0)$. Moreover, one checks directly that if $M \in \GLn(A_0)$ is such that
	\[
		g_m(M) = \mathrm{Id},
	\]
then the matrix $M- \mathrm{Id} $ is necessarily divisible by $ p^m$ component-wise. Moreover, since $A_0$ is $p$-adically complete and $\GLn(\bullet)$ commutes with limits we obtain a natural isomorphism
	\[
		\GLn(A_0) \cong \underset{m \geq 1}{\lim} \big( \GLn( A_0) / \big( \mathrm{Id} + p^m M_n( A_0) \big) \big).
	\]
The same reasoning holds for the topological group
$\mathrm{Id} + p^{m} \cdot \mathrm{M}_n( A_0)$, for $m \geq 1$. More precisely, we have isomorphisms
	\[
		\mathrm{Id} + p^{m} \cdot \mathrm{M}_n( A_0) \cong \underset{s \geq 1}{\lim } \big( \mathrm{Id} + p^{m} \cdot \mathrm{M}_n( A_0)/ ( \mathrm{Id} + p^{m + s}  \cdot \mathrm{M}_n( A_0) \big).
	\]
\end{construction}

\begin{definition} \label{def:analytic_space}
A \emph{$k$-analytic space} is defined as a locally ringed space which locally is equivalent to a $k$-affinoid space. We denote by
$\An_k$ the category of $k$-analytic spaces and morphisms between these. 
\end{definition}

One is then able to globalize most of the previous notions, in particular it is possible to give a global definition of the relative interior of a morphism between $k$-analytic spaces.
We refer the reader to \cite{berkovich1993etale}, \cite{conrad2008several} and \cite{bosch2005lectures} for a more detailed exposition on rigid geometry, from different points of view.

\subsection{Hom spaces}

\begin{definition}
Let $\mathrm{FinGrp}$ denote the category of finite groups. The category of profinite group is defined as the pro-completion of $\pro( \mathrm{FinGrp})$. 
\end{definition}

\begin{notation}
We denote by $\fr$ a fixed free profinite group of rank $r$. It can be explicitly realized as the profinite completion of a free group on $r$ generators, denoted $\rmF_r$. The latter is a dense full subgroup of $\fr$. We will thus fix throughout the text a
continuous 
dense group inclusion homomorphism $\rmF_r \to \fr$. This provide us with a set of topological generators $e_1, \dots , e_r \in \fr$.
\end{notation}

For each $r \geq 1$, the groups $\fr \in \pro(\mathrm{FinGrp})$ satisfy the universal property given by the
formula
	\[ 
		\Hom_{\pro(\mathrm{FinGrp})}(\fr, G) \cong G^r, \quad \text{for any }G \in \pro(\mathrm{FinGrp}).
	\]

\begin{notation}
Let us fix $\mathcal{J}_r$ a final family of normal open subgroups of finite index in $\fr$, i.e., such that we have a continuous group isomorphism,
	\[ 
		\lim_{ U \in \mathcal{J}_r} \fr/ U \simeq \fr.
	\]
\end{notation}

\begin{remark}
Given $U \in \mathcal{J}_r$, the quotient group 
	\[
		\fr / U \simeq \Gamma,
	\]
is finite and thus it is of finite presentation. Moreover, thanks to the Nielsen-Schreier theorem, cf. \cite[Theorem 3.3.1]{ribes2008wreath}, the group 
$U$ is topologically finitely generated free profinite by elements $\sigma_1, \dots \sigma_l \in U$. Consider furthermore the dense group inclusion homomorphism
	\[
		\rmF_r \to \fr, 
 	\] 
then $U \cap \rmF_r \to U$ is a discrete subgroup of U which is again dense in $U$. Therefore, we can assume without loss of generality that  $\sigma_1, \dots , \sigma_l \in U 
\cap  \rmF_r$. 
\end{remark}

\begin{notation} Let $A$ denote a $k$-affinoid algebra.
Let $\sigma = \prod_i e_i^{n_{j_i}} \in \fr$ be a general element of the profinite group $\fr$. Suppose furthermore we are given a continuous group morphism 
	\[
		\rho \colon \fr \to \GLn(A).
	\]
Let $M_1 \coloneqq \rho(e_1), \dots , M_r \coloneqq \rho(e_r)$.
We denote by 
	\[
		\sigma \big( M_1, \dots M_r \big) \coloneqq \prod_i M_i^{n_{j_i}} \in \GLn(A), 
	\]
whenever the right hand side is well defined, (which is always the case when the displayed product is finite).
\end{notation}

\begin{definition}
Let $U \in \mathcal{J}_r$ and fix $\{\sigma_i\}_{i=1}^l \in U \cap \fr$ a finite number of topological generators for $U$. We define the functor
	\[
		\LocSysfr(U, \{\sigma_i\}_{i=1}^l) \colon \cC \Afd_k \to \mathrm{Set},
	\] 
given on objects by the formula,
	\begin{align*}
		A & \mapsto \LocSysfr(U , \{\sigma_i\}_{i=1}^l)( A) \\
		&\coloneqq \{ (M_1, \dots, M_r) \in \mathrm{GL}_n (A)^r \colon \text{for each } i \in [1,l], \ \vert \sigma_i(M_1, \dots, M_r)- \mathrm{Id} \vert \leq \vert p \vert \}. 
	\end{align*}
When $A_0  $ is a formal model for $A$, we denote by $\LocSysfr(U, \{\sigma_i\}_{i=1}^l ) (A_0) $ the set of those $(M_1, \dots, M_r ) \in \GLn(A_0)^r$ such that
	\[
		\sigma( M_1, \dots, M_r) = \mathrm{Id}, \quad \textrm{mod } p.
	\]
\end{definition}

\begin{remark}
	By the fact that each $\sigma_i \in U \cap \fr$ it follows that these can be expressed as a finite product involving (powers) of the $e_j \in \fr.$ For this reason, given any $M_1, \dots , M_r \in \GLn(A)$, the elements
		\[
			\sigma_i(M_1, \dots, M_r) \in \GLn(A),
		\]
	are well defined.
\end{remark}

\begin{remark} \label{sdf}
Let $U \in \mathcal{J}_r$, $A \in \Afd_k^\op$ and $A_0$ be a formal model for $A$. Then the set 
	\[
		\LocSysfr ( U, \{\sigma_i\}_{i=1}^l)( A_0) \in \mathrm{Set}
	\]
does not depend on the choice of the topological generators for $U$. Indeed, if $\tau_1, \dots, \tau_m$ are different generators for $U$ we can express each $\sigma_i$ as a non-linear combination of the $\tau_1, \dots, \tau_m$, and vice-versa. Since
$\mathrm{Id} + p \cdot \rmM_n(A_0)$ is an (open) subgroup of $\GLn(A_0)$
we deduce that the elements $\{ \sigma_i(M_1, \dots, M_r) \}_{i =1}^l$ belong to $\mathrm{Id} + p \cdot \rmM_n(A_0)$ if and only if the elements in the system $\{ \tau_j(M_1, \dots, M_r) \}_{j =1 }^m$, do.
\end{remark}

\begin{notation}
Following \cref{sdf}, we will denote the set $\LocSysfr(U, \{\sigma_i\}_{i=1}^l )(A_0) $ simply by 
	\[
		\LocSysfr(U)(A_0) \in \mathrm{Set}.
	\]
\end{notation}

\begin{lemma}
Let $A \in  \cC \Afd_k$ be a $k$-affinoid algebra and $A_0$ an $\Ok$-formal model for $A$. Then, for each $r \ge 1$, there is a bijection
	\[
		\Hom_{\mathrm{cont}}\big(\fr, \GLn(A_0) \big) \cong
		\colim_{U \in \mathcal{J}_r} \LocSysfr \big( U \big)(A_0) ,
	\]
of sets.
\end{lemma}

\begin{proof}
The hom-space
	\[
		\Hom_{\mathrm{cont}}\big( \fr, \GLn(A_0))\big) \in \mathrm{Set},
	\]
denotes the set of continuous group homomorphisms in the category of pro-discrete groups. Consequently, we have a bijection of the form
	\[
		\Hom_{\mathrm{cont}}\big( \fr, \GLn(A_0)\big) \cong \underset{k}{ \lim} \colim_{U \in \mathscr{J}_r} \Hom_{\mathrm{Grp}}\big( \Gamma_U, \GLn(A_0/ p^{k+1} A_0)\big)  ,
	\]
where $\Gamma_U \coloneqq \fr / U$.
It therefore suffices to show that there exists a bijection of the form
	\[ 
		\underset{k}{\lim} \colim_{U \in \mathcal{J}_r} \Hom_{\mathrm{Grp}}\big( \Gamma_U, \GLn(A_0/ p^{k+1} A_0)\big) \cong \colim_{U \in \mathcal{J}_r} \LocSysfr \big( U\big) (A_0).
	\]
We first assert that there exists a canonical morphism
	\[
		\phi: \underset{k}{\lim} \colim_{U \in 
		\mathcal{J}_r} \Hom_{\mathrm{Grp}}\big( \Gamma_U, \GLn(A_0/ p^{k+1} 
		A_0)\big) \to \colim_{U \in \mathcal{J}_r} \LocSysfr \big( U \big)(A_0).
	\]
In order to prove this assertion we start by observing that a group morphism
	\[
		\rho_k: \Gamma_{U_k} \to \GLn(A_0 / p^{k+1}A_0),
	\]
with $U_k \in \mathcal{J}_r$, is completely determined by the image of the $r$ generators of $\Gamma_{U_k}$. Furthermore, these correspond to $r$ matrices in $\GLn(A_0/ 
p^{k+1} 
A_0)$. Therefore, given such a system of compatible group homomorphisms $\{ \rho_k \}_k$, one can associate an $r$-vector $(M_1, \dots, M_r) \in \GLn(A_0)^r$, whose mod $p$ reduction satisfies 
	\[
		\sigma_i( M_1, \dots , M_r) = \text{Id}.
	\]
By $\sigma_1, 
		\dots , \sigma_l \in U_1 \cap F_r$
we denote a choice of a finite set of topological generators for $U_1$. Thus 
	\[
		(M_1, \dots , M_r) \in \LocSysfr(U_1)(A_0).
	\]
This shows the existence of the desired map. For each $U \in \mathcal{J}_r$, we construct maps 
	\[
		\psi_{U}: \LocSysfr \big( U \big) (A_0) \to \underset{k}{ \lim} \colim_{U' \in 
		\mathcal{J}_r} \Hom_{\mathrm{Grp}}\big( \Gamma, \GLn(A_0/ p^{k+1} 
		A_0) \big),
	\]
such that when we assemble these together we obtain the desired inverse for $\phi$. In order to construct $\psi_{U}$, we start by fixing topological generators 
	\[
		\{\sigma_i\}_{i=1}^l \in U \cap F_r,
	\]
for $U$. Let $(M_1, 
\dots, M_r) \in \LocSysfr \big(U \big)(A_0)$. As we have seen, these matrices define a continuous group homomorphism 
	\[
		\fr \to \GLn(A_0/ p A_0).
	\]
Thanks to \cref{rmk1}, below, the matrices 
	\[
		\sigma_1(M_1, \dots , M_r), \dots, \sigma_l (M_1, \dots, M_r)  \in  \mathrm{Id} + p \cdot \rmM_n(A_0) 
	\]
determine a continuous group homomorphism 
	\[
		\rho_1: U \to \mathrm{Id} + p \cdot \rmM_n(A_0).
	\]
Then the inverse image 
	\[
		U_2' := 
		\rho_1^{-1} \big( \mathrm{Id} + p^2 \cdot \rmM_n(A_0) \big)
	\]
is an open normal subgroup of $U$, of finite index. As $U$ itself is an open subgroup of $\fr$ of finite index, we conclude that $U_2'$ is also a finite index subgroup of $\fr$. As open normal subgroups of finite index in $\fr$ define a final family for $\fr$, we further deduce
that there exists $U_2 \in \mathcal{J}_r$ such that $\rho_1(U_2) $ is a subgroup of $\mathrm{Id} + p^2 \cdot \rmM_n(A_0)$. Consequently, the matrices $(M_1,\dots,M_r) \in \LocSysfr(U)(A_0)$ define a group homomorphism 
	\[
		\rho_2 \colon \fr / U_2 \to \GLn( A_0 / p^2 A_0).
	\]
By 
iterating the process, we obtain a sequence of continuous group homomorphisms
	\[
		\{ \rho_i: \fr / U_i \to \GLn(A_0 / p^i A_0) \}_i \in \underset{i}{ \lim} \colim_{U \in \mathcal{J}_r} \Hom_{\mathrm{grp}}( \Gamma_U, \GLn(A_0 / p^i A_0)).
	\]
Assembling these together, we obtain a continuous group homomorphism $\rho \in \Hom_{\cont}(\fr, \GLn(A_0)$. It follows easily by our construction that, 
	\[
		\colim_{U \in \mathcal{J}_r}( \psi_U) \colon \colim_{U \in \mathcal{J}_r} \LocSysfr \big(U \big) (A_0) \to \Hom_{\cont}\big( \fr, \GLn(A_0)\big), 
	\]
is the inverse map of $\phi$, as desired.
\end{proof}

\begin{lemma}[Burnside problem for topologically nilpotent $p$-groups] \label{rmk1} Let $A_0$ be a $p$-adic complete $\Ok$-algebra. For each $m \geq 1$, we have a natural bijection
	\[
		\Hom_{\mathrm{cont}}\big(\fr, \mathrm{Id}+ p^{m+1} \cdot \rmM_n(A_0) \big) \cong \big( \mathrm{Id} + p^{m+1} \cdot \rmM_n(A_0) \big)^r.
	\]
\end{lemma}

\begin{proof} Let $s > 0$ be an integer.
Notice that the quotient groups,
	\[ 
		\big( \mathrm{Id} + p^m \cdot \rmM_n( A_0)/ ( \mathrm{Id} + p^{m + s +1} \cdot \rmM_n( A_0) \big), 
	\]
are torsion, i.e., every element has finite order. We conclude that
	\[ 
		\Hom_{\mathrm{cont}}  \big( \widehat{\bZ}, \mathrm{Id} + p^{m+1} \cdot \rmM_n( A_0) \big) \cong \mathrm{Id} + p^{m+1} \cdot \rmM_n( A_0),
	\]
where $\widehat{\bZ}$ denotes the profinite completion of $\mathbb{Z}$. This finishes the proof when $r =1$.
The same holds for general $\fr$, i.e., we have a canonical equivalence, 
	\[
		\Hom_{\mathrm{cont}}\big(\fr, \mathrm{Id}+ p^{m+1} \cdot \rmM_n(A_0) \big) \cong  \big( \mathrm{Id} + p^{m+1} \cdot \rmM_n(A_0) \big)^m.
	\]
In order to prove this last assertion it suffices to show that
any finitely generated subgroup of the quotient 
	\[
		\big( \mathrm{Id} + p^m \cdot \rmM_n( A_0)/ ( \mathrm{Id} + p^{m + s +1} \cdot \rmM_n( A_0) \big),
	\]
is finite (i.e. the 
Burnside problem admits an affirmative answer in this particular case). In order to justify the given assertion we fix $G$ a finitely generated subgroup of 
	\[
		 \big( \mathrm{Id} + p^m \cdot \rmM_n( A_0)/ ( \mathrm{Id} + p^{m + l+1} \cdot  \mathrm{M}_n( A_0) 
		\big).
	\] 
By assumption it is generated by matrices of the form $\mathrm{Id} + p^{m+1}N_1,  \dots , \mathrm{Id} + p^{m+1}N_s$. Therefore a general element of $G$ can be written as, 
	\[
		\mathrm{Id} + p^{m+1}( n_{1,1} N_{1, 1} + \dots + n_{1, s }N_{s, 1} ) + \dots + p^{m+l-1}(n_{a-1, 1} N_{a-1, 1} + \dots + n_{a-1, s^{l-1}} N_{a-1,s^{l-1}}), 
	\]
where the $N_{i, j}$, for $i,j \in [1, a-1] \times [1, s^{a-1}]$, denote non-linear products of the $N_i$ having at most $a-1$ terms, counted with multiplicity, where $a$ denotes the least integer such that $m \times (a+1) \geq l$.
By the Pigeonhole principle there are only finite number of such choices for the integers $n_{i,j}$ for $(i,j) \in [1, l-1]\times [1, s^{l-1}]$ and the result follows.
\end{proof}

\begin{proposition} \label{union_subguys}
Let $A\in  \cC \mathrm{Afd}^{\op}_k$ denote a $k$-affinoid algebra then we have a natural bijection,
	\[
		\Hom_{\mathrm{cont}}\big(\fr, \GLn(A)\big) \cong  \colim_{(U ,  \sigma_1, \dots , \sigma_l)} \LocSysfr \big( U,\{ \sigma_i \}_{i=1}^l \big)(A),
	\]
where $U \in \mathcal{J}_r$ and $\{\sigma_i\}_{i=1}^l$ denote a choice of topological generators for $U$.
\end{proposition}

\begin{proof}
Let $\rho: \fr \to \GLn( A)$ be a continuous homomorphism of topological groups and let $e_1, \dots, e_r$ be the fixed topological generators of $\fr$. Let
	\[
		M_i \coloneqq \rho( e_i) \in \GLn(A)
	\]
for each $1 \leq i \leq r$. The group $\mathrm{Id} + p \cdot \mathrm{M}_n( A_0)$ is 
open in $\GLn( A_0)$ and the latter open in $\GLn(A)$. We thus deduce that the inverse image 
	\[
		U \coloneqq \rho^{-1} \big( \mathrm{Id} + p \cdot \mathrm{M}_n( A_0) \big)
	\]
is an open subgroup of $\fr $ and it has thus finite index in $\fr$. Moreover, as $\mathcal{J}_r$ is a final 
family for $\fr$ one can suppose without loss of generality, up to shrinking $U$, that $U \in \mathcal{J}_r$. Thus we might as well assume that $U$ is normal in $\fr$.
Choose a finite set of topological generators for $ \rho^{-1} \big( \mathrm{Id} + p \cdot \mathrm{M}_n( A_0) 
\big)$. We deduce that the $(M_1, \dots, M_r)$ satisfy the family of inequalities
	\[
		\{ \sigma_i(M_1, \dots, M_r) \le \vert p \vert \}_{i = 1}^l.
	\]
Therefore 
	\[
		(M_1, \dots, M_r ) \in \colim_{(U, \{\sigma_i\}_{i=1}^l)} \LocSysfr \big( U , \{ \sigma_i \}_{i=1}^l\big)(A),
	\]
with $U \in \mathscr{J}_r$ and $\sigma_1, \dots , \sigma_l$ topological generators for $U$.
This proves the direct inclusion. We also deduce that we have a well defined map of sets
	\[ 
		\rho \in \Hom_{\mathrm{cont}}(\fr, \GLn(A)) \mapsto ( \rho(e_1), \dots , \rho(e_r)) \in \colim_{(U  \sigma_1, \dots , \sigma_l)} \LocSysfr \big( U, \{ \sigma_i \}_{i=1}^l \big) (A).
	\]
Let us construct an inverse map. Consider 
	\[
		(M_1, \dots, M_r) \in \GLn( A)^r,
	\]
such that 
	\[
		\sigma_i (M_1, \dots, M_r) \in \mathrm{Id} + p \cdot M_n( A_0).
	\]
Here we assume that the $		\{\sigma_i \}_{i \in [1,l]} ,$
all lie in
the dense subgroup $F_r \subset \fr$. Let $U$ denote the finite index normal open subgroup of $\fr$, generated by the $\{ \sigma_i \}_{i \in [1, l]}$. 

We remark that U is free profinite by the version of Nielsen-Schreier theorem for open subgroups of free profinite groups, see 
\cite[Theorem 3.3.1]{ribes2008wreath}. By \cref{rmk1}, we deduce that 
	\[
		( \sigma_1( M_1, \dots , M_r) , \dots, \sigma_l (M_1, \dots , M_r) ) \in \GLn(A)^r
	\]
defines a continuous group homomorphism 
	\[
		\bar{\rho}: U \simeq  \widehat{\rmF}
		_l \to \mathrm{Id} + p \cdot \mathrm{M}_n( A_0).
	\]
Therefore, we have the following diagram in the category of topological groups, 
	\[
	\begin{tikzcd}
		F_r \arrow{d}{(M_1, \dots, M_r)} \arrow{r} & \fr & U \arrow{l} \arrow{d} \\
		\GLn(A) \arrow{r}{=} & \GLn(A) & \text{Id} + p \rmM_n(A_0) \arrow{l}.
	\end{tikzcd}
	\]
We want to show that we can fill the above diagram with a continuous morphisms $\fr \to \GLn(A)$ making the whole diagram commutative.
Since $U$ is of finite index in $\fr$, we can choose elements 
	\[
		g_1, \dots g_m \in F_r \subset \fr
	\]
such that these form a (faithful) system of representatives for the finite group $\fr / U$.
For $i \in [1, m]$, write 
	\[
		g_i \coloneqq \prod_{j_i} e_{j_i}^{n_{j_i}},
	\] where this product is finite and unique by the assumption that the $g_i \in F_r$. Every element of $h \in \fr$ can be written as $h = g_i \sigma$, for some $g_i$ as above and $\sigma \in U$. Let us then 
define 
	\[
		\rho( h) \coloneqq \big( \prod_{j_i} M_{j_i}^{n_{j_i}} \big) \bar{\rho}(\sigma) \in \GLn(A).
	\]
We are left to verify that the association 
	\[
		h \in \fr \mapsto \bar{\rho}(h) \in \GLn(A),
	\]
gives a well defined continuous group homomorphism. Let 
	\[
		g \coloneqq \prod_s e_s^{n_s} \in F_r 
		\subset 
		\fr,
	\]
and $\sigma' \in U$ such that $g \sigma' = h = g_i \sigma$. We first prove that 
	\[
		\bar{\rho}(h) = \big( \prod_s M_s^{n_s} \big) \bar{\rho}(\sigma').
	\]
Suppose that $\sigma, \sigma' \in U \cap F_r$, then it follows that $h \in F_r$. The result now follows, in this case, since we have 
fixed a group homomorphism
	\[
		(M_1 , \dots, M_r) \colon F_r \to \GLn(A),
	\]
which is necessarily continuous. Otherwise, assume that it is not the case that
	\[
		\sigma , \sigma' \in U \cap F_r.
	\]
Let $(\sigma_n)_n $ and $(\sigma_{n'})_{n'}$ be sequences of elements in $U \cap F_r$ converging to $\sigma$ and $\sigma'$, 
respectively. We observe that this is possible since $F_r $ is dense in $\fr$ and $U \cap F_r$ is a free (discrete) group whose profinite completion is canonically equivalent to $U$, thus dense in $U$.
For this reason, we obtain that 
	\[
		g^{-1} g_i \sigma = \sigma' ,
	\]
and we get moreover 
that $g^{-1} g_i \sigma_n$ converges to $\sigma'$. Write $g = \prod_i e_i^{n_i}$ and, for each $i$, $g_i = \prod_{j_i} e_{j_i}^{n_{j_i}}$. The elements
	\[
		(\prod_j M_{i}^{-n_{i^{-1}}} ) (\prod_{j_i} M_{j_i}^{n_{j_i}}) \rho(\sigma_m) \in \GLn(A) \ , 
	\]
converge to $\rho(\sigma')$ by continuity of $\rho$. They also converge to the element 
	\[
		(\prod_j M_{i}^{-n_{i^{-1}}} ) (\prod_{j_i} M_{j_i}^{n_{j_i}})  \rho(\sigma) \in \GLn(A).
	\]
This last assertion follows by continuity of the group multiplication on $\GLn(A)$. Since the topology on $A$ comes from a norm on $A$, 
making the latter a Banach $k$-algebra, we conclude that $A$ is Hausdorff and so it is $\GLn(A)$. This implies that converging sequences in $\GLn(A)$ admit a unique limit. We conclude therefore that,
	\[ 
		\rho(\sigma') = (\prod_j M_{i}^{-n_{i^{-1}}} ) (\prod_{j_i} M_{j_i}^{n_{j_i}}) \rho( \sigma).
	\]
We obtain then the desired equality,
	\[
		(\prod_j M_{i}^{-n_{i^{-1}}} ) \rho(\sigma')  = (\prod_{j_i} M_{j_i}^{n_{j_i}}) \rho(\sigma).
	\]		
This proves that $\overline{\rho} \colon \fr \to \GLn(A)$ is a well defined map. We wish to show that it is a continuous group homomorphism. Our definitions make clear that to prove multiplicativity of $\overline{\rho}$ it suffices to show that for every $g \in F_r$ and $\sigma \in 
U$ 
we 
have,
	\[
		\overline{\rho}( g \sigma g^{-1} ) = \overline{\rho}(g) \rho(\sigma ) \overline{\rho}(g^{-1}).
	\]
Choose again a converging sequence $(\sigma_n)_n$, in $F_r \cap U$, such that $\sigma_n$ converges to $\sigma$. Then, for each $n$, we have
	\[
	\overline{\rho}(g \sigma_n g^{-1} ) = \overline{\rho}( g) \rho( \sigma_n) \bar{\rho}( g^{-1}).
	\]
Passing to the limit implies that we obtain the desired equality. We are reduced to show that $\overline{\rho}$ is continuous. Let $V $ be an open subset of $\GLn(A)$. The intersection $V \cap( \mathrm{Id} + pM_n(A_0)$ is open in $\GLn(A)$.
Thus, 
	\[
		\overline{\rho}^{-1} ( V 
		\cap( \mathrm{Id} + p \cdot \rmM_n(A_0)) = \rho^{-1}(V \cap( \mathrm{Id} + p \cdot \rmM_n(A_0)),
	\]
is open in $U$. Therefore, the quotient $U/ (V \cap( \mathrm{Id} + p \cdot \rmM_n(A_0))$ is discrete, since $U$ is of finite index in $\fr$. We conclude that, 
	\[
		 U / \rho^{-1}(V \cap( \mathrm{Id} + p \cdot \rmM_n(A_0)) \to \fr / \overline{\rho}^{-1} (V),
	\]
exhibits the quotient
	\[
		 U / \rho^{-1}(V \cap( \mathrm{Id} + p \cdot \rmM_n(A_0)),
	\]
as a subgroup of finite index in $ \fr / \overline{\rho}^{-1} (V)$. Thus the latter is necessarily discrete. The result now follows, since we have that $\overline{\rho}^{-1}(V)$ is an open subset in $\fr$.
\end{proof}

\begin{notation}
We shall denote by
	\begin{align*}
		\LocSysfr(\fr) (A) & \coloneqq  \colim_{(U , \{ \sigma_i \}_{i=1}^l )} \LocSysfr \big( U,\{ \sigma_i \}_{i=1}^l \big)(A) \\
		& \cong \Hom_{\mathrm{cont}}\big(\fr, \GLn(A)\big).
	\end{align*}
\end{notation}

\begin{remark} Recall the analytification functor, introduced in \cite[\S 2.6]{berkovich1993etale}. Let $\anGLn$ denote the analytification of the general linear group scheme $\mathrm{GL}_n$ over $\mathrm{Spec} (k)$. 
Proposition \ref{union_subguys} allows us to write $\LocSysfr(\fr)$ as a union of subfunctors of $\anGLn$, namely $\{ \LocSysfr(U, \{\sigma_i\}_{i=1}^l) \}_{U \in \fr}$.
This is similar to the situation in the complex case.
\end{remark}

\begin{lemma} \label{subguys_rep}
The functor $\LocSysfr(U, \{ \sigma_i \}_{i=1}^l ) \colon \Afd_k^\mathrm{op} \to \mathrm{Set}$ is representable by a (strict) $k$-analytic space. 
\end{lemma}

\begin{proof}
 Let $\GLn^0 = \mathrm{Sp} \big( k \langle T_{ij} \rangle [ \frac{1}{\mathrm{det}} ] \big)$ denote the closed unit disk in $\anGLn$. Let
	\[
		\LocSysfr \big( U, \{\sigma_i\}_{i=1}^l \big)^0 \in \mathrm{Fun}(\An_k, \mathrm{Set})
	\]
denote the pullback of $\LocSysfr \big( U, \{\sigma_i\}_{i=1}^l \big)$,
along the inclusion morphism $
		 \GLn^0 \hookrightarrow 
		 \anGLn$.
Consider the following cartesian diagram
	\[
	\begin{tikzcd}
		\LocSysfr (U, \{ \sigma_i \}_{i=1}^l, B)^0 \ar{d} \ar{r} & \mathrm{Sp}_\rmB B \ar{d} \\
		\LocSysfr \big( U, \{\sigma_i\}_{i=1}^l \big)^0 \ar{r} & (\GLn^0)^r.
	\end{tikzcd}
	\]
Where $B \in \Afd_k^\op$ is a $k$-affinoid algebra and 
	\[
		(M_1, \dots , M_r) \in \GLn^0(\rmB)^r = \GLn( \rmB^0)^r,
	\]
corresponds to a given morphism of $k$-analytic spaces $\mathrm{Sp} (B) \to (\GLn^0)^r$. It follows that 
	\[
		\LocSysfr (U, \{ \sigma_i \}_{i=1}^l, B)^0 \in \Fun \big( \Afd_k^\op, \mathrm{Set } \big),
	\]
corresponds to the 
subfunctor of $
\mathrm{Sp} (B)$, whose value on any $k$-affinoid algebra $A$ is in bijection with the set
	\[ 
		\mathscr{X}_{U, \{\sigma_i\}_{i=1}^l, B}^0 ( A)  \coloneqq \{ f: B \to A: \text{for each } i, \vert \sigma_i(M_1, \dots, M_r) - \mathrm{Id} \vert \leq \vert p \vert , \ \textrm{in } A\}.
	\]
Therefore, at the level of points, the functor
	\[
		\LocSysfr \big(U, \{\sigma_i\}_{i=1}^l, B \big)^0 
	\]
parametrizes those points $ x \in \mathrm{Sp} (B)$ such that, for each $i$, 
	\[
		\vert \sigma(M_1, \dots, M_r) - \mathrm{Id} \vert (x) \leq \vert p \vert(x),
	\]
It is clear from our description, that this latter functor is
representable by a 
Weierstrass subdomain of $\mathrm{Sp} (B)$. As $(\GLn^0)^r$ is a (strict) $k$-affinoid space, it follows that  $\LocSysfr \big( U, \{\sigma_i\}_{i=1}^l \big)$ is representable in the category $\Afd_k$, (consider in the above diagram with $\mathrm{Sp}
(B) = \GLn^0 $ and $(M_1, 
\dots, M_r) $ the $r$-vector whose matrix components correspond to identity morphism of $\GLn^0$). 

Let $c_i \in |k^{\times}|$ be a decreasing sequence of real numbers converging to $0$, there is a natural isomorphism
	\[
		(\GLn^0)^r \simeq \colim_i (\GLn^0)^r_{c_i},
	\]
where $(\GLn^0)^r_{c_i} $ denotes a copy of $(\GLn^0)^r$ indexed by $c_i$. The inclusion 
morphisms in the corresponding diagram sends $(\GLn^0)^r_{c_i }$ to the closed disk of radius $c_i^{-1}$ inside of $(\GLn^0)^r_{c_{i+1}}$. Henceforth we have canonical isomorphisms 
	\[
		\LocSysfr \big( U, \{\sigma_i\}_{i=1}^l \big) \cong \colim_i \LocSysfr (U, \{ \sigma_i \}_{i=1}^l)^0 \big)_{c_i}.
	\]
The latter is a union of $k$-affinoid subdomains where the image of an element in the filtered diagram lies in the relative interior, of the successive one. We thus conclude that $\LocSysfr \big(U, \{\sigma_i\}_{i=1}^l \big)
$ is itself representable by an $k$-analytic space.
\end{proof}

The following formal statement will be helpful in the proof of our main theorem:

\begin{proposition} \label{prop:representability_of_filtered_union_of_k-analytic_spaces}
Let $\{ \cX_i \}_{i \in I}$ denote a filtered diagram of $k$-analytic subdomains of a $k$-analytic space $\cY$. If for every $i \in I$, there exists an index $j \in I$ together with a morphism $i \to j$ in $J$ such that the image of the transition morphism
	\[
		f_{i , j } \colon \cX_i \to \cX_j,
	\]
lies in the relative interior $\mathrm{Int}(\cX_j/ \cY)$, then the colimit $\cX \coloneqq \colim_{i \in I} \cX_i$ exists in the category $\An_k$ and the induced morphism
	\[
		\cX \to \cY,
	\]
exhibits $\cX$ as a $k$-analytic subdomain of $\cY$.
\end{proposition}

\begin{proof} By our assumptions,
for each $i \in I$, $\cX_i$ is a $k$-analytic subdomain of $\cY$. For this reason, the $k$-analytic subdomain $\cX_i$ represents a functor, also denoted
	\[
		\cX_i \colon \Afd_k^\op \to \mathrm{Set},
	\]
which is a sub-functor of the one represented by $\cY$. This last assertion follows readily 
from the definitions. We will actually need to prove a slightly more general assertion than the stated one, namely: let $\cX \coloneqq \colim_{i \in I} \cX_i$ as functors $\Afd_k^\op \to \mathrm{Set}$. It follows from the construction that we
have a natural morphism
	\[
		f \colon \cX \to \cY,
	\]
which exhibits the former as a sub-functor of the later. We will show that $\cX$ is representable by a $k$-analytic space and $f$ exhibits the later as a $k$-analytic subdomain of $\cY$. Moreover,
$\cX$ coincides with the filtered colimit of the $(\cX_i, i \in I)$, computed in $\An_k$. Thanks to \cite[Exercise 4.5.3]{conrad2008several}, it follows that
the relative interior
	\[
		\mathrm{Int} \big( \cX_i / \cY \big) \hookrightarrow \cY
	\]
is an open subset of 
$\cY$. As a consequence, to $\cX$ we canonically associate a topological subspace 
	\begin{align*}
		\cX^{\mathrm{top}}  & \coloneqq \colim_{i \in I} \cX_i^\mathrm{top} \\
		& \cong \colim_{i \in I } \mathrm{Int} \big(\cX_i / \cY \big)^{\mathrm{top}}, 
	\end{align*}
where the superscript $(\textrm{-})^\mathrm{top}$ denotes the underlying topological space of a $k$-analytic space. By our previous considerations, $\cX^{\mathrm{top}}$ is an open subset of $\cY^\mathrm{top}$.
Consequently, the former is necessarily an 
Hausdorff space. We will construct a canonical $k$-analytic 
structure on it and show that such $k$-analytic space represents 
the functor $\cX$, above. Since, each
	\[
		\cX_i \in \An_k, \quad i \in I
	\] 
is a $k$-analytic space, we can take the maximal atlas and quasi-net on it consisting of $k$-affinoid 
subdomains of $\cX_i$, which we denote by $\mathcal{T}_{i}$. As $\cX^\mathrm{top}$ can be realized as a filtered 
union of the $\cX_i$ we claim that the union of the quasi-nets $\mathcal{T}_i$ induces a quasi-net $
\mathcal{T}$ on $\cX^\mathrm{top}$. In order to prove this we shall show that given a point $x \in \cX^\mathrm{top}
$ we need to be able to find a finite collection $V_1, \dots, V_n$ 
of compact Hausdorff subsets of $\cX^\mathrm{top}$ such that
	\[x \in \bigcap_i V_i,\]
and furthermore $V_1 \cup \dots \cup V_n$ is an open 
neighborhood of $x$ inside $\cX^\mathrm{top}$. In order to show such 
condition on $\mathcal{T}$, we notice first that that we can choose a sufficiently large index $
i \in I$ such 
that 
	\[
		x \in \cX_i,
	\]
lies in the relative interior $\mathrm{Int} \big(
\cX_i /\cY \big)$. Indeed, there exists some sufficiently large $i \in I $ such that $x \in \cX_i^\mathrm{top}$. By our assumptions, there exists $i \to j$ in $I$ such that
	\[
		\cX_i^{\mathrm{top}} \to \cX_j^\mathrm{top},
	\]
has image in the relative interior $\mathrm{Int}(\cX_j / \cY)$. Thus $x \in \mathrm{Int}(\cX_j / \cY)$, as desired. By the $k$-analytic structure on 
	\[
		\cX_j \in \An_k,
	\]
we conclude that we can take $V_1, \dots, V_n$ $k$-
affinoid subdomains of $\cX_j$ satisfying the above 
condition, as desired. 

We are reduced to show that the union $V_1 \cup \dots  
\cup V_n$ is open in $\cX^\mathrm{top}$. As before, we can assume that the $V_i$ are subdomains of a given $\cX_i,$ for a sufficiently large $i \in I$. Furthermore, by applying the same reasoning as before, we can further suppose that
the $ V_i $ lie in the relative interior $\mathrm{Int}(\cX_i, \cY)$. For this reason, the union $V_1 \cup 
\dots \cup V_n$ lies itself in the relative interior $\mathrm{Int} \big(
\cX_i/ \cY \big)$. By our choice of the $V_i$'s, their union is open in $\cX_i^{\mathrm{top}}$ and thus open in $\mathrm{Int}(\cX_i/ \cY)$, as it is contained in the latter. Since 
	\[
		\mathrm{Int}( 
		\cX_i/ \cY \big).
	\]
is open in $\cY$, we conclude that the union $V_1\cup \dots \cup V_n$ is open in $\cY$ and consequently also open in the subspace $\cX^\mathrm{top}$, itself. Clearly, $\cT$ induces quasi-nets on the intersections 
	\[
		W \cap W',
	\]
for any $W, \ W' \in \cT$. Indeed, this is an immediate consequence of the fact that
we can always choose a sufficiently large index $ i \in I$ such that $W, \ W' 
\subset \cX_i$. 

We have just proved that the union of the maximal atlas 
on each $\cX_i$ induces a well defined atlas on $\cX^\mathrm{top}$, with respect to $\cT
$. We conclude that the 
topological space $\cX^\mathrm{top}$ is endowed with a natural structure 
of $k$-analytic space (in fact, a $k$-analytic subdomain of $\cY
$, since our choices are all compatible with the inclusion $\cX^\mathrm{top} \to \cY^\mathrm{top}$). Let us denote by $\cX'$ the $k$-analytic space $(\cX^\mathrm{top}, \cT)$. We shall show that $\cX'$ 
represents the functor
	\[
		\cX \colon \Afd_k^\op \to \mathrm{Set}.
	\]
As $k$-affinoid spaces are 
quasi-compact we conclude that any map 
	\[
		\mathrm{Sp}( A ) \to \cX'
	\]
factors through some $\cX_i$, as their 
union equals the union of the respective relative interiors, which are open $k$-analytic subspaces of $\cX'$. It thus follows that $\cX'$ represents the functor given by filtered colimit
	\[
		\cX \simeq \colim_{i \in I} \cX_i,
	\]
as desired.
\end{proof}

\begin{theorem} \label{truncation_representability_free}
For each $r \geq 1$, the functor 
	\[
		\LocSysfr (\fr) : \Afd_k^{\op} \to \mathrm{S}\mathrm{et},
	\]
given on objects by the formula
	\[
		A \in \Afd_k \mapsto \{ (M_1, \dots , M_r) \in \GLn (A): \textrm{there exists }i , \ |\sigma_i(M_1, \dots M_r) - \mathrm{Id}| \leq |p| \} \in \mathrm{Set},
	\]
is representable by a (strict) $k$-analytic space.
\end{theorem}

\begin{proof} We start by observe that if $U' \subset U$ is an inclusion of groups, both lying in $\mathcal{J}_{r},$ then they induce an inclusion of functors
	\[
		\LocSysfr \big( U, \{\sigma_i\}_{i=1}^l \big) \hookrightarrow \LocSysfr \big( U',  \{ \tau_i \}_{i=1}^s \big).
	\]
Let us employ the same notations as we did previously for
	\[
		\sigma_i(M_1, \dots , M_r) \in \mathrm{Id} + p \cdot \mathrm{M}_n(A_0),
	\]
where the $\sigma_i$ denote a choice of generators for $U$, lying in the dense subgroup $U \cap F_r$. It follows that,
we have
	\[
		\tau'_j(M_1, \dots , M_r) \in \mathrm{Id} + p \cdot \mathrm{M}_n( A_0),
	\]
for a choice of generators for $U'$, lying in $U' \cap F_r$. Thanks to \cref{subguys_rep}, it follows that the functor
	\[
		\LocSysfr \big( U, \{\sigma_i\}_{i=1}^l \big) \colon \Afd_k^\op \to \mathrm{Set},
	\]
is representable by a $k$-analytic 
subdomain of $
\LocSysfr \big( U' , \tau_1, \dots, \tau_s \big)$. We wish to show that
	\[
		\LocSysfr \big( \fr \big)  \cong \colim_{(U \in \mathcal{J}_r, \{ \sigma_i \}_{i =1}^l)}  \LocSysfr \big( U, \{\sigma_i\}_{i=1}^l \big),
	\]
is representable by a $k$-analytic space. We will prove that the inclusions
	\[
		\LocSysfr \big( U, \{\sigma_i\}_{i=1}^l \big)  \hookrightarrow \LocSysfr \big( U' , \{ \tau_i \}_{i=1}^s \big)
	\]
are inclusions of $k$-analytic subdomains as in \cref{prop:representability_of_filtered_union_of_k-analytic_spaces}, whenever $U$ denotes a sufficiently large (finite) index subgroup of $\fr$.
Given a $k$-affinoid algebra $A$ and 
	\[
		(M_1, \dots , M_r) \in \LocSysfr \big( U, \{\sigma_i\}_{i=1}^l \big)(A),
	\]
we can write, for each $i$,
	\[
		\sigma_i(M_1, \dots, M_r) = \mathrm{Id} + p \cdot N_i,
	\]
for suitable matrices $N_i \in M_n( A_0)$. Moreover, the $r$-tuple $(M_1, \dots , M_r)$ defines a continuous group homomorphism 
	\[
		\rho \colon \fr \to \GLn(A).
	\]
By \cref{rmk1}, 
quotients of the pro-$p$-group $\mathrm{Id}+ p \cdot \rmM_n(A_0)$ are of $p$-torsion.
Let $U' \in \mathcal{J}_r$ be such that 
	\[
		U' \subset \rho^{-1} (\mathrm{Id} + p^2 \cdot \rmM_n(A_0))
		.
	\]
Given $\tau_1 , \dots, \tau_s$ generators for $U'$, as above, we have 
	\[
		\vert  \tau_i( M_1, \dots , M_r) -
		\mathrm{Id} \vert \leq  \vert p^2 \vert < \vert p \vert,
	\]
for each $i \in [1,s]$. This implies that 
	\[
		(M_1, \dots, M_r) \in \mathrm{Int} \big( \LocSysfr \big( U', \{ \tau_i \}_{i=1}^s \big) / (\anGLn)^r \big). 
	\]
We are thus precisely in the conditions of \cref{prop:representability_of_filtered_union_of_k-analytic_spaces}, and the result follows.
\end{proof}

\begin{corollary} \label{pr:Hom_G}
Let $G$ be a profinite group, topologically of finite type. Then the functor 
	\[
		\LocSysfr(G) : \Afd_k^{\mathrm{op}} \to \mathrm{Set},
	\]
given on objects by the formula
	\[
		A \in \Afd_k^{\op} \to \Hom_\cont \big(G, \GLn(A) \big) \in \mathrm{Set},
	\]
is representable by a $k$-analytic space.
\end{corollary}

\begin{proof}
Let us fix a continuous surjection of profinite groups 
	\[
		q \colon \fr \to G,
	\]
for some integer $r \leq 1$. Let $H$ denote the kernel of $q$. Thanks to \cref{truncation_representability_free}, we know that $\LocSysfr(\fr)$ is representable by a $k$-analytic 
stack. We have an inclusion at the level of 
functors of points 
	\[
		q_* \colon \LocSysfr(G) \to \LocSysfr(\fr),
	\]
induced by precomposing continuous homomorphisms $\rho \colon G \to \GLn(A)$ with $q$. We show that the morphism $q_*$ is representable and a closed immersion. Let $\Sp A$ be a $k$-affinoid space and 
suppose given a morphism of $k$-analytic spaces,
	\[
		\rho \colon \Sp (A) \to \LocSysfr(G),
	\]
corresponding to a continuous representation $\rho \colon G \to \GLn(A)$. We want to compute the fiber product 
	\[
		\Sp (A) \times_{\LocSysfr(\fr)} \LocSysfr(G).
	\]
Notice that $\Sp( A)$ is quasi-compact and that we have an isomorphism at the underlying topological spaces,
	\[
		\LocSysfr(\fr)  \cong \colim_{U , \sigma_1, \dots , \sigma_l  }
		\mathrm{Int} \big(\LocSysfr(U, \{\sigma_i\}_{i=1}^l ) / (\mathbf{GL}_n^\an)^r \big).
	\]
It follows that
	\[
		q_*(\rho) \colon \Sp A \to \LocSysfr(\fr),
	\]
factors through a $k$-analytic subspace of the form $\LocSysfr \big( U, \{ \sigma_i \}_{i =1}^l \big)$, for suitable $U \in \mathcal{J}_r$ and $\sigma_1, \dots , \sigma_l \in U$. By applying again
 the same reasoning we can
assume further that 
	\[
		\rho \colon \Sp (A)  \to \LocSysfr(\fr),
	\]
factors through some $\LocSysfr \big( U, \{ \sigma_i \}_{i =1}^l \big)^0$, as in the proof of \cref{subguys_rep}. The latter is $k$-affinoid, say 
	\[
		\mathscr{X}_{U,  \{ \sigma_i \} }^0 \cong \Sp (B),
	\]
for some $k$-affinoid algebra $B$.
Let $\mathscr{X}_{G, U, \sigma_1, \dots , \sigma_l}^0$ denote the fiber product,
	\[
	\begin{tikzcd}
		\LocSysfr \big( G ,U, \{\sigma_i\}_{i=1}^l\big)^0 \arrow{d} \arrow{r} & \LocSysfr(G) \arrow{d} \\
		\LocSysfr \big( U,  \{\sigma_i\}_{i=1}^l \big)^0 \arrow{r} & \LocSysfr(\fr)
	\end{tikzcd}.
	\]
By construction, the set 
	\[
		\LocSysfr( G ,U, \{ \sigma_i \}_{i=1}^l)^0(A) \in \mathrm{Set}
	\]
corresponds to those $(M_1, \dots , M_r ) \in \LocSysfr ( U,  \{ \sigma_i \}_{i=1}^l)^0(A)$ such that 
	\[
		 h( M_1, \dots, M_r) = \text{Id},
	\]
for every $h \in H \cap F_r \subset H$. We conclude that we have an equivalence of fiber products,
	\begin{align*}
		Z & :=\Sp A \times_{\LocSysfr(\fr)} \LocSysfr(G)  \\
		& \cong \Sp A \times_{\LocSysfr \big( U,  \{ \sigma_i \}_{i=1}^l \big)^0} \LocSysfr \big( G, U, \{ \sigma_i \}_{i=1}^l \big)^0.
	\end{align*}
Every $k$-affinoid algebra is Noetherian, cf. \cite[Theorem 1.1.5]{conrad2008several}.
For this reason, we conclude that $Z$ parametrizes points which are determined by finitely many equations with coefficients in $ A \in \cC \Afd_k$ (these are induced from the relations defining $H$ inside $\fr$). The latter assertion implies that $Z$ is a closed 
subspace of $\Sp A$ and thus representable. The result now follows.
\end{proof}

\begin{remark} 
Let $G$ be a profinite group as above. There exists a canonical action of the $k$-analytic group $\anGLn$ on $\LocSysfr(G)$, via conjugation. Conjugacy classes of elements in $\LocSysfr(
G)$, under the conjugation action of $\anGLn$, correspond to rank $n$ continuous representations of $G$. 
\end{remark}

\subsection{Geometric contexts and geometric stacks} \label{chapter_2.2}
Our next goal is to give an overview of the general framework that allow us to define the notion of a geometric stack in the $k$-analytic setting. Our motivation comes from the need to define the moduli stack of continuous representations of a profinite group $G$,
as a \emph{$k$-analytic stack}. 

\begin{definition} \label{geometric_context_def}
A geometric context $( \mathcal{C}, \tau, \textbf{P})$ consists of an $\infty$-site $( \mathcal{C}, \tau)$, cf. \cite[Definition 6.2.2.1]{lurie2009higher}, and a class $\textbf{P}$ of morphisms in $\mathcal{C}$ verifying:
\begin{enumerate}
\item Every representable sheaf is a hypercomplete sheaf on $( \mathcal{C}, \tau)$.
\item The class $\textbf{P}$ is closed under equivalences, compositions and pullbacks.
\item Every $\tau$-covering consists of morphisms in $\textbf{P}$.
\item Let $f: X \to Y $ be a morphism in $\mathcal{C}$. Suppose that we are given a $\tau$-covering $\{U_i \to X \},$ such that each composition $U_i \to Y $ belongs to $\textbf{P}$. Then $f$ belongs to $\textbf{P}$.
\end{enumerate}
\end{definition}

\begin{notation}
Let $(\cC, \tau )$ denote an $\infty$-site. We denote by $\Shv (\cC, \tau )$ the \infcat of sheaves on $(\cC, \tau )$. It can be realized as a presentable left localization of the \infcat of presheaves on $\cC$, $\mathrm{PSh} (\cC) 
$, cf. \cite[Lemma 6.2.2.7]{lurie2009higher}.
\end{notation}

Given a geometric context $( \mathcal{C}, \tau , \textbf{P})$ it is possible to form an $\infty$-category of geometric stacks $\mathrm{Geom}( \mathcal{C}, \tau, \textbf{P})$ via an inductive definition as follows:

\begin{definition} \label{geometric_stack_definition}
A morphism in $F \to G $ in $\mathrm{Shv}( \mathcal{C}, \tau)$ is $(-1)$-representable if for every map $ X \to G$, where $X$ is a representable object of $\mathrm{Shv}(\mathcal{C}, \tau)$, the base change $F \times_G X $ is also representable. 

Let $n \geq 0$, 
we say that $F \in \mathrm{Shv}(\mathcal{C}, \tau)$ is $n$-geometric if it satisfies the following two conditions:
\begin{enumerate}
\item It admits an $n$-atlas, i.e. a surjective morphism $p: U \to F$, where $U$ is representable, $p$ is $(n-1)$-representable and it lies in $\textbf{P}$.
\item The diagonal map $F \to F \times F$ is $(n-1)$-representable.
\end{enumerate}
\end{definition}

\begin{definition}
We say that $F \in \mathrm{Shv}( \mathcal{C}, \tau)$ is \emph{locally geometric} if $F$ can be written as an union of $n$-geometric stacks $F = \bigcup_i G_i$, for possible varying $n$. We further require that each morphism $G_i \to F$ is an open immersion. 
\end{definition}

The following is a desirable feature of a geometric context:

\begin{definition}
Let $( \mathcal{C}, \tau)$ be an $\infty$-site. The $\infty$-category $\mathcal{C}$ is \emph{closed under $\tau$-descent} if for any morphism $F \to Y$ satisfying:
	\begin{enumerate}
	\item Both $F, \ Y \in \mathrm{Shv}( \mathcal{C}, \tau)$;
	\item $Y$ is representable; 
	\item For any $\tau$-covering $\{ Y_i \to Y \}$ the pullback $F \times_Y Y_i$ is representable.
	\end{enumerate}
The object $F$ is also representable.
\end{definition}

\begin{remark}
When the geometric context is closed under $\tau$-descent the definition of a geometric stack becomes simpler since it turns out to be ambiguous to require the representability of the diagonal map, \cite[Corollary 8.6]{porta2016derived}.
\end{remark}

\begin{example}
We are mainly interested in the geometric context $( \Afd_k, \tau_{\text{\'et}}, \textbf{P}_{\mathrm{sm}})$. Here $\tau_{\text{\'et}}$ denotes the quasi-\'etale topology on $\Afd_k$, and $\textbf{P}_{\mathrm{sm}}$ 
denotes the collection of quasi-smooth morphisms, cf. \cite[\S 3]{berkovich1994vanishing}. Such geometric context is closed under $\tau_{\text{\'et}}$-descent, cf. \cite[Proposition 8.7]{porta2016derived}. 
\end{example}

\begin{notation}
	We will refer to geometric stacks for the geometric context $( \Afd_k, \tau_{\text{\'et}}, \textbf{P}_{\mathrm{sm}})$ simply by \emph{$k$-analytic stacks}.
\end{notation}

Let $G$ be a smooth group object in the $\infty$-category $\mathrm{Shv}( \mathcal{C}, \tau)$. Suppose that $G$ acts on a representable object $X$. The \emph{quotient stack of $X$ by the $G$-action} is defined as the (homotopy) colimit of the diagram,
	\[
	\begin{tikzcd}
		\dots \arrow[r, shift left = 1.5] \arrow[r, shift left= 0.5] \arrow[r, shift right=0.5] \arrow[r, shift right=1.5]
		& G^2 \times X \arrow[r, shift left =1] \arrow[r] \arrow[r, shift right=1] 
		& G \times X  \arrow[r, shift left=1] \arrow[r, shift right=1]
		& X.
	\end{tikzcd}
	\]
We denote such (homotopy) colimit by $[X/ G]$.

\begin{lemma} \label{stack_quotient}
Let $( \mathcal{C}, \tau, \textbf{P})$ be a geometric context satisfying $\tau$-descent. Let $G$ be a smooth group object in the $\infty$-category $\mathrm{Shv}( \mathcal{C}, \tau)$ acting on a representable object $X$. Then the quotient $[X/G]$ is a 
geometric stack.
\end{lemma}

\begin{proof}
It suffices to verify condition (1) of Definition 2.12. By definitoin of $[X/ G]$ we have a canonical morphism $X \to [X/ G]$ which is easily seen to be (-1)-representable and smooth. Therefore, $[X/G]$ is a 0-geometric stack.
\end{proof}

\begin{definition} \label{rep_def}
Let $G$ be a profinite group of topological finite presentation. We define the $k$-analytic stack of continuous representations of $G$ as
	\[
		\LocSys(G) \coloneqq [ \LocSysfr(G) / \anGLn] \in \St \big( \Afd_k, \tau_\et, \rmP_\sm \big).
	\]
\end{definition}

We now obtain the important result:

\begin{theorem} \label{thm:Theorem_1}
Let $G$ be a profinite group of topological finite presentation. Then the functor 
	\[	
		\LocSys(G) \colon \Afd_k \to \cS,
	\]
is representable by a geometric stack. 
\end{theorem}

\begin{proof}
The result is a direct consequence of \cref{stack_quotient} together with \cref{pr:Hom_G}.
\end{proof}

\begin{corollary}
Let $X$ be a smooth and proper scheme over an algebraically closed field. Then the $k$-analytic stack parametrizing continuous representations of $\pi^\emph{\et}_1(X)$ is representable by a geometric stack.
\end{corollary}

\begin{proof}
It follows immediately by \cref{thm:Theorem_1} together with the fact that under such assumptions on $X$ its \'etale fundamental group $\pi_1^\et(X)$ is topologically of finite generation.
\end{proof}

\subsection{Pro-etale lisse sheaves on $X_{\text{\'et}}$} Let $X$ be a proper and smooth scheme over an algebraically closed field $K$. Fix a geometric point
	\[
		x \colon \Spec K \to X.
	\]
Thanks to \cite[Thm 2.9, Expos\'e 10]{grothendieck224revetements}, the \'etale fundamental group
	\[
		\pi_\et^1(X) \coloneqq \pi_\et^1(X, x),
	\]
is profinite topologically of finite generation. Moreover,
\cite[Lemma 7.4.10]{bhatt2013pro} implies that the pro-\'etale and \'etale fundamental groups, of $X$, agree. For this reason, the \'etale fundamental group $\pi_1^\et(X)$ parametrizes \emph{rank $n$ pro-\'etale $\bQ_p$-local systems on $X$}, \cite[Lemma 7.4.7]{bhatt2013pro}. Our goal is to generalize the latter assertion to the case where we replace the field $\bQ_p$ by an affinoid $\bQ_p$-algebra, $A$. We refer the reader to
\cite[Definition 4.1.1]{bhatt2013pro} for the definition of the \emph{pro-\'etale site on $X$}, $X_{\mathrm{pro}\textrm{-\'etale}}.$

\begin{definition}[Noohi group]
Let $G$ be a topological group and consider the category of \emph{$G$-sets}, denoted $G$-Set. Consider the forgetful functor 
	\[
		F_G \colon 
		G\textrm{-}\mathrm{Set}\to \mathrm{Set}. 
	\]
We say that $G$ is a \emph{Noohi group} if there is a canonical equivalence $G \simeq \mathrm{Aut}(F_G)$, 
where $\mathrm{Aut}(F_G)$ is topologized via the compact-open topology on $\mathrm{Aut}(S)$, for each $S \in \mathrm{Set}$.
\end{definition}

\begin{lemma}
Let $G$ be a topological group such that it contains an open Noohi subgroup $U \le G$. Then $G$ is itself a Noohi group.
\end{lemma}
\begin{proof}
This is \cite[Lemma 7.1.8]{bhatt2013pro}.
\end{proof}

\begin{lemma} \label{Noohi}
Let $A$ be an $k$-affinoid algebra, then $\GLn(A)$ is a Noohi group.
\end{lemma}

\begin{proof}
Let $A_0$ be a formal model for $A$. By our choice, $A_0$ is a $p$-adically complete ring. We have furthermore a sequence of natural isomorphisms
	\begin{align*}
		\GLn( A_0) & \cong \underset{k}{\lim} \GLn( A_0 / p^k A_0 ) \\
				   & \cong  \underset{k}{\lim} \GLn(A_0 ) / \left( \mathrm{Id} + p^k \rmM_n( A_0) \right).
	\end{align*}
In particular, $\GLn(A_0$) is a pro-discrete group, as in \cite[Definition 2.1]{noohi2004fundamental}. Moreover, the system 
	\[\{ \GLn(A_0) \} \cup \{ \mathrm{Id} + p^m \cdot \rmM_n( A_0)  \}_{m \ge 1},\]
forms a basis of open normal 
subgroups of 
$\GLn(A_0)$. For this reason, $\mathrm{GL}_n(A_0)$ is a Noohi group, cf. \cite[Proposition 2.14]{noohi2004fundamental}. As $A_0$ is an open subring of $A$, we conclude that 
	\[\GLn(A_0) \le \GLn(A),\]
is an open subgroup. It then follows by
\cite[Lemma 7.1.8]{bhatt2013pro} that $\GLn(A) $ is a Noohi group, as desired. 
\end{proof}

\begin{construction} Let $A $ denote a $p$-adically complete $\Ok$-algebra.
	Let $\mathcal{F}_{\GLn(A)}$ denote the pre-sheaf on the pro-\'etale site $X_{\text{pro\'et}}$ defined by the formula
	\[	
		T \in X_{\text{pro-\'etale}} \mapsto \Map_\cont( \vert T \vert, \GLn(A)),
	\]
	where $\vert T \vert$ denotes the underlying topological space of the scheme $T$. The above formula actually defines a sheaf (of groups) on the pro-\'etale topology on $X$, cf. \cite[Lemma 4.2.12]{bhatt2013pro}.
\end{construction}

\begin{definition}
	The category of \emph{pro-\'etale free of rank $n$ $A$-local systems on $X$}, denoted $\mathrm{Loc}_{X, n}(A)$, is defined as the category of $\cF_{\GLn(A)}$-torsors on $X_{\textrm{pro-\'etale}}$.
\end{definition}

The following Proposition is an immediate generalization of \cite[Lemma 7.4.7]{bhatt2013pro}. We present its proof for the sake of completeness.

\begin{proposition}\label{scholze&bhatt}
Let $A$ be a $k$-affinoid algebra. Then there is a natural equivalence of groupoids,
	\[
		\LocSys(X)(A) \simeq \mathrm{Loc}_{X,n}(A).
	\]
\end{proposition}

\begin{proof}
Let $A_0$ be a formal model for $A$. As in the proof of \cref{Noohi}, $\GLn(A_0)$ is a pro-discrete subgroup (thus a Noohi subgroup) of $\GLn(A)$.
It thus follows from \cite[Lemma 7.4.6]{bhatt2013pro} that we are allowed to replace $A$ by $A_0$, in the statement of the Lemma. Let,
	\[
		\rho: \pi_1^\et(X) \to \GLn(A),
	\]
be a continuous representation. We define
	\[
		U \coloneqq \rho^{-1}( \GLn(A_0)).
	\]
In particular, $U$ is an open subgroup of $\pi_1^\et(X)$. 
It thus defines a pointed covering $X_U \to X$ such that 
	\[\pi_1^\et(X_U) \cong U.\]
The induced representation,
	\[
		\pi_1^\et(X_U) \to \GLn(A_0),
	\]
produces a well-defined element $M \in \mathrm{Loc}_{X_U, n}(A_0)$. Hence, after inverting $p$, it induces a 
local system $M' \in \mathrm{Loc}_{X_U}(A)$. Such element $M'$ comes equipped with descent data for $X_U \to X$ and therefore comes from a unique element $N(\rho) \in \mathrm{Loc}_{X,n}(A)$. Conversely, fix some element $ N \in \mathrm{Loc}_{X,n} ( A)$ 
which corresponds to a well defined $ \mathcal{F}_{\GLn(A)}$-torsor. 
Let $S \in \GLn(A)$-$ \mathrm{Set}$, we then have an induced representation,
	\[
		\rho_S :  \mathcal{F}_{\GLn(A)} \to \mathcal{F}_{\mathrm{Aut}(S)},
	\] 
of pro-\'etale local systems. The pushout of $N$ along $\rho_S$ defines thus an element $N_S \in \mathrm{Loc}_{X,n}(A)$. Moreover, its stalk at the base point is isomorphic to  $S$. Moreover, this association is functorial in $S$. It thus defines a functor
	\[
		\GLn(A) \textrm{-} \mathrm{Set} \to \mathrm{Loc}_{X,n}(A),
	\]
compatible with the fiber 
functor. By \cref{Noohi}, $\GLn(A)$ is a Noohi group. We can thus associate to it a continuous homomorphism 
	\[
		\rho_N : \pi_1^\et(X) \to \GLn(A).
	\]
This construction provides us with an inverse for the previously defined functor. The result now follows.
\end{proof}

\begin{corollary} \label{loc_vs_rep}
The non-archimedean stack $\LocSys(X)$ represents the functor $\Afd_k^\op \to \cS$ given on objects by the formula,
	\[ 
		A \mapsto \mathrm{Loc}_{X,n} ( A) \in \cS.
	\]
\end{corollary}
\begin{proof}
It follows by the construction of quotient stack combined with \cref{scholze&bhatt}.
\end{proof}

\section{Enriched \infcats}

\subsection{Preliminaries on enriched \infcats} 
In this \S,  we will state and prove certain results about enriched \infcats, that will be convenient for us. Our main reference is \cite{gepner2015enriched}. We do not attempt to give a full account of the results contained in the given citation. Instead, we only recall
a few notions that will be important for our study:

\begin{definition}
Let $\cV^\otimes$ be a presentably symmetric monoidal \infcat. We denote by $\mathrm{Alg}_{\mathrm{cat}}(\cV^\otimes)$, the \infcat of \emph{categorical algebras} on $\cV^\otimes$, as in \cite[Definition 4.3.1]{gepner2015enriched}.
\end{definition}

\begin{definition}A \emph{$\cV$-enriched \infcat} corresponds to a \emph{complete object} in $\mathrm{Alg}_{\mathrm{cat}}(\cV)$. We refer the reader to \cite[Definition 5.2.2]{gepner2015enriched}, for the notion of complete objects.
\end{definition}

\begin{definition}
Following \cite[Definition 5.4.3]{gepner2015enriched} we denote by $
\Cat (\cV^\otimes) \coloneqq \Cat^{\cV}$ the \infcat of \emph{$\cV$-enriched \infcats}. The latter corresponds to the full subcategory of $\mathrm{Alg}_{\mathrm{cat}}(\cV^\otimes)$, which is the (presentable) localization
at the family of \emph{essentially surjective fully faithful functors},
cf. \cite[Corollary 5.6.3 and Corollary 5.6.4]{gepner2015enriched}.
\end{definition}

\begin{notation}
	When the symmetric monoidal structure on $\cV$ is clear from the context, we simply denote $\Cat(\cV^\otimes)$ by $\Cat(\cV)$.
\end{notation}

We now introduce the notion of space of objects, as in \cite[Definition 5.1.1]{gepner2015enriched}:
\begin{definition}Consider the unit $\mathbf{1}_{\cV^\otimes} \in \cV^\otimes$. There is an essentially unique way to extend $\mathbf{1}_{\cV^\otimes} $ to an associative algebra object
	\[
		E_{\cV^\otimes}^0 \colon \Delta^\op \to \cV^\otimes.
	\]	
On the other hand, this defines a natural object in the \infcat $\Cat(\cV^\otimes)$.
We then define the \emph{space of objects} of $\cC$, denoted $\iota_0(\cC)$, as
	\begin{align*}
		\iota_0(\cC) & \coloneqq \Map_{\Cat(\cV^\otimes)}(E_{\cV^\otimes}^0, \cC) \\
				   & \in \cS.
	\end{align*}
\end{definition}

The following statement will prove to be very useful in our further study:

\begin{lemma} \label{lem:limits_of_enr_cats_on_objects}
	Let $\cV^\otimes$ be a presentably symmetric monoidal \infcat. Suppose we are given a small diagram
$F \colon I \to \Cat( \cV^\otimes)$. If the limit 
	\[
		\cC \coloneqq \underset{I}{\lim} F,
	\]
exists in the \infcat
$\Cat(\cV^\otimes)$ then the space of objects $\iota_0 \cC$ can be naturally identified with the limit
	\[	
		\lim_i (\iota_0(\cC_i)) \in \cS,
	\]
where, for each $i \in I$, $\cC_i \coloneqq F(i)$, in $\Cat(\cV^\otimes)$. 
\end{lemma}

\begin{proof}
	In this case, we have a chain of equivalences in $\cV^\otimes$
	\begin{align*}
		\iota( \cC) & \simeq \\
		& \simeq  \Map_{  \Cat(\cV^\otimes)} \big( E_{\cV^\otimes}^0, \cC)   \\
		& \simeq \underset{i \in I} \lim \Map_{\Cat(\cV^\otimes)} \big(E_{\cV^\otimes}^0, \cC_i \big) \\
		& \simeq  \underset{i \in I}{\lim}  (\iota_0(\cC_i)).
	\end{align*}
This finishes the proof
of the statement.
\end{proof}

We will also need to recall \cite[Definition 5.1.3]{gepner2015enriched}:

\begin{definition}
	Let $\cV^\otimes$ denote a presentably symmetric monoidal \infcat. The \infcat of categorical algebras $\mathrm{Alg}_{\mathrm{cat}}(\cV^\otimes)$ is naturally tensored over $\mathrm{Alg}_{\mathrm{cat}}(\cS)$, cf.
	\cite[Corollary 4.3.17]{gepner2015enriched}. Let $[1]$ denote the ordered category
	$\{0, 1\}$ regarded as an \infcat. We denote by
		\[
			[1]_{\cV} \coloneqq [1] \otimes \mathbf{1}_{\cV^\otimes} \in \mathrm{Alg}_{\mathrm{cat}}(\cV^\otimes),
		\]
	Let $\cC \in \Cat(\cV^\otimes)$. We define the \emph{space of morphisms} of $\cC$ as the mapping space
		\[
			\Map_{\mathrm{Alg}_{\mathrm{cat}}(\cV^\otimes)}([1]_{\cV}, \cC).
		\]
	We refer the reader to \cite[Lemma 5.1.4.]{gepner2015enriched} for a justifiication of this definition.
\end{definition}

\begin{notation}
	The inclusion morphisms $\{ 0 \} \to [1] \leftarrow \{ 1\}$ induce a natural \emph{source-target} morphism
		\[
			(s, t) \colon \Map_{\mathrm{Alg}_{\mathrm{cat}}(\cV^\otimes)}([1]_{\cV^\otimes}, \cC) \to \Map_{\mathrm{Alg}_{\mathrm{cat}}(\cV^\otimes)}(E_{\cV^\otimes}^0, \cC)^{\times 2}.
		\]
	Given $x, y \in \iota(\cC)$ we denote by
		\[
			\underline{\Map}_{x \coprod y / \cC} ([1]_{\cV^\otimes}, \cC) \coloneqq \Map_{\mathrm{Alg}_{\mathrm{cat}}(\cV^\otimes)}([1]_{\cV^\otimes}, \cC) \times_{ \Map_{\mathrm{Alg}_{\mathrm{cat}}(\cV^\otimes)}(E_{\cV^\otimes}^0, \cC)^{\times 2}} \{ (x, y) \},
		\]
	the \emph{space of morphisms in $\cC$ with source $x \in \iota(\cC)$ and target $y \in \iota(\cC)$}. We shall also denote by
		\begin{align*}
		 	\cC(x, y) &  \coloneqq \underline{\Map}_{x \coprod y / \cC} ([1]_{\cV^\otimes}, \cC), \\
			\mathbf{Map}_{\cC}(x, y) &  \coloneqq \cC(x, y).
		\end{align*}
\end{notation}

\begin{lemma} \label{lim:tech}
Let $\cV^\otimes$ be a presentably symmetric monoidal \infcat. Suppose we are given a small diagram
$F \colon I \to \Cat( \cV^\otimes)$. Then the limit 
	\[
		\cC \coloneqq \underset{I}{\lim} F,
	\]
exists in the \infcat
$\Cat(\cV^\otimes)$. Furthermore, given any two objects $x, \ y \in \cC$, we have
an equivalence of mapping objects
	\[
		\cC (x, y) \simeq \underset{i \in I}{\lim} \cC_i (x_i, y_i) \in \cV^\otimes,
	\]
where, for each $i \in  I$, $\cC_i \coloneqq F(i)$ and $x_i$, $y_i$ denote the images of both $x$ and $y$
under the projection functor $\cC \to \cC_i$, respectively.
\end{lemma}

\begin{proof} The existence statement follows from presentability of $\Cat(\cV^\otimes)$, cf. \cite[Corollary 5.4.5]{gepner2015enriched}.
In this case, we have a chain of equivalences in $\cV^\otimes$
	\begin{align*}
		\cC(x, y) & \simeq \\
		& \simeq  \underline{\Map}_{ x \coprod y / \cC)} \big( [1]_{\cV}, \cC)   \\
		& \simeq \underset{i \in I} \lim \underline{\Map}_{x_i \coprod y_i / \cC} \big([1]_{\cV}, \cC_i \big) \quad \mathrm{(by \ \cref{lem:limits_of_enr_cats_on_objects})} \\
		& \simeq  \underset{i \in I}{\lim} \cC_i (x_i, y_i ).
	\end{align*}
This finishes the proof
of the statement.
\end{proof}

We now recall the definition of fully faithful functors, cf. \cite[Definition 5.3.1]{gepner2015enriched}

\begin{definition}
	Let $F \colon \cC \to \cD$ be a functor between $\cV^\otimes$-enriched \infcats. We say that that $F$ is \emph{fully faithful} if for every $C, \ D \in \iota_0(\cC)$, the canonical morphism
		\[
			\mathbf{Map}_{\cC}(C, D) \to \mathbf{Map}_{\cD}(F(C), F(D)),
		\]
	is an equivalence in $\cV$.
\end{definition}

We will also need to introduce the notion of the \emph{space of equivalences} in enriched \infcats, cf. \cite[Definition 5.1.6]{gepner2015enriched}:

\begin{definition} Let $\cV^\otimes$ denote a presentably symmetric monoidal \infcat.
	We define the trivial $\cV^\otimes$-enriched \infcat $E^1$ as the composite
		\[
			\Delta^\op_{[1]} \to \Delta^\op \xrightarrow{\mathbf{1}_{\cV^\otimes}} \cV^\otimes,
		\]
	where $\Delta^\op_{[1]}$ is defined as in \cite[Definition 4.1.1]{gepner2015enriched}. Let $\cC \in \Cat(\cV^\otimes)$. We define the \emph{space of equivalences} in $\cC$ as 
		\[
			\Map_{\mathrm{Alg}_{\mathrm{cat}}(\cV^\otimes)}(E^1, \cC) \in \cS.
		\]
	See \cite[Proposition 5.1.11]{gepner2015enriched} for a justification of this definition.
\end{definition}	
We also introduce the notion of essentially surjective functor, cf. \cite[Definition 5.3.3]{gepner2015enriched}.

\begin{definition}
	A functor $F \colon \cC \to \cD$ between $\cV^\otimes$-enriched \infcats is \emph{essentially surjective} if and only if for every $D \in \iota_0 \cD$, there exists $C \in \iota_0 \cC$ such that there exists an equivalence
		\[
		 	E^1 \to \cD,
		\]
	whose both target and source are equivalent to $D$  and $F(C)$, respectively.
\end{definition}

The following result will prove to be very useful in our study:

\begin{proposition} \label{prop:ff_es_equivalences_enriched}
 Let $\cV^\otimes$ denote a presentably symmetric monoidal \infcat. Let $F \colon \cC \to \cD$ be a functor between $\cV^\otimes$-enriched \infcats. Then $F$ is an equivalence in $\Cat(\cV^\otimes)$ if and only if it is both essentially surjective and
 fully faithful.
\end{proposition}

\begin{proof}
	The content of the proposition is contained in \cite[ Corollary 5.6.3 and Corollary 5.6.4]{gepner2015enriched}.
\end{proof}

\begin{proposition} \label{prop:monoids embed ff in enriched infcats}
Let $\cV^\otimes$ denote a presentably Cartesian symmetric monoidal \infcat. Then we have a natural fully faithful embedding
	\[
		F \colon \Mon_{\bE_1}(\cV^\otimes) \to \Cat(\cV^\otimes).
	\]
\end{proposition}

\begin{proof}
	Since the symmetric monoidal structure on $\cV^\otimes$ is Cartesian, it is clear from the definitions, cf. \cite[Definition 4.3.1]{gepner2015enriched}, that we have a natural fully faithful functor
		\[
			F \colon \Mon_{\bE_1}(\cV^\otimes) \to \mathrm{Alg}_{\mathrm{cat}}(\cV^\otimes).
		\]
	It suffices to show that $F$ factors through the right adjoint inclusion functor $\Cat(\cV^\otimes) \subseteq \mathrm{Alg}_{\mathrm{cat}}(\cV^\otimes).$ Let $M \in \Mon_{\bE_1}(\cV^\otimes)$. Thanks to \cite[Proposition 5.1.11]{gepner2015enriched},
	it follows that $F(M)$ has contractible spaces of objects and equivalences.
	The result now follows from \cite[Corollary 5.2.10]{gepner2015enriched} combined with \cite[Proposition 5.4.4]{gepner2015enriched}.
\end{proof}

\begin{remark} Let $\cV^\otimes$ denote a presentably Cartesian symmetric monoidal \infcat.
	Given $M \in \Mon_{\bE_1}(\cV^\otimes)$, we can regard the latter as an $\cV^\otimes$-\infcat whose space of objects is contractible. If we denote $\mathbf 1 \in \iota_0 ( F( M) )$ the underlying object, we have that
		\begin{align*}
			M  & \simeq \mathbf{Map}_{F(M)}(\mathbf 1, \mathbf 1)  \\
											& \in \cV^\otimes.
		\end{align*}
\end{remark}

\subsection{Enriched pro-$p$-nilpotent \infcats} In this \S, we will introduce the \infcat of \emph{enriched pro-$p$-nilpotent} \infcats. We will compare our definitions with other known constructions. The material in this section will be mainly used to study enrichments of the \infcats
of $\kc$-adic perfect modules.

\begin{definition}
We define the \infcat of \emph{$p$-nilpotent $\Ok$-modules,} denoted 
	\[
		 \Mod_{\Ok} ^{\mathrm{nil}} \subseteq \Mod_{\Ok},
	\]
as the full subcategory of $\Mod_{\Ok}$ spanned by $p$-nilpotent $\kc$-modules, see \cite[Definition 7.1.1.1]{lurie2016spectral} for a definition of the latter.
\end{definition}

\begin{remark}
	The \infcat $\Mod_{\Ok}^{\mathrm{nil}}$ is a stable \infcat.
\end{remark}

\begin{definition}
	The \infcat of \emph{pro-objects of $p$-nilpotent $\kc$-modules} is defined as the \infcat $\pro(\Mod_{\kc}^{\mathrm{nil}}) $.
\end{definition}

\begin{lemma} \label{lem:sym_mon_structure_on_pro_nil_modules}
The \infcat $\Mod_{\Ok}^{ \mathrm{nil}}$ admits a naturally symmetric monoidal structure induced from the $\otimes_{\kc}$-product
on $\Mod_{\Ok}$.  Similarly, the \infcat $\pro(\Mod_{\kc}^{\mathrm{nil}})$ admits a canonical symmetric monoidal structure such that the canonical inclusion functor 
	\[\Mod_{\Ok}^{\mathrm{nil}} \subseteq \pro( \Mod_{\Ok}^{ \mathrm{nil}}),\] is
symmetric monoidal.
\end{lemma}

\begin{proof} The \infcat $\Mod_{\Ok}^{\mathrm{nil}}$ admits a natural symmetric monoidal structure induced from the usual one on $\Mod_{\Ok}$.  We first show that given $M,  N \in \Mod_{\Ok}^{\mathrm{nil}}$, their tensor product

	\[
		M \otimes_{\Ok} N \in \Mod_{\Ok},
	\]
is still nilpotent. This last assertion follows from the chain of equivalences
	\begin{align*}
		(M \otimes_{\Ok} N) \otimes_{\Ok} k & \simeq ( M \otimes_{\Ok} k) \otimes_k (N \otimes_{\Ok} k) \\
								       & \simeq 0,
	\end{align*}
where the last equivalence follows from nilpotence of both $M$ and $N$. 
Moreover, \cite[Proposition 8.2.2.5]{lurie2016spectral} implies that $\Mod_{\kc}^{\mathrm{nil}}$ admits a uniquely defined unit object, denoted $\kc^\mathrm{nil}$.

Thanks to the dual of \cite[Corollary 6.3.1.13]{lurie2012higher} the \infcat of pro-objects $\pro(\Mod_{\Ok}^{\mathrm{nil}} )$ admits a uniquely defined, up to contractible indeterminacy,
symmetric monoidal structure for which,
given 
$M \coloneqq \lim_{\alpha} M_{\alpha} $ and $N \coloneqq \lim_\beta N_\beta$ in $\Mod_{\Ok}^{\mathrm{pro \textrm{-} nil}}, $ we have the formula
	\[
		M \otimes N \simeq \lim_{\alpha, \ \beta} \big( M_\alpha \otimes_{\kc^\mathrm{nil}} N_\beta)
	\]
in the \infcat $\Mod_{\Ok}^{\mathrm{pro \textrm{-} nil}} $. The second assertion of the Lemma follows from the previous considerations.
\end{proof}

\begin{remark}
Thanks to the dual of \cite[Proposition 1.1.3.6]{lurie2012higher} it follows that $\pro(\Mod_{\Ok}^{\mathrm{nil}})$ is a stable \infcat.
\end{remark}

\begin{remark}
 Let $\kc \coloneqq \{ \cO_{k, n} \}_{n \ge 1} \in \pro(\Mod_{\Ok}^\mathrm{nil})$. The latter admits a natural structure of a commutative algebra object on $\pro(\Mod_{\Ok}^\mathrm{nil})$, with respect to the symmetric monoidal structure considered on \cref{lem:sym_mon_structure_on_pro_nil_modules}.
\end{remark} 

\begin{definition}
We shall denote by $\Mod_{\Ok}^{\mathrm{pro \textrm{-} nil}} \coloneqq \Mod_{\Ok}(\pro \big( ( \Mod_{\kc} )_\nil \big))$, the \infcat of \emph{pro-$p$-nilpotent $\kc$-modules}.
\end{definition}

\begin{remark} \label{rem:conservativity of kc-mod}
Notice that, in particular, the forgetful functor $G \colon \Mod_{\Ok}^{\mathrm{pro \textrm{-} nil}} \to \pro(\Mod_{\kc}^\mathrm{nil})$ is conservative and admits a (symmetric monoidal) left adjoint, cf. \cite[Theorem 6.2.2.5]{lurie2012higher}.
\end{remark}

\begin{corollary} \label{cor:i_nil_at_the_level_of_enriched_cats}
	The inclusion functor $\Mod_{\Ok}^{\mathrm{nil}} \subseteq \Mod_{\Ok}^{\mathrm{pro} \textrm{-} \mathrm{nil}}$ induces a well defined, up to contractible indeterminacy, functor
		\[
			i_{\mathrm{nil}} \colon \Cat(\Mod_{\Ok}^{\mathrm{nil}}) \to \Cat(\Mod_{\Ok}^{\mathrm{pro} \textrm{-} \mathrm{nil}}).
		\]	
\end{corollary}

\begin{proof}
	It follows immediately from \cite[Corollary 5.7.6]{gepner2015enriched} combined with \cref{lem:sym_mon_structure_on_pro_nil_modules}.
\end{proof}

\begin{lemma} \label{lem:mat_on_modules_is_lax_sym_mon}
	Consider the materialization functor $\Mat_{\Ok} \colon \Mod_{\Ok}^{\mathrm{pro} \textrm{-} \mathrm{nil}} \to \Mod_{\Ok}$ given on objects by the formula
		\[
			\{ M_n \} \in  \Mod_{\Ok}^{\mathrm{pro} \textrm{-} \mathrm{nil}} \mapsto \lim_{n} M_n \in  \Mod_{\Ok}.
		\]
	Then the functor $\Mat_{\Ok}$ is lax symmetric monoidal.
\end{lemma}

\begin{proof} We first observe that the functor $\Mat_{\Ok}$ can be described as the composite
	\[
		\Map_{\pro(\Mod_{\kc}^{\mathrm{nil}})}(\kc^\mathrm{nil}, - ) \circ G \colon \Mod_{\kc}^{\mathrm{pro} \textrm{-} \mathrm{nil}} \to \Mod_{\kc},
	\]
	where $G$ is as in \cref{rem:conservativity of kc-mod}. The latter being lax symmetric monoidal we are reduced to show that
		\[
			\Map_{\pro(\Mod_{\kc}^{\mathrm{nil}})}(\kc^\mathrm{nil}, - ) \colon \pro(\Mod_{\kc}^{\mathrm{nil}}) \to \Mod_{\kc},
		\]
	is lax symmetric monoidal.
	Let $\{ M_n \}_n, \ \{ N_m \}_m \in \pro( \Mod_{\Ok}^{\mathrm{nil}})$. Then the transition morphisms
		\begin{align*}
			\{ M_n
			 \}_n \to M_n, \\
			\{ N_m \}_\beta \to N_m,
		\end{align*}
	induce canonically defined morphisms
		\begin{align} \label{eq:trans_of_tensor_products}
			\lim_{n} M_n \to M_n , \\
			\lim_{m} N_m \to N_m,
		\end{align}
	in the \infcat $\Mod_{\Ok}$.
	By taking tensor products in \eqref{eq:trans_of_tensor_products} above, we obtain a canonical morphism
		\[
			\lim_{n}(M_n) \otimes_{\kc} \lim_{m} (N_m) \to \lim_{n, m} (M_n \otimes_{\kc} N_n),
		\]
	in $\Mod_{\kc}$. This observation provides us with a natural morphism
		\[
			\theta_{M, N} \colon \Mat_{\kc} ( \{M_n \}) \otimes_{\kc} \Mat_{\kc}( \{ N_m \}) \to \Mat_{\kc}( \{M_n\} \otimes_{\kc} \{N_m\}),
		\]
	in the \infcat $\Mod_{\kc}$. The morphisms $\{\theta_{M, N} \}$ assemble thus inducing the lax symmetric monoidal structure on $\Mat_{\kc}$, the result now follows.
\end{proof}

\begin{construction} \label{constr_of_forget_from_kc_mod_to_spaces}
	Consider the forgetful functor $F \colon \Mod_{\Ok} \to \Sp$. The functor $F$ is lax symmetric monoidal since it is right adjoint to the base change functor
		\[
			- \otimes \Ok \colon \Sp \to \Mod_{\Ok}.
		\]
	The latter being symmetric monoidal. Precomposing with the lax symmetric monoidal functor
		\[
			\Omega^\infty \colon \Sp^\otimes \to \cS^\times,
		\]
	we obtain a lax symmetric monoidal functor $\Mod_{\Ok} \to \cS$.
\end{construction}

\begin{lemma} \label{lem:mat_on_pro_nil}
The composite functor $\Omega^\infty \circ F \circ \Mat_{\Ok}$ induces a well defined functor
		\[
			\Mat_{\kc}^{\mathrm{cat}} \colon \Cat(\ind(\Mod_{\kc}^{\mathrm{pro} \textrm{-} \mathrm{nil}})) \to \Cat.
		\]
\end{lemma}

\begin{proof}
It is an immediate consequence of \cite[Corollary 5.7.6]{gepner2015enriched} combined with \cref{constr_of_forget_from_kc_mod_to_spaces}.
\end{proof} 

\begin{definition}
	We shall refer to the functor $\Mat_{\kc}^\mathrm{cat}$ as the categorical \emph{$\kc$-linear materialization functor}.
\end{definition}

\begin{proposition} \label{prop:mat_commutes_with_cofiltered_limits}
	The materialization functor $\Mat_{\Ok}^\mathrm{cat}$ commutes with cofiltered limits.
\end{proposition}

\begin{proof} Let $F \colon I \to \Cat(\Mod_{\Ok}^{\mathrm{pro}\textrm{-}\mathrm{nil}})$ be a cofiltered diagram of $\Mod_{\Ok}^{\mathrm{pro}\textrm{-}\mathrm{nil}}$-enriched \infcats. Denote by $\cC_i \coloneqq F(i) \in \Cat( \Mod_{\Ok}^{\mathrm{pro}\textrm{-}\mathrm{nil}})$. We set $\cC \coloneqq \lim_I F \in \Cat(\Mod_{\Ok}^{\mathrm{pro}\textrm{-}\mathrm{nil}})$.
 	By the universal property of limits we have a canonical functor
		\[	
			\theta \colon \Mat_{\Ok}^\mathrm{cat} (\cC) \to \lim_I \Mat_{\Ok}^\mathrm{cat} ( \cC_i),			
		\]
	of (usual) \infcats. We claim that $\theta$ is an equivalence of \infcats. We start by proving fully faithfulness of $\theta$.
	First observe that the composite functor 
		\[ \Omega^\infty \circ F \circ \Mat_{\kc} \colon \Mod_{\Ok}^{\mathrm{pro}\textrm{-}\mathrm{nil}} \to \cS,\]
	commutes with cofiltered limits. It is clear that $\Mat_{\Ok}$ does commutes with cofiltered limits. The conclusion now follows from the fact that
	both $F$ and $\Omega^\infty$ are right adjoint functors. Fully faithfulness now follows from \cref{lim:tech} combined with \cref{rem:conservativity of kc-mod}.
	Essential surjectivity follows from \cref{lem:limits_of_enr_cats_on_objects} combined with \cite[Lemma 5.1.2]{gepner2015enriched} together with \cite[Theorem 3.3.3.2]{lurie2009higher}.
\end{proof}

\subsection{Enriched pro-$p$-nilpotent perfect modules} We study the different enrichments on the \infcat of perfect $A$-modules, denoted $\Perf(A)$, for $A \in \adCAlg$. Recall from the definition of $\adCAlg$ that $A$ is necessarily $(p)$-complete. For this reason, we have the following statement:

\begin{lemma} \label{lem:perf_complete}
	Let $A \in \adCAlg$. Then the natural functor
		\[
			f \colon \Coh(A)  \to \lim_{n \ge 1} \Coh(A_n),
		\]
	is an equivalence of \infcats. The analogous statement holds for the stable \infcat of perfect $A$-modules, $\Perf(A)$.
\end{lemma}

\begin{proof} Let $i_n \colon A \to A_n $ denote the natural morphism, for each $n \ge 1$. We denote by 
		\[
			i_{n, *} \colon \Mod_{A_n} \to \Mod_A,
		\]
	the associated forgetful functor. Since $i_n$ is a complete intersection morphism, the functor $i_{n, *}$ restricts to well defined functor
		\[
			i_{n, *} \colon \Perf(A_n ) \to \Perf(A) , \quad i_{n, *} \colon \Coh(A_n) \to \Coh(A),
		\] 
	right adjoint to the base change functor $ - \otimes_A A_n$. For this reason, the functor
		\[
			F \colon \Coh(A) \to \lim_{n \ge 1} \Coh(A_n),		
		\]
	given on objects by the formula	
		\[
			M \in \Coh(A) \mapsto \{ M \otimes_{\kc} \cO_{k, n} \}_{n \ge 1}	\in \lim_{n \ge 1} \Coh(A),
		\]
	admits a natural right adjoint. Indeed, the latter can be described by the formula
		\[
			\{ M_n \}_{n \ge 1} \in \lim_{n \ge 1} \Coh(A_n) \mapsto \lim_{n \ge 1}(i_{n, *}(M_n)) \in \Coh(A).
		\]
	Let us denoted such functor by
		\[
			G \colon \lim_{n \ge 1} \Coh(A_n) \to \Coh(A).
		\]
	We claim that both $F$ and $G$ are fully faithful, which proves the claim.  To show this we claim that for every $M \in \Coh(A)$ the natural morphism
		\[
			M \to \lim_{n \ge 1} i_{n, *} (M \otimes_A A_n),
		\]
	is an equivalence in $\Coh(A)$. Consider the Koszul resolution
		\[
			M \xrightarrow{t^n} M \to i_{n , *} (M \otimes_A A_n).
		\]
	By passing to the limit over $n \ge 1$, we obtain a fiber sequence of the form
		\[
			\lim_t M \to M \to \lim_{n \ge 1} i_{n, *} (M \otimes_{A } A_n).
		\]
	We claim that $\lim_t M \simeq 0$. Let $N \in \Perf(A)$ such that we have a morphism $h \colon N \to M$ whose cofiber, denoted $P$, is $m$-truncated for some $m \ge 0$. Then
		\[
			\lim_t N \to \lim_t M \to \lim_t P,
		\]
	is a fiber sequence such that $\lim_t P$ is $(m-1)$-truncated (indeed, in the \infcat $\Mod_A$ limits have at most non-vanishing $\lim^1$). By varying $N \in \Perf(A)$ such that the cofiber of $h$, $P$, is $m$-connective, for
	increasing $m \ge 0$, we reduce ourselves to prove the statement in the case where $M = N$ is perfect itself. 
	
	Since $N \in \Perf(A)$, we have that $N$ can be realized as a finite sequence of finite colimits and retracts of $A \in \Mod_A$. Since $\lim_t$ commutes with finite colimits and retracts, we reduce ourselves to prove that
		\[
			\lim_t A \simeq 0.
		\]
	But this follows from the fact that the morphism $A \to \lim_{n \ge 1} i_{n, *} A_n$ is an equivalence, since $A$ is itself $(p)$-complete.
	Let us now prove that both $f$ and $g$ are fully faithful. Let $M , N \in \Coh(A)$, then we have a chain of natural equivalences
		\begin{align*}
			\Map_{\Coh(A)}(M, N) & \simeq \lim_{n \ge 1} \Map_{\Coh(A)}(M, i_{n, *} (N \otimes_A A_n)) \\
							  & \simeq \lim_{n \ge 1} \Map_{\Coh(A_n)}(M \otimes_A A_n, N \otimes_A A_n) \\
							  & \simeq \Map_{\lim_{n \ge 1} \Coh(A_n)}(f(M), f(N)).
		\end{align*}
	By a similar reasoning as before we conclude that $g$ is itself fully faithful. Indeed, the only non-trivial step is to show that given $\{M_n \}_{n \ge 1} \in \lim_{n \ge 1}\Coh(A_n)$, then
		\[
			\lim_{n \ge1 } i_{n, *} (M_n) \otimes_A A_m \simeq i_{m, *}(M_m),
		\]
	for every $m \ge 1$. But this is clear by the existence of a fiber sequence
		\[
			\lim_{n \ge 1} M_n \xrightarrow{t^m} \lim_{n \ge 1}M_n \to i_{m, *}(M_m),
		\]
	which identifies the latter with $i_{m, *}(\lim_{n \ge1 } i_{m, *} (M_n) \otimes_A A_m)$, as desired.
	It is clear from the preceding considerations that the analogous statement holds when we $\Coh$ replaced by $\Perf$. 
\end{proof}

\begin{construction} \label{const:enrichment_over_pro_nil_mod} Let $n \ge 1$, be an integer.
	Consider the functor
		\[
			g_n \colon \Mod_{\cO_{k, n}} \to \Mod_{\Ok},
		\]
	right adjoint to the base change functor $- \otimes_{\Ok} \cO_{k, n} \colon \Mod_{\Ok} \to \Mod_{\cO_{k, n}}.$ The functor $g_n$ factors through the full subcategory of nilpotent objects
		\[
			\Mod_{\Ok}^{\mathrm{nil}} \subseteq \Mod_{\Ok}.
		\]
	For this reason, for every $A \in \adCAlg$ and $n \ge 1$ the \infcat $\Perf(A_n)$, (resp., $\Coh(A_n)$) are naturally enriched over the \infcat $\Mod_{\Ok}^{\mathrm{nil}}$. Thanks to \cref{const:enrichment_over_pro_nil_mod} it follows that
	we can consider $\Perf(A_n)$, (resp., $\Coh(A_n)$) as enriched over the symmetric monoidal \infcat $\Mod_{\Ok}^{\mathrm{pro} \textrm{-} \mathrm{nil}}$.
\end{construction}

\begin{definition} Let $A \in \adCAlg$ and $n \ge 1$, and consider $\Perf(A_n) \in \Cat(\Mod_{\Ok}^{\mathrm{pro}\textrm{-}\mathrm{nil}})$ via \cref{const:enrichment_over_pro_nil_mod}.
	We define 
		\[\mathbf{Perf}(A) \coloneqq \lim_{n \ge 1} \Perf(A_n) ,\]
	in $\Cat(\Mod_{\Ok}^{\mathrm{pro}\textrm{-}\mathrm{nil}})$. Similarly, we define $\mathbf{Coh}^+(A) \coloneqq \lim_{n \ge 1} \Coh(A_n)$. We refer these as the \emph{pro-$p$-nilpotent enriched \infcat of perfect $A$-modules} and \emph{pro-$p$-nilpotent enriched \infcat of almost perfect $A$-modules}, respectively.
\end{definition}

\begin{remark}
	Thanks to \cref{lim:tech} it follows that given objects $M , \ N \in \mathbf{Perf}_{\Ok}(A) $ we have an equivalence of pro-objects
		\[
			\mathbf{Map}_{\mathbf{Perf}_{\Ok}(A)}(M, N) \simeq \{ \underline{ \Map}_{\Perf(A_n)}(M_n, N_n) \}_{n \ge 1} 
		\]
	in the \infcat $\Mod_{\Ok}^{\mathrm{pro}\textrm{-}\mathrm{nil}}$, where $\underline{ \Map}$ denotes the $\Mod_{\kc}^\mathrm{nil}$-enriched mapping object.
\end{remark}
\begin{lemma} \label{lem: pro-p nilpotent enrichement of perf}
	Let $A \in \adCAlg$. The $\kc$-linear materializations 
		\[
			\Mat_{\kc}^\mathrm{cat} (\mathbf{Perf}_{\Ok}(A))\quad \mathrm{and} \quad \Mat_{\kc}^\mathrm{cat}(\mathbf{Coh}^+(A)  \in \Cat( \cS) 
		\]
	are naturally equivalent to $\Perf(A)$ and $\Coh(A)$, respectively.
\end{lemma}

\begin{proof}
The assertion is a direct consequence of \cref{prop:mat_commutes_with_cofiltered_limits} combined with \cref{lem:perf_complete}.
\end{proof}

\subsection{Enriched  ind-pro-$p$-\infcats} 
\begin{definition}
Let $\kc \in \Mod_{\Ok}^{\mathrm{pro \textrm{-} nil}}$ denote the unit object. Then multiplication by $p$ defines an endomorphism
	\[
		p \colon \kc \to \kc.
	\]
For this reason, given any $M \in  \Mod_{\Ok}^{\mathrm{pro \textrm{-} nil}}$, we have an induced endomorphism of multiplication by $p$, which we denote by 
	\[
		p \colon M \to M.
	\]
\end{definition}

\begin{construction}
Consider now the \infcat $\ind \big( \Mod_{\Ok}^{\mathrm{pro \textrm{-} nil}} \big)$ of ind-objects on the presentable stable
\infcat $\Mod_{\Ok}^{\mathrm{pro \textrm{-} nil}} $. Consider the functor
	\[
		- \otimes_{\Ok} k \colon \ind \big( \Mod_{\Ok}^{\mathrm{pro \textrm{-} nil}}  \big) \to \ind \big( \Mod_{\Ok}^{\mathrm{pro \textrm{-} nil}}  \big)
	\]
given informally on objects by the formula
	\[
		M \mapsto M \otimes_{\Ok} k \coloneqq \colim_{\textrm{mult } p} M,
	\]
where the latter denotes the filtered colimit induced by multiplication by the $p$-map $p \colon M \to M$.
\end{construction}

\begin{lemma} \label{lem:enriched_sym_monoidal_base_change}
The \infcat $\ind(\Mod_{\Ok}^{\mathrm{pro \textrm{-} nil}})$ admits a natural symmetric monoidal structure induced from the natural symmetric monoidal structure
on $\Mod_{\Ok}^{\mathrm{pro \textrm{-} nil}} $. Moreover, the functor
	\[
		- \otimes_{\Ok} k \colon \Mod_{\Ok}^{\mathrm{pro \textrm{-} nil}} \to \ind(\Mod_{\Ok}^{\mathrm{pro \textrm{-} nil}})
	\]
admits an essentially unique natural extension to a symmetric monoidal functor \[(- \otimes_{\Ok} k)^\otimes \colon 
(\Mod_{\Ok}^{\mathrm{pro \textrm{-} nil}})^\otimes  \to \ind(\Mod_{\Ok}^{\mathrm{pro \textrm{-} nil}})^\otimes.\]
\end{lemma}

\begin{proof}
The symmetric monoidal structure on $\ind \big( \Mod_{\Ok}^{\mathrm{pro \textrm{-} nil}}  \big)$ is induced by the symmetric monoidal
structure on $\Mod_{\Ok}^{\mathrm{pro \textrm{-} nil}} $, by extending it via filtered colimits. This defines a well defined symmetric monoidal structure on $\ind \big( \Mod_{\Ok}^{\mathrm{pro \textrm{-} nil}}  \big)$, cf. \cite[Corollary 6.3.1.13]{lurie2012higher}.
By a direct computation, given $M , \ M' \in
\Mod_{\Ok}^{\mathrm{pro \textrm{-} nil}} $ we have natural equivalences 
	\begin{align*}
		(M \otimes_{\Ok} k) \otimes (M' \otimes_{\Ok} k) & \simeq \big( M \otimes M' \big)  \otimes_{\Ok} k \\
											& \in \ind( \Mod_{\Ok}^{\mathrm{pro \textrm{-} nil}}) ,
	\end{align*}
which establishes the second assertion of the Lemma.
\end{proof}

\begin{corollary} \label{cor:ind_colimit_by_mult_by_p_induces_functor_of_enriched_cats}
 The symmetric monoidal functor $- \otimes_{\Ok} k \colon \Mod_{\Ok}^{\mathrm{pro \textrm{-} nil}} \to \ind(\Mod_{\Ok}^{\mathrm{pro \textrm{-} nil}})$ induces a well defined functor
 	\[
		- \otimes_{\Ok} k \colon \Cat( \Mod_{\Ok}^{\mathrm{pro \textrm{-} nil}} ) \to \Cat( \ind(  \Mod_{\Ok}^{\mathrm{pro \textrm{-} nil}} )),
	\]
to which we refer as the {base change functor along $\kc \to k$}.
\end{corollary}

\begin{proof}
It is a direct consequence of \cite[Corollary 5.7.6]{gepner2015enriched} combined with \cref{lem:enriched_sym_monoidal_base_change}.
\end{proof}

\begin{definition}
	Let $A \in \adCAlg$. We denote by 
		\[\mathbf{Coh}^+(A \otimes_{\kc} k ) \coloneqq \mathbf{Coh}^+(A) \otimes_{\kc} k,\]
	in $\Cat(\ind(\Mod_{\Ok}^{\mathrm{pro \textrm{-} nil}}))$, where $- \otimes_{\kc} k$ is the functor introduced in  
	\cref{cor:ind_colimit_by_mult_by_p_induces_functor_of_enriched_cats}.
\end{definition}

\begin{construction} \label{const:k linear materialization}
Consider the \emph{$k$-linear materialization functor}
	\[
		\Mat_k \colon \ind(\Mod_{\Ok}^{\mathrm{pro \textrm{-} nil}}) \to \cS,
	\]
given on objects by the formula
	\[
		M \in \ind(\Mod_{\Ok}^{\mathrm{pro \textrm{-} nil}}) \mapsto \Map_{\ind(\Mod_{\Ok}^{\mathrm{pro \textrm{-} nil}})}(\Ok, M) \in \cS.
	\]
The latter functor agrees with the extension under filtered colimits of the lax symmetric monoidal functor, $\Omega^\infty \circ F \circ \mathrm{Mat}_{\kc}$, introduced in \cref{lem:mat_on_pro_nil}. For this reason, $\Mat_k$ is itself
lax-symmetric monoidal. Thanks to \cite[Corollary 5.7.6]{gepner2015enriched}, we have a well defined induced functor
	\begin{equation} \label{eq:mat_functor_ind_mod}
		\Mat_k^{\mathrm{cat}}\colon \Cat(\ind(\Mod_{\Ok}^{\mathrm{pro \textrm{-} nil}})) \to \Cat.
	\end{equation}
We shall refer to $\Mat_k^\mathrm{cat}$ as the \emph{categorical $k$-linear materialization functor}.
\end{construction}

We now prove the following result:

\begin{proposition} \label{prop:mat of coh k-linear}
Let $A \in \adCAlg$ be a derived $\Ok$-adic algebra.
Then we have a natural equivalence of \infcats
	\[
		\Mat^\mathrm{cat}_k( \mathbf{Coh}^+(A \otimes_{\kc} k) ) \simeq \Coh(A \otimes_{\kc} k).
	\]
\end{proposition}

Before giving the proof of \cref{prop:mat of coh k-linear} we will need to introduce some few auxiliary considerations:

\begin{remark}
	Let $M\in \Mod_{\Ok}^{\mathrm{pro \textrm{-} nil}} $.
	 Observe that multiplication by $p$ induces a well defined morphism of mapping spaces
	 	\[\theta_M(p) \colon \Map_{\Mod_{\Ok}^{\mathrm{pro \textrm{-} nil}}}(N, M) \to \Map_{\Mod_{\Ok}^{\mathrm{pro \textrm{-} nil}}}(N, M)\]
	sending every morphism $f \colon M \to N$ to the composite
		\[
			N \xrightarrow{f} M \xrightarrow{p} M,
		\]
	in $\Mod_{\Ok}^{\mathrm{pro \textrm{-} nil}}$. Let $N \in \Mod_{\Ok}^{\textrm{pro-nil}}$ be arbitrary. The filtered colimit
		\[
			\colim_{\textrm{mult by }p } M \in \ind(\Mod_{\Ok}^{\mathrm{pro \textrm{-} nil}} ),
		\]
	induces an equivalence
		\begin{equation} \label{eq:map_equiv_mult_by_p_goes_outside}
			\Map_{\ind(\Mod_{\Ok}^{\mathrm{pro \textrm{-} nil}})} (N, M \otimes_{\Ok} k) \simeq \colim_{\theta_M(p)} \Map_{\Mod_{\Ok}^{\mathrm{pro \textrm{-} nil}} }(N, M),
		\end{equation}
	of mapping spaces. Indeed,
	it is a direct consequence of the fact that $N \in \Mod_{\Ok}^{\mathrm{pro \textrm{-} nil}} $ is, by definition, a compact object in $\ind(\Mod_{\Ok}^{\mathrm{pro \textrm{-} nil}} )$.
	Similarly, we have a natural equivalence
		\[
			\Map_{\ind(\Mod_{\Ok}^{\mathrm{pro \textrm{-} nil}} )}(N \otimes_{\Ok} k, M \otimes_{\Ok} k) \simeq  
			\underset{\theta_N(p)}{\lim} \colim_{\theta_M(p)}\Map_{\Mod_{\Ok}^{\mathrm{pro \textrm{-} nil}}}
			(N, M)
		\]
\end{remark}

\begin{lemma} \label{lem:mapping_spaces_enriched_in_ind_pro_are_computed_by_mult_by_p}
Let $M, \ M' \in \Mod_{\Ok}^{\mathrm{pro \textrm{-} nil}} $. Then there exists a natural morphism
	\[	
		 \Map_{\ind(\Mod_{\Ok}^{\mathrm{pro \textrm{-} nil}}) } \big(N \otimes_{\Ok} k, 
		M \otimes_{\Ok} k  \big) \to  \colim_{ \theta_{M'}(p)}
		\Map_{
		\Mod_{\Ok}^{\mathrm{pro \textrm{-} nil}} } \big(N, M \big) 
	\]
which is furthermore an equivalence of mapping spaces.
\end{lemma}

\begin{proof}
Let $M, \ N \in \Mod_{\Ok}^{\mathrm{pro \textrm{-} nil}}$, we have a chain of natural morphisms of mapping spaces
	\begin{align} \label{eq:com_limit_with_colimit_given_by_mult_by_p}
		\Map_{\ind(\Mod_{\Ok}^{\mathrm{pro \textrm{-} nil}})}(N \otimes_{\Ok} k, M \otimes_{\Ok} k ) & \simeq \underset{\theta_{M}(p)}{ \lim} \colim_{\theta_M(p) } \Map_{\Mod_{\Ok}^{\mathrm{pro \textrm{-} nil}}}(N, M) \\
										& \to  \colim_{\theta_M(p) } \Map_{\Mod_{\Ok}^{\mathrm{pro \textrm{-} nil}}}(N, M).
	\end{align}
We must show that the second natural morphism displayed in \eqref{eq:com_limit_with_colimit_given_by_mult_by_p} is an equivalence. We start by observing that the canonical morphism 
	\[
		\theta_N(p)\colon  \colim_{\theta_M(p) } \Map_{\Mod_{\Ok}^{\mathrm{pro \textrm{-} nil}}}(N, M) \to \colim_{\theta_M(p) } \Map_{\Mod_{\Ok}^{\mathrm{pro \textrm{-} nil}}}(N, M),
	\]
is an equivalence. Indeed, by $\kc$-linearity we have that $\theta_M(p) \simeq \theta_N(p) $ in the mapping space
	\[\Map_{\cS}(\Map_{\Mod_{\Ok}^{\mathrm{pro \textrm{-} nil}}}(N, M), \Map_{\Mod_{\Ok}^{\mathrm{pro \textrm{-} nil}}}(N, M)).\]
Furthermore, for every \[f  \in \Map_{\Mod_{\Ok}^{\mathrm{pro \textrm{-} nil}}}(N, M),\] the composites
	\[	
		N \xrightarrow{f} M \xrightarrow M, \quad N \xrightarrow{p} N \xrightarrow{f} M,
	\]
agree, up to contractible space of choices.
But $\theta_M(p)$ is an equivalence in the colimit
	\[
		\colim_{\theta_M(p) } \Map_{\Mod_{\Ok}^{\mathrm{pro \textrm{-} nil}}}(N, M),
	\]
by construction.
We conclude that the transition morphisms in the diagram
	\[
	 	\dots \xrightarrow{\theta_N(p)} \colim_{\theta_M(p) } \Map_{\Mod_{\Ok}^{\mathrm{pro \textrm{-} nil}}}(N, M) \xrightarrow{\theta_N(p)} \colim_{\theta_M(p) } \Map_{\Mod_{\Ok}^{\mathrm{pro \textrm{-} nil}}}(N, M),
	\]
are equivalences. We now conclude that the corresponding limit is equivalent to 
	\[\colim_{\theta_M(p)} \Map_{\Mod_{\Ok}^{\mathrm{pro \textrm{-} nil}}}(M, N) \in \cS,\]
itself, via the natural map displayed in \eqref{eq:com_limit_with_colimit_given_by_mult_by_p}, as desired.
\end{proof}

\begin{proof}[Proof of \cref{prop:mat of coh k-linear}] We have a natural equivalence $\theta \colon \Mat^{\mathrm{cat}}_{\kc}(\mathbf{Coh}^+(A)) \simeq \Coh(A)$, as in \cref{lem: pro-p nilpotent enrichement of perf}. We have a further natural functor $\mathbf{Coh}^+(A) \to \mathbf{Coh}^+(A \otimes_{\kc} k)$ induced by the the natural inclusion
	\[
		\Mod_{\Ok}^{\mathrm{pro} \textrm{-} \mathrm{nil}} \to \ind(\Mod_{\Ok}^{\pro \textrm{-} \mathrm{nil}}).
	\]
By applying the $k$-linear materialization functor we obtain a natural functor of \infcats
	\[
		\Mat^\mathrm{cat}_k(\mathbf{Coh}^+(A))   \to \Mat^\mathrm{cat}_k(\mathbf{Coh}^+(A \otimes_{\kc} k)),
	\]
(notice that $\Mat_k^\mathrm{cat}(\mathbf{Coh}^+(A)) \simeq \Mat_{\kc}^\mathrm{cat}(\mathbf{Coh}^+(A))$, by construction).
Moreover, by construction, the \infcat $\Mat^\mathrm{cat}_k(\mathbf{Coh}^+(A \otimes_{\kc} k))$ is naturally $k$-linear. Thanks to \cite[Theorem 6.3.4.6]{lurie2012higher} we obtain a well defined induced functor
	\[
		\theta \colon \Mat^\mathrm{cat}_k(\mathbf{Coh}^+(A)) \otimes_{\kc}k \to \Mat^\mathrm{cat}_k(\mathbf{Coh}^+(A \otimes_{\kc} k)).
	\]
Thanks to \cite[Corollary 3.8]{Antonio_Porta_Non_archimedean_Hilbert}, we conclude that 
	\[
		\Coh(A) \otimes_{\kc} k \simeq \Coh(A \otimes_{\kc} k).
	\]
For this reason, the analogue of \cref{lem: pro-p nilpotent enrichement of perf} for $\Coh$ implies that we are reduced to prove that $\theta$ is an equivalence of \infcats.  Fully faithfulness of $\theta$ follows from the explicit description of enriched mapping objects
provided in \cref{lem:mapping_spaces_enriched_in_ind_pro_are_computed_by_mult_by_p} together with \cite[Corollary 3.17]{Antonio_Porta_Non_archimedean_Hilbert}. In order to show that $\theta$ is essentially surjective we refer to
\cite[Theorem 3.23]{Antonio_Porta_Non_archimedean_Hilbert} which implies that every $M \in \mathbf{Coh}^+(A \otimes_{\kc} k)$ admits a formal model $\mathfrak M \in \mathbf{Coh}^+(A)$, implying essential surjectivity of $\theta$. The result is now proved.
\end{proof}

\begin{remark}
	Let $A \in \adCAlg$. It was proven in \cite[Corollary 3.7]{Antonio_Porta_Non_archimedean_Hilbert} that $\Coh(A \otimes_{\kc} k)$ is a \emph{Verdier localization}, of stable \infcats, of $\Coh(A)$ under the base change functor along $A \to A \otimes_{\kc} k$.
	It turns out that the analogous statement holds for $\Perf(A)$ and $\Perf(A \otimes_{\kc} k)$, up to take idempotent completion, see \cite[Theorem 5.1]{thomason1990higher} for a proof in the discrete setting which also applies in the case of derived coefficients. 
\end{remark}

The previous remark motivates the following definition:

\begin{definition} \label{mod}
Let $A \in \adCAlg$ be a derived $\Ok$-algebra. Denote by $\mathbf{Perf}(A \otimes_{\Ok} k)$ the
$\ind(\Mod_{\Ok}^{\mathrm{pro \textrm{-} nil}})$-enriched subcategory of $\mathbf{Coh}^+(A \otimes_{\Ok} k)$ spanned by those $M \in \mathbf{Coh}^+(A \otimes_{\Ok} k)$ such that its image under the materialization $\Mat_k^\mathrm{cat}$
is a dualizable object of the \infcat $\Coh(A \otimes_{\Ok} k)$.
\end{definition}


\begin{corollary} \label{rmk:mat_of_Perf_enriched_is_perf} Let $A \in \adCAlg$. Then one has a natural equivalence of stable \infcats
	\[
		\Mat_k^\mathrm{cat}( \mathbf{Perf}(A \otimes_{\kc} k)) \simeq \Perf(A \otimes_{\kc} k).
	\]
\end{corollary}

\begin{proof}
	Thanks to \cref{prop:mat of coh k-linear}, it follows that $\Mat_k^\mathrm{cat}(\mathbf{Perf}(A \otimes_{\kc} k)$ is a full subcategory of $\Coh(A \otimes_{\kc} k)$. By our definition of $\mathbf{Perf}_k(A \otimes_{\kc} k)$ it follows that such full subcategory
	is necessarily contained in $\Perf(A \otimes_{\kc} k)$ and it spans the latter. The result now follows, as desired.
\end{proof}

\begin{warning}
The enriched mapping objects in $\mathbf{Perf}(A \otimes_{\Ok} k)$ depends on the choice of a
lifting, i.e. if $A' \ \in \adCAlg$ is such that $A' \otimes_{\kc} k \simeq A \otimes_{\kc} k$, as derived rings, one could have, a priori, that 
	\[
		\mathbf{Perf}_k(A \otimes_{\kc} k) \quad \textrm{and} \quad  \mathbf{Perf}_k(A ' \otimes_{\kc} k)
	\]
are non-equivalent. However, one is able to prove that whenever $A$ is \emph{$n$-truncated} the above does not depend on the choice of a formal model, see \cref{non_dep}.
\end{warning}

\begin{definition}
	We denote by $\big( \adCAlg \big)^{< \infty}$ the full subcategory of $\adCAlg$ spanned by those $A \in \adCAlg$, for which there exists $m \ge 0$ sufficiently large such that
		\[
			\pi_n(A) \simeq 0,
		\]
	for every $n > m$. We will refer to objects of $\big( \adCAlg \big)^{< \infty}$ as \emph{truncated derived $\Ok$-adic algebras.}
\end{definition}

Our current goal is to prove that the functor
$\mathbf{Perf}_k \colon \adCAlg \to \Cat \big(\ind(\Mod_{\Ok}^{\mathrm{pro \textrm{-} nil}}) \big)$ does not depend on the choice of
formal model, as long as we restrict ourselves to truncated derived $\Ok$-adic algebras. We will need a few preliminary lemmas before proving this result:

\begin{lemma} \label{lem:null_obj_Ind}
	Let $\cD$ be an idempotent complete stable \infcat and let $G \in \ind(\cD)$. Express $G$ as a colimit of the form
		\[
				G \simeq \colim_{\alpha \in I} G_\alpha , \quad \textrm{with } G_\alpha \in \cD,
		\]
	indexed by a filtered \infcat, $I$. Then 
		\[
			G \simeq 0
		\]
	if and only if for every object $F \in \cD$, every index $\alpha \in I$ and every morphism $f_\alpha \colon F \to G_\alpha$ there exists an index $\beta \coloneqq \beta(F, \alpha) > \alpha$ in $I$, such that
	the composite
		\[
			F \to G_\alpha \to G_\beta,
		\]	
	is null-homotopic.
\end{lemma}

\begin{proof}
	The direct implication follows from the fact that the \infcat $\ind(\cD)$ is compactly generated by objects of $\cD$. Let us prove the converse. Assume that $G \simeq 0$ in the \infcat $\ind(\cD)$. Write $G = \colim_{\alpha \in I} G_\alpha$, where $I$ is a filtered \infcat and,
	for each $\alpha$, $G_\alpha \in \cD$. Let $F \in \cD$ be an object. 
	Then we have a chain of equivalences
		\begin{align*}
			\colim_\alpha \Map_{\cD}(F, G_\alpha) & \simeq 	\Map_{\ind(\cD)}(F, G) \\
												  & \simeq 0.
		\end{align*}
	We conclude that there exists a certain index $\alpha$ for which the structural morphism
		\[
			G_\alpha \to G,
		\]
	factors the zero morphism $F \to G$,
		\[
			F \to G_\alpha.
		\]	
	Moreover, the latter becomes null-homotopic in the colimit. We thus conclude that there exists $\beta > \alpha$ in $I$ such that the composite
		\[
			F \to G_\alpha \to G_\beta	
		\]
	is null-homotopic, as desired.
\end{proof}

\begin{lemma} \label{lem:base_change_of_mapping_spaces}
Let $A \in \adCAlg$. For each $n \ge 1$, and $M , \ N \in \Mod_A$ then the natural morphism of $A$-modules
	\[
		\theta_{M, N} \colon \underline{\Map}_{\Mod_A}(M, N) \otimes_A A_n \to \underline{\Map}_{\Mod_{A_n}}(M_n, N_n),
	\]
is an equivalence.
\end{lemma}

\begin{proof}
	The morphism $\theta_{M, N}$ is uniquely determined by the universal property of the natural morphism $A \to A_n$. Consider the fiber sequence
		\begin{equation} \label{eq:fiber_sequence_of_A_p^n_A_A_n}
			A_n[-1] \to A \xrightarrow{p^n} A ,
		\end{equation}
	in the \infcat $\Mod_{A_n}$. By tensoring \eqref{eq:fiber_sequence_of_A_p^n_A_A_n} with the $A$-module $\underline{\Map}_{\Mod_A}(M, N)$ we obtain a fiber sequence of the form
		\[
			 \underline{\Map}_{\Mod_A}(M, N) \otimes_A A_n[-1] \to  \underline{\Map}_{\Mod_A}(M, N) \xrightarrow{\theta_p} \underline{\Map}_{\Mod_A}(M, N),
		\]
	in $\Mod_A$. Moreover, the above fiber sequence is naturally equivalent to 
		\[
			\underline{\Map}_{\Mod_A}(M, N \otimes_A A_n) [-1] \to  \underline{\Map}_{\Mod_A}(M, N) \xrightarrow{\theta_p} \underline{\Map}_{\Mod_A}(M, N).
		\]
	By the universal property of base change along the morphism $A \to A_n$ we have a natural equivalence
		\[
		\underline{\Map}_{\Mod_A}(M, N_n) \simeq \underline{\Map}_{\Mod_{A_n}}(M_n, N_n).
		\]
	The conclusion now follows by shifting the fiber sequence
		\[
			\underline{\Map}_{\Mod_{A_n}}(M_n, N_n) [-1] \to  \underline{\Map}_{\Mod_A}(M, N) \xrightarrow{\theta_p} \underline{\Map}_{\Mod_A}(M, N).
		\]
\end{proof}

We are now able to prove the independency of $\mathbf{Perf}_k$ on truncated formal models:

\begin{proposition} \label{non_dep}
Let $A , \ A' \in \big( \adCAlg \big)^{< \infty}$ be truncated derived $\Ok$-adic algebras. Suppose that there exists
an equivalence
	\[
		A \otimes_{\Ok} k \simeq A' \otimes_{\Ok} k ,
	\]
in the \infcat $\CAlg_k$. Then we have an equivalence
	\[
		\mathbf{Perf}_k(A \otimes_{\Ok} k) \simeq \mathbf{Perf}_k(A' \otimes_{\Ok} k),	
	\]
in the \infcat $\Cat \big( \ind(\Mod_{\Ok}^{\pro \textrm{-} \mathrm{nil}}) \big)$.
\end{proposition}

\begin{proof}
Let $A , \ A' \in \adCAlg$ as in the statement of the proposition.
Thanks to \cite[Corollary 4.4.11]{antonio2018p}, we can suppose that there exists a morphism
$f \colon A \to A'$ in the \infcat $(\adCAlg )^{ < \infty}$ such that the rigidification of 
	\[
		\Spf (f) \colon \Spf(A') \to \Spf(A),
	\]
is an equivalence in the \infcat $\dAfd_k$. In particular, $f$ becomes an equivalence after base change along the natural morphism $\Ok \to k$. Thanks to \cref{prop:ff_es_equivalences_enriched} we are reduced to show that the canonical functor
	\begin{equation} \label{func:enr}
		(f \otimes_{\Ok} k)^* \colon \mathbf{Perf}_k(A \otimes_{\Ok} k ) \to \mathbf{Perf}_k(A' \otimes_{\Ok} k)	,
	\end{equation}
is both essentially surjective and fully faithful. We first prove that $(f \otimes_{\kc} k)^*$ is essentially surjective. Thanks to \cite[Lemma 5.3.4]{gepner2015enriched} we are reduced to verify that $(f \otimes_{\kc} k)^*$ is essentially surjective after applying the materialization functor, introduced in
\eqref{eq:mat_functor_ind_mod}. Thanks to \cref{rmk:mat_of_Perf_enriched_is_perf} we are reduced to show that the usual base change functor
	\[
		(f \otimes_{\Ok} k)^* \colon \Perf(A \otimes_{\Ok} k) \to \Perf(A' \otimes_{\Ok} k),
	\]
is essentially surjective. Since, by assumption, $(f \otimes_{\Ok} k) \colon A \otimes_{\Ok} k \to A' \otimes_{\Ok} k$ is an equivalence in $\CAlg_k$, it is clear that base change along $f \otimes_{\Ok} k$, on perfect modules, is an equivalence. In particular, it is essentially
surjective.

We are thus reduced to show that the functor displayed in \eqref{func:enr} is fully faithful. Let $M, \ N \in \mathbf{Perf}_k(A \otimes_{\Ok} k)$ be two objects. By construction, these correspond to 
	\[
		M \simeq \colim_{\textrm{mult by }p } \mathfrak M , \quad N \simeq \colim_{\textrm{mult by } p} \mathfrak N,
	\]
for suitable $\mathfrak M, \ \mathfrak N \in \mathbf{Coh}^+(A)$. Since $A$ is truncated, both $M$ and $N$ are truncated $A \otimes_{\kc} k$-modules. For this reason, up to taking truncations, we can assume that $\mathfrak M$ and $\mathfrak N$ are truncated
almost perfect $A$-modules.
We need to show that the functor displayed in \eqref{func:enr} induces an equivalence
	\begin{equation}
			\theta \colon
			\colim_{\theta(p)} \big( \underset{n \geq 1}{\lim} \big( \underline{ \Map}_{\Coh(A_n)} \big( \mathfrak M_n  , \mathfrak N_n \big) \big) \big)
			\to \colim_{\theta(p)} \big( \underset{n \geq 1}{\lim} \big( \underline{ \Map}_{\Coh(A'_n)} \big( \mathfrak M'_n  , \mathfrak N'_n \big)  \big)
			\big)
	\end{equation}
in the \infcat $\ind(\Mod_{\Ok}^{\mathrm{pro \textrm{-} nil}})$. We have denoted 
	\begin{align*}
		\mathfrak M_n \coloneqq \mathfrak M \otimes_A A_n, & \textrm{ and } \mathfrak N_n \coloneqq \mathfrak N \otimes_A A_n, \\
		\mathfrak M' \coloneqq (\mathfrak M \otimes_A A') \otimes_{A'} A'_n, & \textrm{ and } \mathfrak N' \coloneqq ( \mathfrak N \otimes_A A') \otimes_{A'} A'_n.
	\end{align*}
The \infcat $\Mod_{\Ok}^{\mathrm{pro \textrm{-} nil}}$ is a stable \infcat.
It thus suffices to prove that the morphism
	\[
		\theta ' \colon \underset{n \geq 1}{\lim} \big( \underline{ \Map}_{\Coh(A_n)} \big( \mathfrak M_n  , \mathfrak N_n \big) \big) \big)
			\to \underset{n \geq 1}{\lim} \big( \underline{\Map}_{\Coh(A'_n)} \big( \mathfrak M'_n  , \mathfrak N'_n \big)  \big),
	\]
in $\Mod_{\Ok}^{\mathrm{pro} \textrm{-} \mathrm{nil}}$, has cofiber annihilated after multiplication by a sufficiently large power of $p$.
We further observe that we have an equivalence of pro-objects
	\begin{equation} \label{eq:cofib_theta'_char_as_pro_object}
		\cofib(\theta') \simeq \underset{n \geq 1}{\lim} \cofib(\theta_n'),
	\end{equation} 
in $\Mod_{\Ok}^{\mathrm{pro} \textrm{-} \mathrm{nil}}$,
where $\theta'_n$ denotes the canonical morphism
	\begin{equation} \label{limsp}
		\theta'_n \colon \underline{\Map}_{\Coh(A_n)} \big( \mathfrak M_n  , \mathfrak N_n \big) \to 
		\underline{\Map}_{\Coh(A'_n) } \big( \mathfrak M'_n, \mathfrak N'_n \big),
	\end{equation}
in $\Mod_{\Ok}^{\mathrm{nil}}$. Indeed cofiltered limits commute with fibers in \infcats of pro-objects. Since $(f \otimes_{\kc} k)$ is an equivalence in $\CAlg_k$, it follows from \cite[Corollary 3.17]{Antonio_Porta_Non_archimedean_Hilbert} that
	\begin{equation}
		\colim \big( \Mat_{\kc}(\cofib(\theta')) \xrightarrow{p} \Mat_{\kc}(\cofib(\theta'))  \to \dots) \simeq 0,
	\end{equation}	
in $\Mod_{A}$. Since both $\mathfrak M$, $\mathfrak N$ are truncated coherent $A$-modules, it follows from \cref{lem:null_obj_Ind} that there exists a sufficiently large $n \ge 1$ such that
	\[
		p^n \cdot \pi_m(\Mat_{\kc}( \cofib(\theta') )) \simeq 0 ,
	\]
for every $m \in \bZ$.
Consequently, the object \[(p^n) \otimes_{\kc} \Mat(\cofib(\theta' ) )  \simeq 0.\]
By applying \cref{lem:base_change_of_mapping_spaces}, it follows for every $m \ge 1$, that
	\[
		(p^n) \otimes_{\kc}  \cofib(\theta'_m) \simeq 0,
	\]
in $\Mod_A$.
Therefore, once again by \cref{lem:null_obj_Ind} it follows that the ind-system
	\[
		\colim \big( \{ \cofib(\theta'_m)  \}_m \xrightarrow{p}  \{ \cofib(\theta'_m)  \}_m \xrightarrow{p} \dots \big)
	\]	
is equivalent to the zero object, in $\ind(\Mod_{\kc}^{\mathrm{pro} \textrm{-} \mathrm{nil}}).$ The assertion now follows from \eqref{eq:cofib_theta'_char_as_pro_object}.
\end{proof}

\begin{corollary} \label{invariance}
Let $A \in \big( \adCAlg \big)^{< \infty}$. Suppose we are given bounded below $M, \ N \in \mathbf{Coh}^+(A)$
such that
	\[
		M \otimes_{\Ok} k \in \mathbf{Perf}(A \otimes_{\Ok} k ) , \quad N \otimes_{\Ok} k \in \mathbf{Perf}(A \otimes_{\Ok } k ).
	\]
Then the
mapping object
	\[
		\mathbf{Map}_{\mathbf{Perf}(A)} \big( M \otimes_{\Ok} k , N \otimes_{\Ok} k \big) \in  \ind(\Mod_{\Ok}^{\mathrm{pro \textrm{-} nil}})
	\]
does not depend on the choice of formal models $M$ and $N  \in \mathbf{Coh}^+(A)$.
\end{corollary}

\begin{proof}
It suffices to prove the assertion separately in the variable $M$ and $N$. Suppose we are given $M' \in \mathbf{Coh}^+(A)$ such that 
	\[
		M' \otimes_{\kc} k \simeq M \otimes_{\kc} k,
	\]
Thanks to \cite[Proposition A.2.1]{antonio2018p} we can reduce to the case where there exists a morphism $M \to M' $ in $\mathbf{Coh}^+(A)$ such that
	$
		(f \otimes_{\kc} k)
	$
is an equivalence in the \infcat $\mathbf{Perf}_k(A \otimes_{\kc} k)$. In this case, the reasoning used to deal with fully faithfulness of $(f \otimes_{\kc} k)^*$, in the proof of \cref{non_dep}, can be adapted to the current context. To prove independency on the choice of the formal model $N
\in \mathbf{Coh}^+(A)$, we apply the same argument.
\end{proof}

\begin{remark}
The association $A \in \adCAlg \mapsto \mathbf{Perf}(A \otimes_{\Ok} k) \in \Cat \big( \ind(\Mod_{\Ok}^{\mathrm{pro} \textrm{-} \mathrm{nil}})
\big)$ is functorial.
Therefore, the usual functor
	\[
		\Perf \colon \adCAlg \to \Cat^\st
	\]
can be upgraded naturally to a functor $\mathbf{Perf}_k \colon \adCAlg \to  \Cat ( \ind(\Mod_{\Ok}^{\mathrm{pro \textrm{-} nil}}))$.
\end{remark}

We now study some properties of the functor $\mathbf{Perf}_k \colon \adCAlg \to \Cat$:

\begin{proposition} \label{inf:cart}
The functor $\mathbf{Perf}_k \colon \adCAlg \to \Cat$ is infinitesimally cartesian.
\end{proposition}

\begin{proof}
Let $ A \in \adCAlg$. Consider an $\Ok$-adic derivation
	\[
		d \colon \bL^\ad_A \to M ,
	\]
in the \infcat $\Coh(A)$. By the universal property of the adic cotangent complex, cf. \cite[Proposition 3.4.3]{antonio2018p}, this corresponds to a morphism $A \to A \oplus M$ in the \infcat $\adCAlg$. Consider the pullback diagram
	\[
	\begin{tikzcd}
		A_d[M] \ar{r}  \ar{d}  & A \ar{d}{d} \\
		A \ar{r}{d_0} & A \oplus M
	\end{tikzcd}
	\]
in $\adCAlg$. Thanks to \cref{prop:ff_es_equivalences_enriched} we are required to prove that
the induced functor
	\begin{equation} \label{funfun}
		F_{A, d} \colon \mathbf{Perf}_k \big( A_d[M] \otimes_{\Ok} k \big)  \to \mathbf{Perf}_k (A \otimes_{\Ok} k) \times_{	\mathbf{Perf}_k((A \oplus M ) 	\otimes_{\Ok} k )	}	\mathbf{Perf}_k(A \otimes_{\Ok} k),
	\end{equation}
is both essentially surjective and fully faithful. 
If follows by \cite[Lemma 5.3.4 combined with Lemma 5.1.2]{gepner2015enriched}, that we can check essential surjectiveness after applying the categorical $k$-linear materialization
functor 
	\[
		\Mat_k^{\mathrm{cat}}\colon \Cat  \big( \ind(\Mod_{\Ok}^{\mathrm{pro} \textrm{-} \mathrm{nil}}) \big) \to \Cat,
	\]
defined in \eqref{eq:mat_functor_ind_mod}.
Furthermore, after applying $\Mat_{\mathrm{cat}}$ the functor displayed in \eqref{funfun} is equivalent to the canonical functor
	\[
		F_{A, d} \colon \Perf(A_d[M] \otimes_{\Ok} k ) \to \Perf(A \otimes_{\Ok} k) \times_{	\Perf((A \oplus M) \otimes_{\Ok} k )	} \Perf(A \otimes_{\Ok} k ).
	\]
The essential surjectivity of the latter follows by \cite[Proposition 3.4.10]{lurie2012dag}. We are thus reduced to show that the functor $F_{A, d}$ is fully faithful.
Let $N, P \in \mathbf{Perf}_k(A_d[M] \otimes_{\Ok} k)$. Thanks to \cite[Proposition A.3.2]{antonio2018p} there are $\mathfrak N, \mathfrak P \in \mathbf{Coh}^+(A_d[M])$ and equivalences
	\[
		\mathfrak N \otimes_{\Ok} k \simeq N , \quad \mathfrak P \otimes_{\Ok} k \simeq P,
	\]
in $\mathbf{Perf}_k(A_d[M])$. We are thus required to prove that the canonical morphism
	\begin{equation} \label{eq:ff_of_inf_cart}
		\colim_{\theta(p)} \mathbf{Map}_{\mathbf{Coh}^+(A_d[M])}(\mathfrak N , \mathfrak P)   \to  X_0 \times_{X_{0, 1} } X_1
	\end{equation}
is an equivalence in $\ind( \Mod_{\Ok}^{\mathrm{pro} \textrm{-} \mathrm{nil}})$.
Here we have denoted
	\begin{align*}
		 X_0  & \coloneqq \colim_{\theta(p)} \big( \mathbf{Map}_{\mathbf{Coh}^+(A)}(\mathfrak N \otimes_{A_d[M]} A, \mathfrak P \otimes_{A_d[M]} A) \big)  \\
		X_{0, 1} & \coloneqq	\colim_{\theta(p)} \big( \mathbf{Map}_{\mathbf{Coh}^+(A \oplus M)}(\mathfrak N  \otimes_{A_d[M]} (A \oplus M), \mathfrak P \otimes_{A_d[M]} (A \oplus M) \big)  \\
		 X_1 & \coloneqq  \colim_{\theta(p)} \big( \mathbf{Map}_{\mathbf{Coh}^+(A)}(\mathfrak N  \otimes_{A_d[M]} A, \mathfrak P \otimes_{A_d[M]} A)  \big).
	\end{align*}
We now observe that the fact that filtered colimits commute with finite limits in $\ind$-complete \infcats combined with \cref{lim:tech} imply that the morphism displayed in \eqref{eq:ff_of_inf_cart} is an equivalence, as desired.
\end{proof}

We now proceed to define an $\ind(\pro(\cS))$-enrichment on $\Perf(A \otimes_{\kc} k)$. This is achieved as follows: 

\begin{lemma} \label{lem:omega_infty_is_lax_sym_mon}
The natural inclusion functor
	\[
		i_\mathrm{nil} \colon \Mod_{\kc}^{\mathrm{nil}} \to \Mod_{\kc},
	\]
is lax symmetric monoidal.
\end{lemma}

\begin{proof} Let $\kc^\mathrm{nil} \in \Mod_{\kc}^{\mathrm{nil}}$ denote the unit for the corresponding symmetric monoidal structure. We have a natural morphism $\kc^\mathrm{nil} \to \kc$, by construction. Moreover, given any $M, N \in \Mod^{\mathrm{nil}}$ we 
have natural maps
	\begin{align*}
		M \otimes_{\kc^\mathrm{nil}} N & \simeq  M \otimes_{\kc^\mathrm{nil}} (\kc^\mathrm{nil} \otimes_{\kc} N ) \\
								& \simeq (M \otimes_{\kc^\mathrm{nil}} \kc^\mathrm{nil}) \otimes_{\kc} N  \\
								& \simeq M \otimes_{\kc} N,
	\end{align*}
	see \cite[Definition 8.2.2.1 combined with Proposition 8.2.2.5]{lurie2016spectral} for a justification of the first equivalence. The result now follows.
\end{proof}

\begin{construction} \label{const:Omega_cat^infty_on_ind_pro}
Consider the lax symmetric monoidal functor
	\[
		\Omega_{\kc}^\infty \colon \Mod_{\kc}^\otimes \to \cS^\times,
	\]
obtained by precomposing the usual forgetful $\Mod_{\kc} \to \Sp$ with the lax symmetric monoidal functor
	\[
		\Omega^\infty \colon \Sp^\otimes \to \cS^\times.
	\]
By \cref{lem:omega_infty_is_lax_sym_mon}, we have a lax symmetric monoidal inclusion functor
	\[
	\Mod_{\kc}^{\mathrm{nil}} \subset \Mod_{\kc}.
	\] 
Thus, by composition, we obtain a natural lax symmetric monoidal functor
	\[
		\Omega^{\infty, \mathrm{nil}}_{\kc} \colon \Mod_{\kc}^{\mathrm{nil}, \otimes} \to \cS^\times.
	\]
By passing to ind-pro completions we obtain a canonically defined lax symmetric monoidal functor
	\begin{equation} \label{eq:omega_infty_ind}
		\Omega^{\infty, \ind}_{\kc} \colon \ind(\Mod_{\Ok}^{\mathrm{pro \textrm{-} nil}})^\otimes \to
		\ind \big( \pro( \cS ) \big)^\times.
	\end{equation}
\end{construction}

\begin{lemma} \label{lem:cat materialization functor on ind pro spaces}
The functor $\Omega^{\infty, \ind}_{\kc} \colon \ind(\Mod_{\Ok}^{\mathrm{pro \textrm{-} nil}})^\otimes \to
		\ind \big( \pro( \cS ) \big)^\times$ induces a well defined functor
	\[
		\Omega^\infty_{\mathrm{cat}} \colon \Cat( \ind(\Mod_{\Ok}^{\mathrm{pro \textrm{-} nil}}) ) \to \Cat(\ind(\pro(\cS))).
	\]	
\end{lemma}

\begin{proof}
	It follows from \cref{const:Omega_cat^infty_on_ind_pro} combined with \cite[Corollary 5.7.6]{gepner2015enriched}.
\end{proof}

\begin{notation}
 Let $A \in \adCAlg$. By abuse of notation we will denote $\Omega^\infty_{\mathrm{cat}} (\mathbf{Perf}_k(A \otimes_{\kc} k))$ simply by $\mathbf{Perf}_k(A \otimes_{\kc} k) \in \Cat(\ind(\pro(\cS)))$.
\end{notation}

\begin{definition}
	Let $f \colon A \to B$ be a morphism in $\adCAlg$. We say that $f$ is a \emph{rig-equivalence} if $f \otimes_{\Ok} k$ is an equivalence in the \infcat $\CAlg_k$.
\end{definition}

\begin{corollary}
The functor $\mathbf{Perf}_k \colon \adCAlg \to \Cat(\ind(\pro(\cS)))$ is infinitesimally cartesian. Moreover, its restriction to the full subcategory $( \adCAlg )^{< \infty} \subseteq \adCAlg$ is
invariant under rig-equivalences.
\end{corollary}

\begin{proof} The first statement follows by the precise same proof as in \cref{inf:cart}.
The second assertion follows from \cref{non_dep}.
\end{proof}

\begin{notation}
Consider the \infcat $\dAfd_k$ of derived $k$-affinoid spaces introduced in \cite[Definition 7.3]{porta2016derived}. We denote $\dAfd^{< \infty}_k$
the full subcategory spanned by those truncated derived $k$-affinoid spaces.
\end{notation}

\begin{remark}
The rigidification functor $\rigg \colon \adCAlg \to \dAfd^{\op}_k$ introduced in \cite[\S 4]{antonio2018p} induces, by restriction, a well defined functor
	\[
		\rigg \colon  \big( \adCAlg \big)^{< \infty} \to \big( \dAfd_k^{< \infty} \big)^\op.
	\]
Moreover, the derived Raynaud localization theorem \cite[Theorem 4.4.10]{antonio2018p}
and its proof imply that $\dAfd_k^{< \infty}$ is a localization of the \infcat
$( \adCAlg )^{< \infty}$ at the saturated class of \emph{rig-equivalences} morphisms. We characterize those as morphisms $A \to A'$ which become an equivalence after base change along $\kc \to k$. This
assertion follows by
derived Tate aciclicity theorem, cf. \cite[Theorem 3.1]{Porta_Yue_derived_Hom}.
\end{remark}

The functor $\mathbf{Perf}_k \colon (\adCAlg)^{< \infty} \to \Cat(\ind(\pro(\cS)))$ actually descend to the \infcat $(\dAfd)^{< \infty}$:

\begin{proposition} \label{label}
Let $S^{< \infty}$ denote the saturated class of rig-strong morphisms in $(\adCAlg)^{< \infty}
$. Then the functor
	\[
		\mathbf{Perf}_k \colon \adCAlg \to  \Cat(\ind(\pro(\cS)))
	\]
sends morphisms in $S^{< \infty}$ to equivalences of \infcats in $ \Cat(\ind(\pro(\cS)))$. In particular,
one has a canonical induced functor
	\[
		\mathbf{Perf} \colon \big( \dAfd^{< \infty}_k \big)^\op \to  \Cat(\ind(\pro(\cS))).
	\]
\end{proposition}

\begin{proof}
The first part of the statement follows from \cref{invariance}. The second part of the statement
follows from the derived Raynaud localization theorem, \cite[Corollary 4.4.12]{antonio2018p}.
\end{proof}

\section{Moduli of derived continuous $p$-adic representations}

\subsection{Construction of the functor}
Consider the \infcat of ind-pro objects on $\cS$, $\ind(\pro(\cS))$. We equip it with its natural Cartesian symmetric monoidal structure, $\ind(\pro(\cS))^\times$. Consider the \infcat of $\ind(\pro(\cS))^\times$-enriched \infcats, $\Cat(\ind(\pro(\cS))$. 

\begin{construction} \label{const:materialization_functor_ind-pro_to_S}
	We have a natural symmetric monoidal materialization functor
		\[
			\Mat \colon \ind(\pro(\cS))^\times \to \cS^\times,
		\]
	given by the formula
		\[
			X \in \ind(\pro(\cS)) \mapsto \Map_{\ind(\pro(\cS))}(*, X).
		\]
	The fact that $\Mat$ is symmetric monoidal comes from the fact that filtered colimits commute with finite products. By \cite[Corollary 5.7.6]{gepner2015enriched}, we obtain an induced \emph{categorical materialization functor}
		\[
			\Mat^\mathrm{cat} \colon \Cat(\ind(\pro(\cS))) \to \Cat.
		\]
\end{construction}

\begin{proposition} \label{prop:cat enriched admits an infty-2 cat structure}
	The \infcat $\Cat(\ind(\pro(\cS))$ admits a natural structure of an $(\infty, 2)$-category.
\end{proposition}

\begin{proof}
	This is the content of \cite[Example 7.4.11]{gepner2015enriched}. In fact, the latter states a stronger statement, namely that $\Cat(\ind(\pro(\cS))$ admits a \emph{$\ind(\pro(\cS))^\times$-$(\infty,2)$-structure}. Consequently, we obtain a natural $(\infty, 2)$-categorical structure on $\Cat(\ind(\pro(\cS))$ via the materialization functor, introduced
	in \cref{const:materialization_functor_ind-pro_to_S}.
\end{proof}

\begin{definition}
Let $\cC, \ \cD \in \Cat(\ind(\pro(\cS))$. We denote by 
	\begin{align*}
		\cC \mathrm{ont} \Fun \big(\cC, \cD \big) & \coloneqq \underline{\Fun} \big( \cC, \cD) \\
									      & \in \Cat,
	\end{align*}
the \infcat of \emph{continuous functors from $\cC$ to $\cD$}, where $\underline{\Fun}$ denotes the functor \infcat obtained via \cref{prop:cat enriched admits an infty-2 cat structure}
\end{definition}

\begin{notation}
	Let $\cS^\fc$ denote the \infcat of finite spaces, see \cite[Definition 2.4.1]{2009derived}. We denote by $\pro(\cS^\fc)$ the \infcat of \emph{profinite spaces}.
\end{notation}

\begin{definition}	
Let $X \in \pro(\cS^\fc)$ be a profinite space. We say that $X$ is \emph{connected} if
	\[
		\pi_0(\Map_{\pro(\cS^\fc)}(*, X))  \simeq *.
	\]
\end{definition}

\begin{remark} \label{rem:ff embeddings of profinite spaces in ind-pro spaces} 
Let $X \in \pro(\cS^\fc)$ denote a connected profinite space. We can consider the $\bE_1$-monoid like object
		\begin{equation} \label{eq:X seen as a monoid like object}
		\begin{tikzcd}
			 \dots \ar[r, shift left=-1.5ex] \ar[r, shift left=-0.5ex] \ar[r, shift left=0.5ex] \ar[r, shift left=1.5ex] & \Omega(X)^{\times 2} \ar[r, shift left=1ex] \ar{r}  \ar[r, shift left=-1ex] &  \Omega(X) \ar[r, shift left=0.5ex] \ar[r, shift 	
			left=-0.5ex] &* \in \Mon_{\bE_1}(\ind(\pro(\cS))),
		\end{tikzcd}
		\end{equation}
where we apply the loop functor $\Omega \colon \cS \to \cS$ component-wise.
Thanks to \cref{prop:monoids embed ff in enriched infcats} it follows that we can consider naturally $X \in \Cat(\ind(\pro(\cS))$, via the diagram displayed in \eqref{eq:X seen as a monoid like object}.
\end{remark}

\begin{notation}
	We shall abusively use the symbol $X$ to denote the diagram displayed in \eqref{eq:X seen as a monoid like object}, considered naturally as an object in $\Cat(\ind(\pro(\cS)))$.
\end{notation}
\begin{definition}
	Let $X \in \pro(\cS^\fc)$ denote a profinite space. Given any $A \in \adCAlg$, we denote by 
		\begin{align*}
			\bPerfSys(X)(A \otimes_{\kc} k) & \coloneqq \cC\mathrm{ont} \Fun(X, \mathbf{Perf}(A \otimes_{\kc} k)) \\
				& \in \Cat,
		\end{align*}
	the \infcat of \emph{continuous $A \otimes_{\kc} k$-adic representations of $X$}. Similarly, we shall denote by
		\begin{align*}
			\bPerfSys(X)(A) & \coloneqq \cC \mathrm{ontFun}(X, \mathbf{Perf}_{\kc}(A)) \\
						& \in \Cat,
		\end{align*}
	the \infcat of \emph{continuous $A$-adic representations}.
\end{definition}

\begin{remark}
	Notice that the proof of \cref{prop:cat enriched admits an infty-2 cat structure}, or more precisely \cite[Example 7.4.11]{gepner2015enriched}, implies that the \infcat $\bPerfSys(X)(A \otimes_{\kc} k)$ is enriched over $\Mod_{A \otimes_{\kc } k}$. Similarly, $\bPerfSys(X)(A ) $ is naturally enriched over $\Mod_{A}$.
\end{remark}

\begin{definition}
Let $X \in \pro(\cS^\fc)$ and $A \in \adCAlg$. We refer to the \infcat $\bPerfSys(X)(A)$
as the \infcat of derived \emph{$A \otimes_{\Ok} k$-adic continuous representations of $X$}.
\end{definition}

When $X \simeq *$ we have the following expectable statement:

\begin{lemma}
Let $A \in \adCAlg$. Then there exists a natural equivalence \[\bPerfSys(*)(A) \simeq \Perf(A \otimes_{\Ok} k)\] of stable \infcats. 
\end{lemma}

\begin{proof}
	It suffices to prove that $\bPerfSys(*)(A \otimes_{\kc} k) $ is equivalent to the materialization of $\mathbf{Perf}_k(A \otimes_{\kc} k)$. This follows from the fact that $* \in \ind(\pro(\cS))^\times$ is the unit object combined with \cref{prop:cat enriched admits an infty-2 cat structure}.
\end{proof}

\begin{notation}
Let $X \in \pro(\cS^\fc)$ be a connected profinite space.
Denote by $\pi_X \colon * \to X$ the uniquely defined, up to a contractible space of indeterminacy, morphism in $\pro(\cS^\fc)$. We shall denote by 
	\begin{align*}
		 \pi_X^* \colon \bPerfSys(A \otimes_{\kc} k) & \to \bPerfSys(*)(A \otimes_{\kc} k) \\
		 							      & \simeq \Perf(A \otimes_{\kc } k),
	\end{align*}
the corresponding composite functor in $\Cat$.
\end{notation}

\begin{notation}
Let $A \in \adCAlg$. Let $\mathfrak M \in \mathbf{Coh}^+(A)$ and $ M \in \mathbf{Perf}_k(A \otimes_{\kc} k)$. We shall denote by 
	\[
		\mathbf{End}(M) \coloneqq \mathbf{Map}_{\mathbf{Perf}_k(A \otimes_{\kc} k)}(M, M) \quad \textrm{and} \quad \cEnd(\mathfrak M ) \coloneqq \mathbf{Map}_{\mathbf{Coh}^+(A)}(\mathfrak M , \mathfrak M ) 
	\]
in $\ind(\pro(\cS))$. Moreover, composition in $\mathbf{Perf}_k(A \otimes_{\kc} k)$ induces a naturally defined $\bE_1$-monoid like structure on $\mathbf{End}(M)$. We shall also denote by 
	\[
	\begin{tikzcd}
		\rmB \mathbf{End}(M)  \coloneqq  \big( \dots \ar[r, shift left=-1.5ex] \ar[r, shift left=-0.5ex] \ar[r, shift left=0.5ex] \ar[r, shift left=1.5ex] & \cEnd(M)^{\times 2} \ar[r, shift left=1ex] \ar{r}  \ar[r, shift left=-1ex] &  \cEnd(M) \ar[r, shift left=0.5ex] \ar[r, shift left=-0.5ex] &*  \big) \in \Mon_{\bE_1}(\ind(\pro(\cS)))
	\end{tikzcd}
	\]
the corresponding $\bE_1$-monoid like object in $\ind(\pro(\cS))^\times$. Similarly, the object $\cEnd(\mathfrak M)$ can be naturally upgraded to a well defined object
	\[
		\rmB \cEnd(\mathfrak M ) \in \Mon_{\bE_1}(\ind(\pro(\cS))).
	\]
\end{notation}



\begin{remark} Let $A \in \adCAlg$ and $\mathfrak M \in \mathbf{Coh}^+(A)$. Then multiplication by $p$
	\[
		p \colon \mathbf{End}(\mathfrak M ) \to \mathbf{End}(\mathfrak M),
	\]
	is not a morphism of $\bE_1$-monoid like objects. Nonetheless, whenever $\mathfrak M $ is a formal model for $M \in \mathbf{Perf}_k(A \otimes_{\kc} k)$, the object 
		\[ \cEnd(M) \simeq \mathbf{End}( \mathfrak M) \otimes_{\kc} k, \quad (\mathrm{\cref{lem:mapping_spaces_enriched_in_ind_pro_are_computed_by_mult_by_p}})\]
	does admit a $\bE_1$-monoid like structure in the
	symmetric monoidal Cartesian \infcat $\ind(\pro(\cS))^\times$.
\end{remark}

\begin{notation} Let $X \in \pro(\cS^\fc)$ be a connected profinite space and $A \in \adCAlg$.
	Let $M \in \mathbf{Perf}_k(A \otimes_{\kc} k)$, we shall often denote abusively
		\[
			\Map_{\Mon_{\bE_1}(\ind(\pro(\cS)))}(\Omega X, \cEnd(M) ) \coloneqq \Map_{\Mon_{\bE_1}(\ind(\pro(\cS)))}( X,\rmB \cEnd(M) ) 
		\]
\end{notation}

\begin{proposition}  \label{par} Let $X \in \pro(\cS^\fc)$ be a connected profinite space.
Let $A \in  \adCAlg$ and $M \in \mathbf{Perf}_k(A \otimes_{\kc} k)$.  The fiber of the functor
	\[
		\pi^* \colon \bPerfSys(X)(A \otimes_{\kc} k) \to \Perf(A \otimes_{\kc} k),
	\]
is canonically equivalent to the mapping space
	\[
		\Map_{\Mon_{\bE_1}(\ind(\pro(\cS)))} \big( \Omega X, \mathbf{End}(M) \big) \in \cS.
	\]
\end{proposition}

\begin{proof} Consider the Bar-construction
	\[
		\rmB \cEnd(M) \in \Mon_{\bE_1}(\ind(\pro(\cS))),
	\]	
naturally considered as an object in $\Cat(\ind(\pro(\cS))$, via \cref{prop:monoids embed ff in enriched infcats}. We thus have a natural fully faithful functor of enriched \infcats
	\[
		F_M \colon \rmB \cEnd(M) \to \mathbf{Perf}_k(A \otimes_{\kc} k).
	\]
The fiber
of \[\pi^* \colon \bPerfSys(X) (A \otimes_{\kc} k) \to \Perf(A \otimes_{\kc} k)\]
over $M \in \mathbf{Perf}_k(A \otimes_{\kc} k)$,
is then naturally equivalent to the functor \infcat
	\begin{equation} \label{eqqq}
		\cC \mathrm{ontFun}(X, \rmB \cEnd(M)) \in \Cat.
	\end{equation}
Since both $X$ and $\rmB \cEnd(M)$ lie in the full subcategory of $\bE_1$-monoid like objects
	\[
		\Mon_{\bE_1}(\ind(\pro(\cS))) \subseteq \Cat(\ind(\pro(\cS))),
	\]
it follows that we have a natural equivalence
	\[
		\cC \mathrm{ontFun}(X, \rmB \cEnd(M))  \simeq \Map_{\Mon_{\bE_1}(\ind(\pro(\cS)))}(X, \rmB \cEnd(M)),
	\]
as desired.
\end{proof}

\begin{corollary}
	The functor $\pi_X^* \colon \bPerfSys(A \otimes_{\kc} k) \to \Perf(A \otimes_{\kc} k)$ is a coCartesian fibration associated to the functor
		\[
			F \colon \Perf(A \otimes_{\kc } k) \to \Cat,
		\]
	which sends every perfect $A \otimes_{\kc} k$-module to the space
		\[
			M \in \Perf(A \otimes_{\kc} k) \mapsto \Map_{\Mon_{\bE_1}(\ind(\pro(\cS)))}(X, \cEnd(M)) \in \cS,
		\]
	of \emph{continuous $(A \otimes_{\kc} k)$-adic representations of $X$}.
\end{corollary}

\begin{proof}
	It is an immediate consequence of \cref{par}.
\end{proof}

\begin{corollary} Let $A \in (\adCAlg)^{< \infty}$.
	The functor
		\[
			\pi_X^* \colon \bPerfSys(X)(A \otimes_{\kc} k) \to \Perf(A \otimes_{\kc} k),
		\]
	is conservative.
\end{corollary}

\begin{proof}
	Let $f \colon \rho \to \rho' $ in $\bPerfSys(X)(A \otimes_{\kc} k)$ such that
		\[
			\pi_X^*(f) \colon M \to M',
		\]
	is an equivalence in $\Perf(A \otimes_{\kc}k)$. Since $A$ is truncated it follows that both perfect $(A \otimes_{\kc} k)$-modules $M $ and $M'$ are truncated as well. For this reason,
	\cref{invariance} and its proof imply that $\pi_X^*(f)$ induces an equivalence 
		\[
			\widetilde{f} \colon \cEnd(M) \simeq \cEnd(M'),
		\]
	in $\ind(\pro(\cS))$. Since the forgetful functor
		\[
			\mathrm{forget} \colon \Mon_{\bE_1}(\ind(\pro(\cS))) \to \ind(\pro(\cS)),
		\]
	is conservative, we conclude that $\widetilde{f}$ induces an equivalence of $\bE_1$-monoid like objects
		\[
			\cEnd(M) \to \cEnd(M').
		\]
	Let $g \colon \cEnd(M') \to \cEnd(M)$ be an inverse to $\widetilde{f}$ in $\Mon_{\bE_1}(\ind(\pro(\cS)))$. Then $g$ necessarily fits into a commutative diagram of the form
		\[
		\begin{tikzcd}
			X \ar{r}{\overline{\rho}}  \ar{rd}{\overline{\rho}'}& \rmB \cEnd( \mathfrak M') \ar{d}{g} \\
			 & \rmB \cEnd(\mathfrak M),
		\end{tikzcd}
		\]
	in $\Mon_{\bE_1}(\ind(\pro(\cS)))$. This provides an inverse to $f  \colon \rho \to \rho'$.
\end{proof}

\begin{definition}
We define the $\Cat$-valued functor of \emph{continuous $k$-adic representations of $X$} as the functor
	\[
		\bPerfSys(X)(-) \colon \big( \adCAlg \big)^{< \infty} \to \Cat,
	\]
given on objects by the formula
	\[
		A \in  \big( \adCAlg \big)^{< \infty} \mapsto \bPerfSys(X) (A \otimes_{\kc} k) \in \Cat.
	\]
\end{definition}

An important consequence of \cref{label} is the following result:
\begin{proposition}
The functor $\bPerfSys(X) \colon \big( \adCAlg \big)^{< \infty} \to \Cat$ descends to a well defined functor
	\[
		\bPerfSys(X) \colon \big( \dAfd_k^{< \infty} \big)^{\op} \to \Cat,
	\]
which is given on objects by the formula
	\[
		Z \in \dAfd_k^{< \infty} \mapsto \cC \mathrm{ont} \Fun \big( X, \mathbf{Perf}(\rmR \Gamma (Z, \cO_Z)) \big) \in \Cat.
	\]
\end{proposition}

\begin{proof}
The result is a direct consequence of the analogous statement for $\mathbf{Perf}$ which is the content of \cref{label}.
\end{proof}

\subsection{Truncated enriched objects}

In this \S, we will develop some framework to deal with Postnikov towers in the \infcat $\ind(\pro(\cS))$. Consider the usual $n$-th truncation functor
	\[
		\tau_{\le n} \colon \cS \to \cS.
	\]
We obtain a canonical extension by cofiltered limits
	\[
		\tau_{\le n} \colon \pro(\cS) \to \pro(\cS).
	\]
Given $\{ X_i \}_i \in \pro(\cS)$ we have that 
	\[
		\tau_{\le n}(\{ X_i \}_i ) \simeq \{ \tau_{\le n} X_i \}_i \in \pro(\cS).
	\]
We have the following result:

\begin{lemma}
	Let $\{ X_i \}_{i \in I } \in \pro(\cS)$ be a pro-space such that for every $i \in I$, we have that
		\[
			X_i \in \cS_{\le n},
		\]
	is a $n$-truncated space. Then for every $\{Y_j \}_j \in \pro(\cS)$ there exists a natural equivalence
		\[
			\Map_{\pro(\cS)}(\{Y_j \}_{j } , \{ X_i \}_i ) \simeq \Map_{\pro(\cS)}( \{ \tau_{\le n} Y_j \}_j , \{ X_i \}_i ),
		\]
	of mapping spaces.
\end{lemma}

\begin{proof}
	It follows by the  explicit description of mapping spaces in $\pro(\cS)$ as limit-colimit combined with the universal property of $\tau_{\le n} \colon \cS \to \cS$.
\end{proof}

We will also need to introduce a relative version:

\begin{definition} \label{def:truncated morphisms of pro objects}
Let $f \colon X \to Y $ be a morphism in $\pro(\cS)$ be two pro-spaces indexed by the same diagram $I^\op$, where $I$ is a filtered \infcat. Let $n \ge 1$, we define the \emph{pointwise relative $n$-th truncation of $f$} as
	\[
		\tau_{\le n}(f) \colon X^n \to Y,
	\]
where for each $i \in I$, $\tau_{\le n}(f)_i \colon X^n_i \to Y_i$ coincides with the relative $n$-th truncation of the morphism $f_i \colon X_i \to Y_i,$ \cite[Definition 5.5.6.8 and Proposition 5.5.6.18]{lurie2009higher}.
\end{definition}

\begin{lemma} \label{lem:tau_n of fib of maps in pro}
	Let $f \colon X \to Y$ be a morphism satisfying the conditions of \cref{def:truncated morphisms of pro objects}. Then the fiber of 
		\[
			\fib(\tau_{\le n}(f)) \in \pro(\cS),
		\]
	is naturally equivalent to $\tau_{\le n}(\fib(f))$.
\end{lemma}

\begin{proof}
	Since fibers commute with cofiltered limits we have that
		\begin{align*}
			\fib( \tau_{\le n}(f))  & \simeq \fib( \lim_{i \in I^\op} \tau_{\le n}(f )_i) \\
						      & \simeq \lim_{ i \in I^\op} \fib(\tau_{\le n}(f)_i) \\
						      & \simeq \lim_{i \in I^\op} \tau_{\le n}(\fib(f_i)) \\
						      & \simeq \tau_{\le n}(\fib(f)),
		\end{align*}
	as desired.
\end{proof}

Similarly, by extending via filtered colimits we obtain a canonical functor $\tau_{\le n} \colon \ind(\pro(\cS)) \to \ind(\pro(\cS)).$ In this case, we have the analogue:

\begin{lemma} \label{lem:universal property of truncation}
	Let $X \in \ind(\pro(\cS))$ such that 	
		\[
			X \simeq \tau_{\le n} X,
		\]
	in $\ind(\pro(\cS))$. Then for every $ Y \in \ind(\pro(\cS))$ we have that
		\[
			\Map_{\ind(\pro(\cS))}(Y, X) \simeq \Map_{\ind(\pro(\cS))}( \tau_{\le n}(Y), X),
		\]
	in $\cS$.
\end{lemma}

\begin{proof}
	The proof follows as in the case of \cref{lem:tau_n of fib of maps in pro}
\end{proof}

We now specialize the above considerations to our setting: let $A \in \adCAlg$ denote a truncated derived $\kc$-adic algebra.

\begin{lemma} \label{lem:monoidal structure on truncations}
	Let $M \in \mathbf{Perf} (A \otimes_{\kc} k)$. Then for every for every formal model $ \mathfrak M \in \mathbf{Coh}^+(A)$ and $n \ge 1$, both objects
		\[
			\tau_{\le n} (\cEnd(M))  \quad \textrm{and    } \ \ \ \ \tau_{\le n}(\cEnd(\mathfrak M)),
		\]
	admit natural $\bE_1$-monoid-like structures induced by the ones on $\cEnd(M)$ and $\cEnd(\mathfrak M)$, respectively.
\end{lemma}

\begin{proof}
	We have an equivalence
	\[\tau_{\le n}(\cEnd(M)) \simeq \colim_{\mathrm{mult} \ p} \tau_{\le n}(\cEnd(\mathfrak M))\]
	in $\ind(\pro(\cS))$. Moreover, the functor 
			\[
			\mathrm{forget} \colon \Mon_{\bE_1}(\ind(\pro(\cS))) \to \ind(\pro(\cS))
		\]
	is conservative and preserves filtered colimits. Since the cartesian monoidal structure on $\ind(\pro(\cS))$ is stable under filtered colimits it follows that if $\tau_{\le n}(\cEnd( \mathfrak M))$ admits a natural $\bE_1$-monoid like structure, then so does the filtered
	colimit $\tau_{\le n}(\cEnd(M))$, in $\ind(\pro(\cS))$. For this reason, we are
 	reduced to prove the statement of the lemma for $\tau_{\le n}(\cEnd(\mathfrak M))$.
	The latter folllows from the fact that the transition maps
		\[
			\dots \to \cEnd(M \otimes_{\kc} \cO_{k, n}) \to \cEnd( \mathfrak M \otimes_{\kc} (\kc)_{n-1} ) \to \dots \to \cEnd(\mathfrak M \otimes_{\kc} (\kc)_{2}) \to \cEnd(\mathfrak M \otimes_{\kc} (\kc)_1)
		\]
	in $\pro(\cS)$, are morphisms of $\bE_1$-like objects. Indeed, the above diagram can be regarded naturally as a diagram in the full subcategory $\Mon_{\bE_1}(\cS) \subseteq \Mon_{\bE_1}(\pro(\cS))$. It follows that after applying the usual truncation functor 
		\[\tau_{\le n} \colon \cS \to \cS,\]
	we obtain a diagram in
		\[
			\dots \to \tau_{\le n} \cEnd(M \otimes_{\kc} \cO_{k, n}) \to \tau_{\le n} \cEnd( \mathfrak M \otimes_{\kc} \cO_{k, n-1} ) \to \dots \to \tau_{\le n} \cEnd(\mathfrak M \otimes_{\kc} \cO_{k, 2}) \to \tau_{\le n} \cEnd(\mathfrak M \otimes_{\kc} \cO_{k, 1}),
		\]
	in $\Mon_{\bE_1}(\cS)$. By passing to the limit, we obtain that the object $\tau_{\le n}(\cEnd(\mathfrak M)) $ can be naturally upgraded into an object in the \infcat $\Mon_{\bE_1}(\pro(\cS))$, as desired.
\end{proof}

\begin{lemma} \label{lem:fib of End's of formal to an}
The canonical morphism $q \colon \cEnd(\mathfrak M) \to \cEnd(M)$ admits a fiber equivalent to 
	\[
		\fib( q )\simeq \colim_{n \ge 1} (\fib( \cEnd(\mathfrak M) \xrightarrow{p^n } \cEnd(\mathfrak M))),
	\]
in $\Mon_{\bE_1}(\ind(\pro(\cS)))$.
\end{lemma}

\begin{proof} The forgetful functor
	\[
		\mathrm{forget} \colon \Mon_{\bE_1}(\ind(\pro(\cS)) \to \ind(\pro(\cS)),
	\]
is conservative and commutes with limits. Therefore, we are reduced to compute the fiber $\fib(q)$ in $\ind(\pro(\cS))$.
Since
	\[
		\cEnd(M) \simeq \colim_p \cEnd(\mathfrak M),
	\]
combined with the fact that filtered colimits preserve finite limits we conclude that computing the fiber of $q$, in $\Mon_{\bE_1}(\ind(\pro(\cS)))$, reduces to compute the fiber of the multiplication by $p^n$ of $\cEnd(\mathfrak M)$, for each $n \ge 1$. The claim of the
lemma now follows.
\end{proof}

\begin{definition}
We define the \emph{relative $n$-truncation} $q_n \colon \cEnd(\mathfrak M)_{\le n} \to \cEnd(M)$ of the morphism
as the colimit
	\[
		\colim_{n \ge 1} \tau_{\le n}( p^n ) \in \ind(\pro(\cS)).
	\]
\end{definition}

\begin{remark}
	Thanks to \cref{lem:monoidal structure on truncations} and \cref{lem:fib of End's of formal to an} we deduce that \[\underset{n \ge 1}{\colim} \tau_{\le n}(p^n),\] admits a canonical structure of $\bE_1$-monoid like object in $\ind(\pro(\cS))$, induced from the
	$\bE_1$-monoid like structure on each fiber
	$\fib(p^n)$.
\end{remark}

\begin{notation}
	Consider the suspension functor $\Sigma \colon \cS \to \cS$. We shall abusively denote by $\Sigma \colon \ind(\pro(\cS)) \to \ind(\pro(\cS))$ the extension of the latter along ind-pro objects in $\cS$.
\end{notation}

\begin{lemma} \label{lem:fiber_of_truncation_of_End's} Let $n \ge 0$. Consider the natural morphism
			\[ q_n \colon \cEnd(\mathfrak M) _{\le n}\to \cEnd(M).\]
 		Then we have a natural equivalence	
			\[\fib(q_n) \simeq \tau_{\le n}(\fib( \cEnd(\mathfrak M ) \to \cEnd(M))),\]
		in the \infcat $Mon_{\bE_1}(\ind(\pro(\cS)))$. Moreover, for each $n \ge 0$, we have natural morphisms	
		\[
			j_{n+1} \colon \cEnd(\mathfrak M)_{\le n+1} \to \cEnd(\mathfrak M)_{\le n},
		\]
	whose fiber is naturally equivalent to
		\[
			\fib(j_{n+1}) \simeq  \Sigma^{n+1}(\pi_{n+1}(\fib(q))),
		\]
	in $\Mon_{\bE_1}(\ind(\pro(\cS)))$.
\end{lemma}

\begin{proof}
	Since the forgetful functor 
		\[
			\mathrm{forget} \colon \Mon_{\bE_1}(\ind(\pro(\cS))) \to \ind(\pro(\cS)),
		\]
	is conservative and commutes with fibers, we are reduced to compute $\fib(q_n)$ and $\fib(j_{n+1})$ in $\ind(\pro(\cS))$. 
	The first assertion follows from \cref{lem:tau_n of fib of maps in pro} combined with the fact that filtered colimits preserve fibers and truncations.
	The existence of the morphism of $j_{n+1}$ is guaranteed by the universal property of the truncation functor.
	To compute $\fib(j_{n+1})$ we notice that the latter fits into a fiber sequence of the form	
		\[
			\fib(j_{n+1}) \to \fib( {q_{n+1}}) \to \fib(q_n).
		\]
	Since $\fib(q_{n+1}) \simeq \tau_{\le n + 1} (\fib(q))$ and $\fib(q_n) \simeq \tau_{\le n}(\fib(q))$, we deduce that 
	$\fib(j_{n+1}) \simeq  \Sigma^{n+1}(\pi_{n+1}(\fib(q)))$, as desired.
\end{proof}

\begin{remark}
	Let $n \ge 0$. Thanks to \cref{lem:fiber_of_truncation_of_End's} the fiber $\fib(j_{n+1})$ is \emph{ind}-$p$-nilpotent. Indeed, since the morphism \[q \colon \cEnd( \mathfrak M) \to \cEnd(M),\] is an equivalence after passing to the colimit under multiplication by $p$, we deduce from \cref{lem:fib of
	End's of formal to an} that
	$
		\fib(j_{n+1}),
	$ is a colimit of $p$-nilpotent objects.
\end{remark}

\subsection{Lifting results for continuous $p$-adic representations of profinite spaces} In this \S, we generalize to the setting of higher algebra the following standard result: let $G$ be a profinite group and
	\[
		\rho \colon G \to \GLn(\overline{\bQ}_p),
	\]
be a continuous representation of $G$. Then, up to conjugation, $\rho$ factors through the canonical inclusion
	\[
		\GLn(\overline{\bZ}_p) \subseteq \GLn(\overline{\bQ}_p).
	\]
Our goal is to prove an analogue when we replace $\rmB G$ by an object $X \in \pro(\cS^\fc)$ and we allow more general coefficients than simply a finite extension $E$ of $\bQ_p$. A priori, one cannot expect to have such an analogue in this degree of generality. For this reason, we will need to introduce the following definition:

\begin{definition}
Let $X \in \pro(\cS^\fc)$ be a connected profinite space. We say that $X$ is \emph{$p$-adically cohomological
compact} if, for any discrete $p$-torsion $\bZ_p$-module $N \in \Mod^\heartsuit_{\bZ_p}$, we have an equivalence of
mapping spaces
	\[
		\Map_{\Mon_{\bE_1}(\pro(\cS))} \big( \Omega X, N \big) \simeq \colim_\alpha
		\Map_{\Mon_{\bE_1}(\pro(\cS))} \big( \Omega X, N_\alpha \big),
	\]
for any presentation of $N \simeq \colim_{\alpha } N_\alpha$ as a filtered colimit of finite type (discrete) $\bZ_p$-modules $N_\alpha \in \Mod_{\bZ_p}^\heartsuit$.
\end{definition}

\begin{remark}
It is possible to give an analogous definition in the case where $X \in \cS$. The latter is equivalent
to require that $X$ admits a cellular decomposition with finitely many cells, in each dimension. Nonetheless,
$X$ itself might have infinitely many non-zero (finite) homotopy groups.
\end{remark}

\begin{definition}
	Let $f \colon Y \to X$ be a morphism in $\pro(\cS^\fc)$. We say that $f$ is a \emph{finite morphism} if its fiber
		\[
			\mathrm{fib}(f) \in \pro(\cS^\fc),
		\]
	is equivalent to a finite constructible space, i.e., an object in $\cS^\fc$.
\end{definition}

\begin{example}
\begin{enumerate}
\item
Let $Y \to X$ is a finite morphism in $\pro(\cS^\fc)$.
If we assume that $X$ is $p$-adically cohomological compact,
then so it is $Y$. More generally, the notion of $p$-adic compactness is stable under
fiber sequences.
\item Suppose $X \in \pro(\cS^\fc)$ is the \'etale homotopy type of a smooth variety over an
algebraically closed field. Then $X$ is $p$-adicallly cohomological compact. This is a consequence of the fact that $\pi_1^\et(X)$ is topologically of finite type. Indeed, in such case every continuous representation
	\[
		\rho \colon \Omega (\Sh^\et(X)) \to N ,
	\]
factors through the canonical morphism $\mathrm{can} \colon \Omega(\Sh^\et(X)) \to \pi_1^\et(X)$ (obtained by applying $\pi_0$ pro-component-wise). The result now follows from the observation that such $\rho$ is determined by the image of a finite set
of topological generators, and their images land in a sufficiently large $N_\alpha \subseteq N$.
\end{enumerate}
\end{example}

Fix now $X \in \pro(\cS^\fc)$ a connected $p$-adically compact profinite space. Let $A \in (\adCAlg)^{< \infty}$.
The main result of this \S \ can be formulated as:

\begin{theorem} \label{homotopy1}
Suppose we are given
$\rho \in \bPerfSys(X)(A \otimes_{\kc} k)$ such that 
	\[
		M \coloneqq \pi^*(\rho) \in \Perf(A \otimes_{\Ok} k),
	\]
admits a formal model $\mathfrak M \in \Coh(A)$. Then there exists $Y \in \Mon_{\bE_1}(
\pro(\cS^\fc))$ together with a finite morphism of enriched $\bE_1$-monoid-like objects
	\[
		f \colon Y \to \Omega X
	\]
in $\pro(\cS^\fc)$, such that we have a commutative diagram
	\[
	\begin{tikzcd}
		Y \ar{d} \ar{r}{\rho'} & \cEnd(\mathfrak M) \ar{d} \\
		\Omega X \ar{r}{\rho} & \cEnd(M)
	\end{tikzcd}
	\]
in the \infcat $\Mon_{\bE_1}(\ind( \pro(\cS)))$.
\end{theorem}


\begin{proof} The fact that
$A \in \adCAlg$ is assumed to be truncated, implies that so it is  $M $. For this reason, we can assume that also $\mathfrak M \in \mathbf{Coh}^+(A)$ is truncated (otherwise we replace $\mathfrak M$ by a sufficiently large truncation). Therefore, there exists a sufficiently
large $m > 0$ such that
	\begin{equation} \label{eq:iptrun}
		\tau_{\leq m} \cEnd(\mathfrak M) \simeq \cEnd(\mathfrak M) \quad \textrm{and } \ \ \ \tau_{\le m} \cEnd(M) \simeq \cEnd(M).
	\end{equation}
We now proceed to construct such a profinite $\bE_1$-monoid like  object $Y \in \Mon_{\bE_1}(\pro(\cS^\fc))$,
satisfying the conditions of the statement. We argue by induction on the relative Postnikov tower of the canonical morphism
	\[
		q \colon \cEnd(\mathfrak M) \to \cEnd(M).
	\]
For each $n \ge 0$, we construct morphisms $Y_{\le n} \to X$ which factor $\rho$ along the natural morphism
	\[q_n \colon \cEnd(\mathfrak M)_{\le n} \to \cEnd(M).\] Moreover, we shall prove that $Y_{\le n} \to X$ admits a finite fiber.
Let $n = 0 $, and consider the $0$-th relative truncation of $q$
	\[
		\cEnd(\mathfrak M) \to  \cEnd( \mathfrak M)_{\leq 0}  \to \cEnd(M),
	\]
in the \infcat $\Mon_{\bE_1}(\ind(\pro(\cS)))$. By construction, the morphism $q_0 \colon \cEnd(\mathfrak M)_{\le 0} \to \cEnd(M)$ fits into a pullback square of the form
	\begin{equation} \label{eq:pullback 0 level for lifting}
	\begin{tikzcd}
		\cEnd(\mathfrak M)_{\le 0} \ar{r} \ar{d} & \pi_0(\cEnd(\mathfrak M)) \ar{d} \\
		\cEnd(M) \ar{r} & \pi_0(\cEnd(M)),
	\end{tikzcd}
	\end{equation}
in $\Mon_{\bE_1}(\ind(\pro(\cS)))$. Indeed, the latter assertion follows by a direct computation of the fibers of each displayed vertical morphism, using the description of $q_n$ provided in \cref{lem:fiber_of_truncation_of_End's}.
Thanks to \cref{lem:universal property of truncation}, the composite
	\[
		\Omega X \xrightarrow[]{\rho} \cEnd(M) \to \pi_0(\cEnd(M)),
	\]
factors as a morphism
	\[
		\rho_0 \colon \pi_0(\Omega X) \to \pi_0(\cEnd(M)),
	\]	
in $\Mon_{\bE_1}(\ind(\pro(\cS)))$. We can naturally identify $\pi_0(\Omega X) = \pi_1(X)$ with a profinite group and $\pi_0(\cEnd(\mathfrak M) )$ with a topological group (the topology on the latter being induced by the ind-pro-structure).
Moreover, we can also regard $\pi_0(\cEnd(\mathfrak M))$ as an open subgroup of $\pi_0(\cEnd(M))$. Since $\pi_1(X)$ is a profinite
group, it follows that the inverse image
	\[
		\rho_0^{-1} (\pi_0(\cEnd( \mathfrak M))) \leq \pi_1(X),
	\]
corresponds to an open subgroup of $\pi_1(X)$.
Let $U \lhd \pi_1(X)$ be an open normal subgroup such that
	\[
		\rho_0(U) \subseteq \pi_0(\cEnd(\mathfrak M)) \subseteq \pi_0(\cEnd(M)),
	\]
and such that $\pi_1(X) / U \cong G$, where $G$ is a finite group. Consider the pullback
diagram
	\[
	\begin{tikzcd}
		Y_{\leq 0 } \ar{r} \ar{d}{h_0} & U \ar{d} \\
		\Omega X \ar{r} & \pi_1(X)
	\end{tikzcd}
	\]
now taken in the \infcat $\Mon_{\bE_1}(\pro(\cS^\fc))$. By construction, the morphism
$Y_{\leq 0} \to \Omega X$ admits a finite constructible fiber, namely $\rmB G$. Furthermore,
we have an equivalence
	\begin{equation} \label{eq:quotient_of_profinite_homotopy_types_by_finite_group}
		\Omega X \simeq Y_{\leq 0} / \rmB G
	\end{equation}
in the \infcat $\Mon_{\bE_1}(\pro(\cS^\fc))$. Moreover, we have a commutative diagram
	\[
	\begin{tikzcd}
		Y_{\le 0} \ar{r} \ar{d} & \pi_0( \cEnd(\mathfrak M)) \ar{d} \\
		\Omega X \ar{r} & \pi_0(\cEnd(M)),
	\end{tikzcd}
	\]
in $\Mon_{\bE_1}(\ind(\pro(\cS)))$. The pullback diagram \eqref{eq:pullback 0 level for lifting} implies that we have a commutative diagram of the form
	\[
	\begin{tikzcd}
		Y_{\le 0} \ar{r} \ar{d} &  \cEnd(\mathfrak M)_{\le 0} \ar{d} \\
		\Omega X \ar{r} & \cEnd(M),
	\end{tikzcd}
	\]
as desired.
The base step of our inductive reasoning is thus
finished. Let $n \geq 0 $, and suppose we have constructed a profinite space $Y_{\le n} \in \pro(\cS^\fc)$ fitting into a commutative
diagram
	\[
	\begin{tikzcd}
		Y_{\leq n } \ar{r} \ar{d}{h_n} & \cEnd(\mathfrak M)_{\leq n } \ar{d} \\
		\Omega X \ar{r}{\rho} & \cEnd(M)
	\end{tikzcd},
	\]
in $\Mon_{\bE_1}(\ind(\pro(\cS)))$, such that 
	\[\fib(h_n) \in \pro(\cS^\fc),\]
is finite constructible. Let
	\[
		q_{n+1} \colon \cEnd(\mathfrak M)_{\leq n+1} \to \cEnd(M)
	\]
be as in \cref{lem:fiber_of_truncation_of_End's}.
We have a commutative diagram of the form
	\[
	\begin{tikzcd}
		\cEnd(\mathfrak M)_{\leq n+1} \ar{r}{j_{n+1}} \ar{rrrd} & \cEnd(\mathfrak M)_{\leq n } \ar{r}{j_{n}} \ar{rrd} & 
		\dots \ar{r}{j_1}
		& \cEnd(\mathfrak M)_{\leq 0} \ar{d}{q_0} \\
		& & & \cEnd(M)
	\end{tikzcd}
	\]
in $\Mon_{\bE_1}(\ind(\pro(\cS)))$. Moreover, \cref{lem:fiber_of_truncation_of_End's} implies that
	\[
		\fib(j_{n+1}) \simeq \colim_k \Sigma^{n+1} \pi_{n+1} ( \fib(q)) .
	\]
Consider the following pullback diagram
	\[
	\begin{tikzcd}
		\widetilde{Y}_{\leq n+1} \ar{r} \ar{d}{\pi_n} & \cEnd(\mathfrak M)_{\leq n+1} \ar{d}{j_{n+1}} \\
		Y_{\leq n} \ar{r} & \cEnd(\mathfrak M)_{\leq n}
	\end{tikzcd},
	\]
in the \infcat $\Mon_{\bE_1}(\ind(\pro(\cS)))$. By construction, the fiber of the morphism $\pi_n \colon \widetilde{
Y}_{\leq n+1} \to Y_{\leq n}$ is equivalent to
	\[
		\fib(\pi_n) \simeq \fib(j_{n+1})  ,
	\]
in $\Mon_{\bE_1}(\ind(\pro(\cS)))$.
The fiber sequence
	\[
		\fib(\pi_n) \to \widetilde{Y}_{\leq n+1} \to Y_{\leq n},
	\]
is thus classified by a morphism 
	\[
		\varphi_n \colon Y_{\leq n } \to \Sigma^{n+2} \pi_{n+1} ( \fib(q)) ,
	\]
in $\Mon_{\bE_1}(\ind(\pro(\cS)))$. We observe that $\pi_{n+1}(\fib(q))$ is a discrete
group and the monoid structure on
	\[\Sigma^{n+2} \pi_{n+1}(\fib(q)) \in \Mon_{\bE_1}(\ind(\pro(\cS)))\]
is necessarily commutative, as $n + 2 \geq
2$. Thanks to \cref{lem:fib of End's of formal to an} it follows that 
	\begin{equation} \label{eq: Sigma n+1 pi_n+1 fib(q)}
		\Sigma^{n+2} \pi_{n+1}(\fib(q)) \simeq \colim_m \Sigma^{n+2}  \pi_{n+1}(\fib( p^m)),
	\end{equation}
in $\Mon_{\bE_1}(\ind(\pro(\cS)))$.
By the fact that $\Sigma^{n+2} \pi_{n+1}(\fib(p^n))$ is a commutative monoid-like object, it follows that each transition morphism 
	\[
		\Sigma^{n+2}  \pi_{n+1}(\fib( p^m)) \to \Sigma^{n+2}  \pi_{n+1}(\fib( p^{m+1}))
	\]
is a morphism of $\bE_1$-monoid like objects in $\ind(\pro(\cS))$. For this reason, we conclude that
	\[
		\varphi_n \colon Y_{\le n} \to  \Sigma^{n+2} \pi_{n+1} ( \fib(q)),
	\]
admits a factorization $Y_{\le n} \to \Sigma^{n+2} \pi_{n+1}(\fib(p^k))$, in $\Mon_{\bE_1}(\ind(\pro(\cS)))$, for sufficiently large $k \ge 1$. This induces a fiber sequence of the form
	\[
		\Sigma^{n+2}\pi_{n+1} ( \fib(p^k ))  \to \overline{Y}_{ \leq n+1} \to Y_{\leq n },
	\]
in $\Mon_{\bE_1}(\ind(\pro(\cS)))$.
Moreover, we have a natural composite
	\[
		\overline{Y}_{\leq n +1} \to \widetilde{Y}_{n+1} \to  \cEnd(\mathfrak M)_{\leq n +1} .
	\]
As $X$ is $p$-adically cohomological compact and the morphism $ Y_{\leq n } \to \Omega X$ admits
a finite constructible fiber, it follows that $ Y_{\leq n }$ is itself
$p$-adically cohomological compact. We can furthermore
regard \[ \pi_{n+1} (\fib( p^k \colon \cEnd(\mathfrak M) \to \cEnd(\mathfrak M))) ,\]
naturally as a discrete 
$\bZ/p^k$-module. We can thus write
	\[
		\pi_{n+1}(\fib(p^k)  \simeq \colim_{\alpha} N_{\alpha} ,
	\]
as a filtered colimit of discrete $\bZ/ p^k$-modules, $N_\alpha \in \Mod_{\bZ/ p^k}^\heartsuit$. By $p$-adic cohomological compactness, we obtain an equivalence
of mapping spaces
	\begin{align*}
		\Map_{\Mon_{\bE_1}(\pro(\cS^\fc))} \big( Y_{\leq n },\Sigma^{n+2} \pi_{n+1}(\cEnd(\mathfrak M) / \Omega(p^k)) \big) & \simeq \\
		&  \simeq \colim_\alpha \Map_{\Mon_{\bE_1}(\pro(\cS^\fc))} \big( Y_{\leq n}, 
		\Sigma^{n+2} N_\alpha  \big).
	\end{align*}
Therefore, the map $\varphi_n$ above factors through a morphism
	\[
		\varphi_{\beta, n } \colon Y_{\leq n } \to \Sigma^{n+2} N_\beta,
	\]
in $\Mon_{\bE_1}(\cS^\fc)$, for some large enough index $\beta$. Such factorization produces an extension
	\[
		\Sigma^{n+2} N_\beta \to Y_{\leq n +1} \xrightarrow{j_{n+1}} Y_{\leq n } ,
	\]	
in the \infcat $\Mon_{\bE_1} (\pro(\cS^\fc))$. Moreover, by construction, it follows that
the composite
	\[
		Y_{\leq n + 1} \to Y_{\leq n } \to \dots \to \Omega X \to \cEnd(M)
	\]
factors through the canonical morphism $\cEnd(\mathfrak M)_{\leq n + 1} \to \cEnd(M)$. The inductive
step is thus completed. 

In order to finish the proof of the statement it suffices now to observe that there exists a sufficiently
large $m \ge 1$ such that
	\[
		\cEnd(\mathfrak M)_{\leq m} \simeq \cEnd(\mathfrak M).
	\]
This is a consequence of the fact that $\mathfrak M$ is $m$-truncated combined with the construction of $\cEnd(\mathfrak M)_{\le m}$ provided in \cref{lem:fiber_of_truncation_of_End's}.
We have thus produced a  finite morphism $Y_{\le m} \to \Omega X$ fitting into a
commutative diagram
	\[
	\begin{tikzcd}
		Y \ar{d} \ar{r} & \cEnd(\mathfrak M) \ar{d} \\
		\Omega X \ar{r} & \cEnd(M)
	\end{tikzcd}
	\]
in the \infcat $\Mon_{\bE_1}( \ind(\pro(\cS)))$, as desired.
\end{proof}

\begin{remark}
Recall that given for any $A \in \adCAlg$ the natural morphism of derived rings $ f \colon A \to A \otimes_{\kc} A$ induces a base change functor
	\[
		f^* \colon \mathbf{Perf}_{\kc}(A) \to \mathbf{Perf}_k(A \otimes_{\kc} k) ,
	\]
in $\Cat(\ind(\pro(\cS)))$. Therefore, we have an induced pullback functor at the level of \infcats of continuous representations
	\[
		f^* \colon \bPerfSys(X)(A) \to \bPerfSys(X)(A \otimes_{\kc} k).
	\]
\end{remark}
	
\begin{definition} Let $X \in \pro(\cS^\fc)$ and $A \in \adCAlg$ be as above.
We say that $ \rho \in \bPerfSys(X)(A \otimes_{\kc} k)$ is \emph{liftable} if it lies in the essential image of the
base change functor
	\[ \bPerfSys(X)(A) \to \bPerfSys(X)(A \otimes_{\kc} k).\]
\end{definition}

Fix $A \in (\adCAlg)^{< \infty}$. As a consequence of \cref{homotopy1} we have the following crucial statement:

\begin{corollary} \label{homotopy2}
Let $X \in \pro(\cS^\fc)$ be a connected and $p$-adically cohomological compact profinite space.
Then every object
$\rho \in \bPerfSys (X)(A \otimes_{\kc} k )$ is a retract of a liftable $\rho' \in  \bPerfSys(X)(A \otimes_{\kc} k)$. In particular,
the pullback functor
	\[
		f^* \colon \bPerfSys(X) (A ) \to \bPerfSys(X)(A \otimes_{\kc} k),
	\]
induces an equivalence 
	\[\bPerfSys(X)(A) \otimes_{\kc} k \simeq \bPerfSys(X)(A \otimes_{\kc} k),\]
in $\Cat^{ \mathrm{perf}}$.
\end{corollary}

\begin{proof}
Let $\rho \in \bPerfSys(X)(A \otimes_{\kc} k)$,
$M \coloneqq \pi^*(\rho) \in \Perf(A \otimes_{\Ok} k )$
and let $g \colon Y \to X$ be a finite morphism in $\pro(\cS^\fc)$. Suppose further that
$Y$ connected and the composite
	\[
		\Omega Y \xrightarrow{g} \Omega X \to \cEnd(M)
	\]
admits a factorization via $\cEnd(\mathfrak M) \to \cEnd(M)$, where $\mathfrak M \in \mathbf{Coh}^+(A)$ is a formal model for $M$. Such datum exists thanks to \cref{homotopy1}.
The proof of the latter implies that we have a commutative
diagram of the form
	\[
	\begin{tikzcd}
		Y \simeq Y_{\leq n+1} \ar{r}{g_n} \ar{rrrd} & Y_{\leq n } \ar{r}{g_{n-1}} \ar{rrd} & 
		\dots \ar{r}{j_0}
		& Y_{\leq 0} \ar{d}{g_0} \\
		& & & X
	\end{tikzcd}
	\]
in the \infcat $\pro(\cS^\fc)$, for which we have 
	\[X \simeq Y_{\leq 0 } / \Gamma,\]
in $\pro(\cS^{\fc})$,
where $\Gamma$ is a suitable finite group (not necessarily abelian), see \eqref{eq:quotient_of_profinite_homotopy_types_by_finite_group}. In
particular, we have an equivalence of \infcats
	\[
		\bPerfSys(X)(A \otimes_{\kc} k) \simeq \bPerfSys(Y_{\leq 0})(A \otimes_{\kc} k)^{\rmB \Gamma},
	\]
of $A \otimes_{\Ok}k$-linear stable \infcats. Moreover, the proof of \cref{homotopy1} implies that
for each integer $0 \leq i \leq n -1$ we can choose the morphism
	\[
		g_{i} \colon Y_{\leq i+1} \to Y_{\leq i},
	\]
in such a way that $g_i$ is a $M_i[n+2]$-torsor, where $M_i$ is a finite abelian group. As $A \otimes_{\Ok} k$ is $k$-linear and $k$ is
a field of characteristic zero, it follows that we have an equivalence of $A \otimes_{\kc} k$-linear stable
\infcats
	\[
		\bPerfSys(Y_{\leq 0 } )(A \otimes_{\kc}) \simeq \bPerfSys(Y_{\leq i }) (A \otimes_{\kc} k),
	\]
for each integer $0 \leq i \leq n $. One thus deduces inductively an equivalence
of $(A \otimes_{\kc} k)$-linear \infcats
	\[
		\bPerfSys(X)(A \otimes_{\kc} k) \simeq \bPerfSys(Y)(A \otimes_{\kc} k)^{\rmB \Gamma}.
	\]
We are thus provided with an adjunction
	\[
	\begin{tikzcd}
		g^* \colon \bPerfSys(X)(A) \arrow[r, shift left] & \bPerfSys(Y)(A) \colon g_* 
		\arrow[l, shift left]
	\end{tikzcd}
	\]
where $g_*$ denotes the restriction functor along $g \colon Y \to X$. Given $\rho \in
\bPerfSys(Y)(A)$ we have an equivalence
	\[
		g^* \rho \simeq \rho \otimes_{\Ok} k [ \Gamma],
	\]
where $k [\Gamma]$ denotes the free $k$-algebra on the finite group $\Gamma$. Moreover, the
representation $\rho$ is a retract of $g_* g^* \rho$, given by the trivial morphism of groups
	\[
		\{ 1\} \to \Gamma.
	\]
Observe further that the representation $g_* (g^* ( \rho))$ is liftable by the choice of $Y$, since
$g^*(\rho)$ was already so. This finishes the proof of the first assertion of the Corollary.

Let us now prove that
$\bPerfSys(X)(A \otimes_{\kc} k)$ is equivalent to the idempotent completion of
\[\bPerfSys(X)(A) \otimes_{\Ok} k .\] It suffices, to prove that $\bPerfSys(X)(A \otimes_{\kc} k)$
is idempotent complete. By the first part of the proof we reduce ourselves to show that given any idempotent morphism
	\[
		f \colon \rho \to \rho,
	\]
for which $\rho$ is liftable, both
	\[
		\fib(f) \quad \textrm{and } \ \ \  \cofib(f)
	\]
do exist in the  \infcat $\bPerfSys(X)(A \otimes_{\kc} k)$. Let
\[M \coloneqq \pi^*(\rho) \in \Perf(A \otimes_{\Ok} k).\]
By evaluation, $f$
induces an idempotent morphism
	\[
		\widetilde{f} \colon M \to M,
	\]
in $\Perf(A \otimes_{\kc} k)$.
We might not be able to find a lift of $\widetilde{f}$ directly on $\Perf(A)$. Nonetheless \cite[Corollary 3.17]{Antonio_Porta_Non_archimedean_Hilbert} combined with
\cite[Proposition 4.4.5.20]{lurie2009higher} imply that there exists a formal model $\mathfrak M \in \Coh(A)$
for $M$, for which the diagram 
	\[\widetilde{f} \colon \mathrm{Idem} \to \Perf(A \otimes_{\Ok} k ),\]
induced by $\widetilde{f}$, lifts to a formal model
diagram $\overline{f} \colon \Idem \to \Coh(A)$.

It follows that $\overline{f}$ induces a well defined diagram
	\[
		\overline{f}^{\cont} \colon \Idem \to \cC \mathrm{ontFun}(X, \mathbf{Coh}^+(A)),
	\] 
corresponding to an idempotent morphism
	\[
		\overline{f}^\cont \colon \overline{\rho} \to \overline{\rho},
	\] 
where $\overline{\rho} \colon \Omega( X) \to \cEnd(\mathfrak M)$ denotes a (suitable) continuous $A$-adic representation of $X$. Moreover,
$\overline{f}^\cont$ admits both a fiber and a cofiber in $\cC \mathrm{ontFun}(X, \mathbf{Coh}^+(A))$. Namely,
	\[
		\Omega(X) \to \cEnd(\fib(\overline{f})) \quad \textrm{and} \quad \Omega(X) \to \cEnd(\cofib(\overline{f})).
	\]
The result now follows from the observation that the image, under (exact) base change functor along $A \to A \otimes_{\kc} k$, of both $\fib(\overline{f}^{\cont})$ and $\cofib(\overline{f}^{\cont})$ lie in the full subcategory
	\[
		\bPerfSys(X)(A \otimes_{\kc} k) \subset \cC \mathrm{ontFun}(X, \mathbf{Coh}^+(A \otimes_{\kc} k)),
	\]
and they correspond to the fiber and cofiber of the idempotent morphism $f$, respectively.
\end{proof}

\section{Representability of the moduli}

\subsection{Representability theorem} As we shall see, the moduli stack
$\LocSys(X) \colon \Afd_k^\op \to \cS$ admits a natural derived extension
which is representable by a derived (geometric) $k$-analytic stack.
In order to prove this assertion, we will use in a crucial way the $k$-analytic analogue of the Artin-Lurie representability theorem, proved
in \cite[Theorem 7.1]{porta2017representability}.
As it will be of fundamental importance we shall motivate such result.

\begin{definition}
Denote by $( \dAfd_k, \tau_\et, \rmP_\sm)$ the \emph{derived $k$-analytic
geometric context} where
$\tau_\et$ denotes the \'etale topology on $\dAfd_k$ and $\rmP_\sm$ denotes
the class of smooth morphisms on $\dAfd_k$, see \cite[Definition 5.1]{porta2016derived} and \cite[Definition 5.46]{porta2017representability}.
\end{definition}

\begin{definition}
Let $F \in \dSt \big( \dAfd_k, \tau_\et \big)$ be a stack. We say that
$F$ is a \emph{derived $k$-analytic stack} if it is representable by a
geometric stack with respect to $(\dAfd_k, \tau_\et, \rmP_\sm)$.
\end{definition}

\begin{theorem}{\emph{\cite[Theorem 7.1]{porta2017representability}}} \label{thm:representability}
Let $F \in \dSt ( \dAfd_k, \tau_\emphet )$. The following assertions are equivalent:
\begin{enumerate}
\item $F$ is a geometric stack;
\item The truncation $\trun_{\leq 0} (F) \in \St(\Afd_k, \tau_\emphet )$ is geometric,
$F$ admits a cotangent complex and it is cohesive and nilcomplete.
\end{enumerate}
\end{theorem}

We shall review the main definitions:

\begin{definition}
Let $F \in \dSt \big( \dAfd_k, \tau_\et \big)$. We say that $F$ admits a \emph{global
analytic cotangent complex} if the following two conditions are verified:
\begin{enumerate}
\item Given $Z \in \dAfd_k$ and $z \colon Z \to F$ a morphism, the functor
	\[
		\mathrm{Der}^\an_F (Z, - ) \colon \Coh(Z) \to \cS,
	\]
given by the formula
	\[
		M \mapsto \fib_z \big( F(Z[M]) \to F(Z) \big),
	\]
is corepresented by an eventually connective object $\bL^\an_{F, z} \in 
\Coh(Z)$.
\item For any morphism $f \colon Z \to Z'$ of derived $k$-affinoid spaces and any
morphism $z' \colon Z \to F$ we have a canonical equivalence,
	\[
		f^* \bL^\an_{F, z'} \simeq \bL^\an_{F, z}
	\]
where $z \coloneqq z' \circ f$.
\end{enumerate}
\end{definition}

\begin{definition}
Let $F \in \dSt( \dAfd_k, \tau_\et )$. We say that $F$ is \emph{cohesive} or \emph{infinitesimally cartesian} if for every
$Z \in \dAfd_k$, every $\cF \in \mathrm{Coh}^{\geq 1}(Z)$
and any derivation	
	\[
		d \colon \bL^\an_X \to \cF,
	\]
the natural map
	\[
		F \big(Z_d[\cF[-1]] \big) \to F(Z) \times_{F(Z[\cF])} F(Z)
	\]
is an equivalence in the \infcat $\cS$.
\end{definition}

\begin{definition}
Let $F \in \dSt \big( \dAfd_k, \tau_\et \big)$. We say that $F$ is
\emph{convergent} if for every derived $k$-affinoid space $Z$ the canonical morphism,
	\[
		F(Z) \to \underset{n \geq 0}{\lim} F(\trun_{\leq n } Z),
	\]
is an equivalence in the \infcat $\cS$.
\end{definition}

\subsection{Deformation theory of $p$-adic continuous representations} In this \S, we define the moduli
of derived continuous $p$-adic representations of a profinite space $X$
and we show that it admits
a derived structure under mild assumptions on $X \in \pro(\cS^\fc)$.

\begin{definition}
Let $X \in \pro(\cS^\fc)$. The \emph{moduli of derived continuous $p$-adic representations}
of $X$ is defined as the right Kan extension, along the canonical inclusion functor
	\[
		j \colon \dAfd_k^{< \infty} \to \dAfd_k,
	\]
of the \emph{\'etale hyper-sheafification} of the moduli functor
	\[
		\PerfSys(X) \coloneqq \functor^\simeq \circ \bPerfSys(X) \colon \big(
		\dAfd^{< \infty}_k \big)^{\op}
		\to \cS.
	\]
\end{definition}

We have the following description:

\begin{lemma}
Let $Z \in \dAfd_k$. Then we have a natural equivalence
	\[
		\PerfSys(X)(Z) \simeq \lim_n \PerfSys(X)(\trun_{\leq n }Z ) 
	\]
in the \infcat $\cS$. In particular, the functor $\PerfSys(X)$ is nilcomplete.
\end{lemma}

\begin{proof}
We shall denote by
\[\cT_Z \coloneqq (\dAfd_k^{< \infty} )^{\mathrm{op}}_{Z / },\]
and define $\cT_Z'$ the full subcategory of $\cT_Z$ spanned
by those objects of the form $\trun_{\leq n }Y \to Z$, for each $n \geq 0$. By the end
formula for right Kan extensions, cf. \cite[4.3.2.17.]{lurie2009higher}, it suffices to show that the inclusion functor
	\[
		\cT'_Z \to \cT_Z,
	\]
is a final functor. Thanks to the dual statement of \cite[Theorem 4.1.3.1]{lurie2009higher} we are reduced to show that for every $(Y \to Z)^\op$ in $\cT_Z$, the \infcat
	\[
		(\cT_Z')_{/ Y} ,
	\]
has weakly contractible enveloping groupoid. We can identify the \infcat $(\cT'_Z)_{/ Y}$ with
the \infcat of factorizations of the morphism $(Y \to Z)^\op$. Thanks to the universal property
of the functor $\tau_{\le n}$ combined with the fact that $Y $ is truncated it follows
that there exists a sufficiently large integer $m$ such that 
	\[
		(Y \to X)^\op,
	\]
factors uniquely as 
	\[
		( Y \to \trun_{\leq m} X \to X)^\op.
	\]
Therefore the \infcat $(\cT'_X)_{/ Y}$ is cofiltered and thus weakly contractible, as desired.
\end{proof}



\begin{proposition}
The functor $\PerfSys(X) \colon \dAfd_k^\op \to \cS$ is cohesive.
\end{proposition}

\begin{proof}
The right adjoint $\functor^\simeq \colon \Cat \to \cS$ commutes with small limits and in particular with finite limits. Moreover, $\PerfSys(X)$ is nilcomplete, thus we can restrict
ourselves to prove the assertion when restricted to truncated objects.
As a consequence, it suffices to show that the functor
	\[
		\bPerfSys (X) \colon \dAfd^{< \infty}_k \to \Cat
	\]
is infinitesimally cartesian. Let $Z \in \dAfd_k^{< \infty}$ and let $d \colon \bL^\an_Z \to M$ be a $k$-analytic derivation of $Z$, and $M \in \Coh(Z)$. Thanks to
\cite[Theorem A.2.1]{antonio2018p} we can lift $d$ to a formal derivation
	\[
		d ' \colon \bL^\ad_A \to \mathfrak M ,
	\]
in $\Coh(A)$, where $A \in (\adCAlg)^{< \infty}$ satisfies
	\[
		(\Spf A)^\rig \simeq Z.
	\]
For this reason $Z_d[M]$ can be obtained as the rigidification $\Spf( A_{d'}[\mathfrak M])^\rig \in \dAfd_k$.
Therefore, the canonical functor
	\[
		\bPerfSys(X)(A_{d'}[\mathfrak M] \otimes_{\kc} k) \to \bPerfSys(X)( A \otimes_{\kc} k) \times_{ \bPerfSys(X)((A \oplus \mathfrak{M}) \otimes_{\kc} k) } \bPerfSys(X)(A \otimes_{\kc} k)
	\]
is an equivalence, which is an immediate consequence of \cref{inf:cart}.
\end{proof}

We now proceed to study the associated tangent complex of $\PerfSys(X)$. In order to do so, we will need to state and prove the following technical result:

\begin{proposition} \label{use}
Let $F \in \St ( \adCAlg, \tau_\emphet, \rmP_\sm )$ and denote by $F^\rig \in
\St( \dAfd_k,, \tau_\emphet, \rmP_\sm )$ its rigidification. Suppose that $F$ admits an
adic cotangent
complex, $\bL^\ad_{F, x}$, at a point
	\[
		x \colon \Spf(A) \to F,
	\]
then $F^\rig$ admits an analytic cotangent complex, $\bL^\an_{F^\rig, x^\rig}$, at the rigidification
	\[
		x^\rig \colon \Spf(A)^\rig \to F^\rig.
	\]
Moreover, we have a canonical equivalence
	\[
		(\bL^\ad_{F, x})^\rig \simeq \bL^\an_{F^\rig, x^\rig},
	\]
in the \infcat $\Coh(\Spf(A)^\rig)$.
\end{proposition}

\begin{proof}
The existence of $\bL^\ad_{F, x}$ implies that for every $M \in \Coh(A)$ we have functorial
equivalences
	\[
		\Map_{\Coh(A)} \big( \bL^\ad_{F, x}, M \big) \simeq
		\fib_x \big( F(A \oplus M ) \to F(A) \big).
	\]
Thanks to \cite[Theorem 3.7]{Antonio_Porta_Non_archimedean_Hilbert} the \infcat $\Coh(\Spf(A)^\rig)$ is a Verdier quotient
of $\Coh(A)$. Furthermore, it follows from \cite[Lemma 3.14]{Antonio_Porta_Non_archimedean_Hilbert} and its proof
that the we have an adjunction
	\[
	\begin{tikzcd}
		\rigg \colon \ind \big( \Coh(A) \big) \arrow[r, shift left] & \ind \big( \Coh(X) \big) \colon
		\functor^+ \arrow[l, shift left],
	\end{tikzcd}
	\]
between presentable \infcats, where $\rigg$ is an accessible localization functor and
$\functor^+$ is consequently fully faithful. We have an equivalence of mapping spaces
	\[	
		\Map_{\Coh(X)} \big( \bL^\rig_{F, x}, M \big) \simeq \Map_{\ind (\Coh(A))}
		\big(\bL_{F, x}, (M^{\rig})^+ \big).
	\]
Since $(M^{\rig})^+ \in \ind ( \Coh(A)) $ we can write it as a filtered colimit
	\begin{align*}
		(M^{\rig})^+ &  \simeq \colim_i M_i  \\
		 	& \simeq \colim_{M \in \Coh(A)_{/ M^\rig}} M,
	\end{align*}
where the last equivalence follows by adjunction. Since $\bL_{F, x}$ is a compact object of $\ind(\Coh(A))$, we have
	\[
		\Map_{\ind(\Coh(A))} \big( \bL_{F, x}, (M^{\rig})^+\big) \simeq 
		\colim_i \Map_{\Coh(A)} \big( \bL_{F, x}, M_i \big).
	\]
Let $N \in \Coh( \Spf(A)^\rig)$. Put $X = \Spf(A)^\rig$. We have a chain of equivalences 
	\begin{align*}
		\colim_{A ' \in (\adCAlg)_{/ X} } \colim_{M_i \in \Coh(A')_{/ N}}
		\fib_{x'} \big( F(A' \oplus M_i ) \to F(A') \big) & \simeq \\
		&\simeq \colim_{A' \in (\adCAlg)_{/ X}} \colim_{M_i \in \Coh(A')_{/ N}}
		\Map_{\Coh(A')} \big(\bL_{F, x}, M_i \big) \\
		& \simeq
		\colim_{A ' \in (\adCAlg)_{/ X}} \Map_{\ind(\Coh(A'))} \big( \bL_{F, x}, N^+ \big) \\
		& \simeq \colim_{A' \in (\adCAlg)_{/ X}} \Map_{\Coh(X)} \big(\bL^\rig_{F, x}, N \big) \\
		& \simeq \Map_{\Coh(X)} (\bL^\rig_{F, x}, N \big).
	\end{align*}
Both colimit indexing \infcats are filtered and $x'$ denotes the composite
	\[
		x \colon \Spf A' \to \Spf A \to F.
	\]
Notice that in the above colimit diagrams it
suffices to consider only colimits indexed by the full subcategories of formal models for $X$
and lying under $A$. Furthermore, we have an equivalence
	\[
		\colim_{A ' \in (\adCAlg)_{/ X} } \colim_{M_i \in \Coh(A')_{/ N}}
		\fib_{x'} \big( F(A' \oplus M_i ) \to F(A') \big) \simeq
		\colim_{\cC} \fib_{x'} \big( F(A' \oplus M') \to F(A') \big),
	\]
where $\cC$ denotes the \infcat of admissible formal models for $X[N]$. This last assertion
follows from the observation that a formal model for $X[N]$ consists of the given of an admissible
formal model for $X$ together with a formal model for $N$. We have thus a chain of natural equivalences
	\begin{align*}
		\colim_{\cC} \fib_{x'} \big( F(A' \oplus M') \to F(A') \big) & \simeq \\
		& \simeq \fib_{x^\rig} \colim_{\cC} \big( F(A' \oplus M') \to F(A') \big) \\
		& \simeq \fib_{x^\rig} \big( F^\rig(X[N]) \to F^\rig(X) \big).
	\end{align*}
Thus, we obtain
	\begin{equation*} \label{16}
		\fib_{x^\rig} \big( F^\rig(X[N]) \to F^\rig(X) \big) \simeq \Map_{\Coh(X)} \big( 
		\bL^\an_{F^\rig, x^\rig} , N \big)
	\end{equation*}
as desired. 
\end{proof}

\begin{notation}
	Let $Z \in \dAfd_k$ and $\rho \in \bPerfSys(X)(Z)$. Suppose we are given $M \in \Coh(Z)$, we shall denote by
		\begin{align*}
			\bT_{\PerfSys(X), \rho} (M) & \coloneqq \fib( \PerfSys(X)(Z[M]) \to \PerfSys(X)(Z) ) \\
								  & \in \cS,
		\end{align*}
	the \emph{tangent space of $\PerfSys(X)$ at $\rho$ and $M$}. When $M = \rmR \Gamma(Z, \cO_Z) [1]$, we shall simply denote \[\bT_{\PerfSys(X), \rho }\coloneqq \bT_{\PerfSys(X), \rho} (\rmR \Gamma(Z, \cO_Z)[1]).\]
\end{notation}

Let $X \in \pro(\cS^\fc)$ denote a $p$-adically compact profinite space. We can compute its tangent complex as follows:

\begin{proposition} \label{prop:tangent_complex_of_PerfSys}
Let $Z \in \dAfd_k$ and $M \in \Coh(Z)$. Suppose we are given a morphism
	\[
		\rho \colon Z \to \PerfSys(X).
	\]
Then we have a natural equivalence of almost perfect $\rmR \Gamma(Z, \cO_Z)$-modules,
	\begin{equation} \label{eq:formula for the tangent}
		\bT_{\PerfSys(X), \rho}(M) \simeq \Map_{\bPerfSys(X)(Z) } \big( \mathbf{1},  \rho \otimes \rho^\vee) \otimes M[1] .
	\end{equation}
Moreover, the $\rmR \Gamma(Z, \cO_Z)$-module $\underline{\Map}_{\bPerfSys(X)(Z) } \big( \mathbf{1},  \rho \otimes \rho^\vee) $ is equivalent to a perfect $\rmR \Gamma(Z, \cO_Z)$-module.
\end{proposition}

Before giving the proof of \cref{prop:tangent_complex_of_PerfSys} we need to introduce a few preliminary considerations:

\begin{notation} Let $A \in \adCAlg$.
	We shall denote by $\mathbf{triv}_A \in \bPerfSys(X)(A)$ the unit of the symmetric monoidal \infcat $\bPerfSys(X)(A)^\otimes$. In particular, given any morphism $f \colon A \to B$ in $\adCAlg$, we have that the base change functor
		\[
			f^* \colon \bPerfSys(X)(A) \to \bPerfSys(X)(B),
		\]
	can be explicitly described by the formula
		\[
			\rho \in \bPerfSys(X)(A) \mapsto \rho \otimes_{\triv_A} \triv_B \in \bPerfSys(X)(B).
		\]
\end{notation}

Let $A \in \adCAlg$ and $M \in \Perf(A)$. Before proving \cref{prop:tangent_complex_of_PerfSys} we shall need a preliminary lemma:

\begin{proposition} \label{split}
Let
$\rho \in \bPerfSys(X)(A)$ and $\rho ' \in \bPerfSys(X)(A \oplus M )$ together with a morphism
	\[
		\theta \colon \rho ' \to \rho \otimes_{\triv_A} \triv_{ A \oplus M},
	\]
in the \infcat $\bPerfSys(X)(A \oplus M )$.
Assume further that $\theta$ becomes an equivalence, after base change along the canonical morphism
	\[
		A \oplus M \to A.
	\]
Then $\theta $ is an equivalence in the \infcat $\bPerfSys(X)(A \oplus M )$.
\end{proposition}

\begin{proof}
It suffices to prove the result in the case where 
	\[
		\theta \colon \rho' \to \rho \otimes_{\mathbf{triv}_A} (\mathbf{triv}_{A \oplus M} ),
	\]
coincides with the identity morphism
	\[
		\id \colon \rho \otimes_{\mathbf{triv}_A} \mathbf{triv}_{A \oplus M} \to \rho \otimes_{\mathbf{triv}(A)} \mathbf{triv}_{A \oplus M}.
	\]
Consider the cofiber $\cofib(\theta) \in \bPerfSys(X)(A \oplus M )$. The latter is a
dualizable object in the \infcat $\bPerfSys(X)(A \oplus M )$. Moreover, its image in $\bPerfSys(X)(A)$ is equivalent
to the zero object. We wish to prove that the coevaluation morphism
	\[
		\mathrm{coev} \colon \mathbf{triv}_{A \oplus M} \to \cofib(\theta) \otimes \cofib(\theta)^\vee,
	\]	
is the zero map in the \infcat $\bPerfSys(X)(A)$. Consider the inclusion morphism
	\[
		\mathbf{triv}_A \to \mathbf{triv}_{A \oplus M} ,
	\]
in $\Mod_A$. By naturality of the tensor product functor, we obtain a commutative diagram of the form
	\[
	\begin{tikzcd}
		\mathbf{triv}_A \ar{r}{\mathrm{coev} \otimes \mathbf{triv}_A } \ar{d} & (\cofib(\theta) \otimes \cofib(\theta)^\vee) \otimes \mathbf{triv}_A \ar{d} 
		\\
		\mathbf{triv}_{A \oplus M} \ar{r}{\mathrm{coev}} & \cofib(\theta) \otimes \cofib(\theta)^\vee.
	\end{tikzcd}
	\]
Notice that $\mathrm{coev} \otimes \mathbf{triv}_A$ corresponds to the coevaluation morphism of the dualizable object
	\[
		\cofib(\theta) \otimes_{\mathbf{triv}_{A \oplus M} }\mathbf{triv}_A \in \bPerfSys(X)(A).
	\]
Thus it must coincide with the identity morphism
	\[
		\cofib(\theta) \otimes \mathbf{triv}_A \to \cofib(\theta) \otimes \mathbf{triv}_A.
	\]
The latter is the zero morphism, by our assumption on $\theta$. It thus follows that the $A$-linear morphism
	\[
		\mathbf{triv}_A \to \cofib(\theta) \otimes \cofib(\theta)^\vee,
	\]
is the zero morphism. By adjunction, the coevaluation map
	\[
		\mathrm{coev} \colon \mathbf{triv}_{A \oplus M} \to \cofib(\theta) \otimes \cofib(\theta)^\vee
	\]
is the zero morphism, as well. Thus $\cofib(\theta) \simeq 0 $ in the \infcat $\bPerfSys(X)(A \oplus M )$,
as desired.
\end{proof}

\begin{proof}[Proof of \cref{prop:tangent_complex_of_PerfSys}]
We first observe that the derived Tate acyclicity theorem, cf. \cite[Theorem 3.1.]{Porta_Yue_derived_Hom}, implies that we have a canonical equivalence of \infcats
	\[
		\Coh(Z) \simeq \Coh( \rmR \Gamma(Z, \cO_Z)).
	\]
We will use this fact without further mention it.
Consider the right fibration of
spaces
	\[
		\lambda_M \colon \PerfSys(X)(Z[M]) \to \PerfSys(X)(Z),
	\]
which classifies a functor
	\[
		\lambda_M \colon \PerfSys(X)(Z) \to \cS,
	\]
whose value at $\rho \in \PerfSys(X)(Z)$ is equivalent to
	\[
		\bT_{\PerfSys(X), \rho} \in \cS.
	\]
Since the formula displayed in \eqref{eq:formula for the tangent} is closed under taking fibers of idempotent morphisms, we reduce ourselves, via \cref{homotopy2}, to the case where
	\[
		\rho \in \bPerfSys(X)(Z),
	\]	
is liftable. Thanks to the proof of \cite[Theorem 5.1]{thomason1990higher}, it follows that 
	\[
		N \coloneqq \pi^*(\rho) \in \Perf(\rmR \Gamma(Z, \cO_Z)),
	\]
admits a formal model $\mathfrak N \in \Perf(A)$, up to retract, for some derived $\kc$-adic algebra $A \in \adCAlg$ such that
	\[
		\Spf(A)^\rig \simeq Z.
	\]
Again by the closure of \eqref{eq:formula for the tangent} under retractions combined with the fact that $\Perf(\rmR \Gamma(Z, \cO_Z))$ is idempotent complete, we are reduced to the case where
$N $ admits a formal model $\mathfrak N \in \Perf(A)$. Let $\mathfrak M \in \Coh(A)$ denote a formal model for $M \in \Coh(\rmR \Gamma(Z, \cO_Z))$. Thanks to second assertion in \cref{homotopy2} it suffices to prove the statement of \label{prop:tangent_complex_of_PerfSys}
in the $A$-adic case.
Observe further that
\cref{split} and its proof allow us to identify the left hand side of \eqref{eq:formula for the tangent} with the morphism
	\begin{align*}
		\Omega( p_M(\rho))& \simeq  \fib_{\id_\rho} \big(\Map_{\bPerfSys(A \oplus \mathfrak M )} \big( \rho \otimes (A \oplus \mathfrak M),
		  \rho \otimes (A \oplus  \mathfrak M) \big) \\
		   & \to \Map_{\bPerfSys(X)(A)} (\rho, \rho).
	\end{align*}
We shall denote the right hand side simply by $\Map_{/ \rho} \big( \rho \otimes_A (A \oplus \mathfrak M ), \rho \otimes_A (A \oplus
\mathfrak M) \big)$. Since the underlying $A$-module of $\rho \otimes_A (A \oplus  \mathfrak M )$ can be identified with $\rho \oplus 
\rho \otimes_A \mathfrak M$, we have a chain of natural equivalences of mapping spaces
	\begin{align*}
		\Map_{/ \rho} \big( \rho \otimes_A (A \oplus \mathfrak  M ), \rho \otimes_A (A \oplus \mathfrak M) \big) & \simeq \\
		& \simeq \Map_{\bPerfSys(X)(A)_{/ \rho}} \big( \rho \otimes_A (A \oplus \mathfrak M ), \rho \otimes_A (A \oplus \mathfrak M) 
		\big) \\
		& \simeq \Map_{\bPerfSys(X)(A)_{/ \rho}} \big( \rho, \rho \otimes_A (A \oplus \mathfrak M) 
		\big) .
	\end{align*}
Notice that the latter mapping space is pointed at the zero morphism. Since $\rho \in \bPerfSys(X)(A)$ is a dualizable
object we have an equivalence of mapping spaces
	\[
		\Map_{\PerfSys(A)} \big( \rho , \rho \otimes_A \mathfrak M \big) \simeq \Map_{\PerfSys(A)} \big(
		\rho \otimes \rho^\vee, \mathfrak M \big).
	\]
Consider the pullback diagram of derived $\Ok$-adic algebras
	\[
	\begin{tikzcd}
		A \oplus \mathfrak M \ar{r} \ar{d} & A \ar{d} \\
		A \ar{r} &  A \oplus \mathfrak M[1]
	\end{tikzcd}.
	\]
As $\bPerfSys(X)$ is cohesive and the right adjoint $\functor^\simeq \colon \Cat \to
\cS$ commutes with limits we obtain a pullback diagram of the form
	\[
	\begin{tikzcd}
		\PerfSys(X)(A \oplus \mathfrak M ) \ar{r} \ar{d} & \PerfSys(X)(A) \ar{d} \\
		\PerfSys(X) (A) \ar{r} & \PerfSys(A \oplus \mathfrak M [1] ).
	\end{tikzcd},
	\]
in $\cS$.
By taking the corresponding fibers at $\rho \in \bPerfSys(X)(A)$, we obtain a pullback diagram of spaces
	\[
	\begin{tikzcd}
		\bT_{\PerfSys(X), \rho}  (\mathfrak M) \ar{r} \ar{d} & * \ar{d} \\
		* \ar{r} & \bT_{\PerfSys(X), \rho}( \mathfrak M[1] )
	\end{tikzcd}	.
	\]
By our previous computations, replacing $M$ with the shift $M[1]$ produces the chain of
equivalences
	\begin{align*}
		p_M(\rho) & \simeq \Omega ( p_{\mathfrak M[1]} (\rho)) \\
		& \simeq \Map_{\bPerfSys(X)(A)} \big( \rho \otimes \rho^\vee, \mathfrak M [1] \big),
	\end{align*}
as desired.
\end{proof}

\begin{remark} Let $\rho \in \PerfSys(X)(Z)$, with $Z \in \dAfd_k$.
The above proposition implies that the tangent complex $\mathbb T^\mathrm{an}_{\PerfSys(X)}$ can be computed as the (derived) global sections
	\[
		\mathbb T^\mathrm{an}_{\PerfSys(X), \rho} \simeq \rmR \Gamma_\cont(X, \mathrm{Ad}(\rho))[1],
	\]
where the right hand side computes the continuous cohomology of $X$ with coefficients in $\Ad(\rho)$. This recovers the usual definition whenever $X = \rmB G$, where $G$ is a profinite group.
\end{remark}

In order to prove that $\PerfSys(X)$ admits a cotangent complex (dual to its tangent complex) one must impose further finiteness conditions on $X \in \pro(\cS^\fc)$:

\begin{definition}
Let $X \in \pro(\cS^\fc)$ be a connected profinite space. Let $A \in \adCAlg$ and $\rho \in \bPerfSys(X)(A)$.
We say that $X$ is \emph{locally $p$-adic cohomological perfect at $\rho$} if the enriched mapping $A$-module
	\[
		\Map_{\bPerfSys(X)(A)} \big( \mathbf{triv}_A , \rho \big) \in \Sp,
	\]
is a perfect $A$-module. We say that $X$ is \emph{$p$-adically cohomological perfect}
if it is locally $p$-adically cohomological perfect for every $A \in \adCAlg$ and $\rho \in \bPerfSys(X)(A)$.
\end{definition}

\begin{remark} \label{rem:p-adic_cohomological_perfect_on_each_reduction}
	Let $X \in \pro(\cS^\fc)$ be a connected profinite space. Let $A \in \adCAlg$ and $\rho \in \bPerfSys(X)(A)$. By $p$-adic completeness, in order to check that $X$ is locally $p$-adically cohomology perfect at $\rho$ it suffices to prove that, for every $n \ge 1$,
		\[
			\Map_{\bPerfSys(X)(A_n) }\big( \mathbf{triv}_{A_n}, \rho \otimes_{\mathbf{triv}_A} \mathbf{triv}_{A_n} \big) 
		\]
	is a perfect $A_n$-module.
\end{remark}
Under such assumption on $X$ we shall prove that the functor $\PerfSys(X) \colon \dAfd^\op_k \to \cS$ admits a well defined (global) cotangent complex. We first need an auxiliary lemma: 

\begin{lemma}  \label{lem:auxiliary_lemma_concerning_base_change_of_mapping_spaces_for_compact_objects} Let $A$ be a derived ring.
	Let $\cC$ be an $A$-linear stable presentable \infcat and $C \in
	\cC$ a compact object of $\cC$. Then for every object $M \in \Mod_A$, we have an equivalence of enriched mapping objects
		\[
			\mathbf{Map}_\cC \big(C, M \otimes \mathbf{1}_\cC \big) \simeq \mathbf{Map}_\cC \big( C, \mathbf 1_\cC \big) \otimes_A M,
		\]
	in the \infcat $\Mod_A$. Here $\mathbf 1_\cC \in \cC$ denotes the unit for the symmetric monoidal structure on $\cC$.
\end{lemma}

\begin{proof}
Let $\cD \subseteq \Mod_A$ denote the full subcategory spanned by those $A$-modules $M$ such that the
assertion holds true. Clearly $A \in \cD$. Since $A \in \Mod_A$ generates the \infcat $\Mod_A$ under
small colimits and shifts, it suffices to show that $\cD$ is closed under small colimits and shifts. Stability of $\cD$ under shifts is clear. Suppose now that 
	\[
		M \simeq \colim_{i \in I} M_i,
	\]
where $M_i \in \cD$ and $I$ is a filtered \infcat. Then by our compactness assumption it follows that
we have a chain of equivalences
	\begin{align*}
		\mathbf{Map}_\cC(C, M ) & \simeq \Map_\cC \big(C, \colim_i M_i \otimes \mathbf{1}_\cC \big) \\
		& \simeq \colim_i \mathbf{Map}_\cC \big( C, M_i \otimes \mathbf{1}_\cC \big) \\
		& \simeq \colim_i \mathbf{Map}_\cC \big(C, \mathbf 1_\cC \big) \otimes_A M_i \\
		&\simeq  \mathbf{Map}_\cC \big(C, \mathbf 1_\cC \big) \otimes_A (\colim_i M_i ) \\
		& \simeq \mathbf{Map}_\cC \big( C, \mathbf 1_\cC \big) \otimes_A M.
	\end{align*}
Thus $\cD$ is closed under filtered colimits. It suffices to show that $\cD$ is closed under finite
colimits. Since $\Mod_A$ is a stable \infcat it suffices to show that $\cD$ is closed under finite coproducts and
cofibers. Let 
	\[
		f \colon C \to D
	\]
be a morphism in $\cD$, we wish to show that $\cofib(f ) \in \cD$. Thanks to \cite[Theorem 1.1.2.14]{lurie2012higher}
we have an equivalence 
	\[
		\cofib(f) \simeq \fib(f)[1].
	\]
As a consequence, we can write 
	\begin{align*}
		\Map_\cC \big( C, \cofib(f) \big) & \simeq \Map_\cC \big( C, \fib(f) [1] \big) \\
		& \simeq \fib \big( \Map_\cC \big(C, f \big) \big)[1] \\
		& \simeq \cofib \big( \Map_\cC \big(C, f \big) \big) \\
		& \simeq \Map_\cC \big(C, \mathbf 1_\cC \big) \otimes_A \cofib(f).
	\end{align*}
The case of coproducts follows along the same lines and it is easier. From this we conclude that 
	\[
		\cD \simeq \Mod_A,
	\]
as desired. 
\end{proof}

\begin{proposition} \label{computation}
Let $X \in \pro(\cS^\fc)$ be a $p$-adically cohomological perfect profinite space. Then for every $Z \in \dAfd_k$
and every $\rho \in \bPerfSys(X)(Z)$ the $\rmR \Gamma(Z, \cO_Z)$-module
	\[
		\mathbf{Map}_{\bPerfSys(X)(Z)} \big( \mathbf{triv}_Z, \rho \otimes \rho^\vee [1] \big)^\vee \in \Mod_A,
	\]
is perfect and it corepresents the functor $F \colon \Coh(Z) \to \cS$, given by the formula
	\[
		M \in \Coh(Z) \mapsto \fib_\rho \big( \bPerfSys(X)(Z[ M ] )\to \bPerfSys(X)(Z) \big) \in \cS.
	\]
\end{proposition}

\begin{proof} The above formula is closed under Postnikov towers, so we must only prove the assertion in the case where $Z \in \dAfd_k$ is a eventually truncated derived $k$-affinoid space. Let $A \in \adCAlg$ be a derived $\kc$-adic algebra such that 
	\[
		\Spf(A)^\rig \simeq Z,
	\]
in $\dAfd_k$. Suppose first that $\rho \in \PerfSys(X)(Z)$ is liftable. And let $\rho' \in \bPerfSys(X)(A)$ lifting such $\rho$.
Notice that \cref{homotopy2} implies the existence of a natural equivalence
	\[
		\Map_{\bPerfSys(X)(Z)}(\mathbf{triv}, \rho \otimes \rho^\vee [1]) \simeq \colim_{\mathrm{mult} \ \mathrm{by} \ p} \Mat \big( \mathbf{Map}_{\bPerfSys(X)(A)}( \mathbf{triv}_A, \rho' \otimes (\rho')^\vee[1])  \big),
	\]
For this reason, in the case where $\rho$ is liftable, we reduce ourselves to prove the statement of the lemma for a lift $\rho' \in \bPerfSys(X)(A)$.
Denote by 
	\[
		\mathbf{ModSys}_{p}(X)(A) \coloneqq \ind(\bPerfSys(X)(A)).
	\]
The latter is a presentable stable \infcat which is furthermore $A$-linear. 
\cref{lem:auxiliary_lemma_concerning_base_change_of_mapping_spaces_for_compact_objects} implies that,
for every $M \in \Mod_A$, we have a chain of equivalences
	\begin{align*}
		\mathbf{Map}_{\mathbf{ModSys}_{p}(X)(A)} \big( \rho \otimes \rho^\vee, M [1] \big) & \simeq \mathbf{Map}_{\mathbf{ModSys}_{p}(X)(A)}
		\big(  \rho \otimes \rho^\vee , \mathbf{triv}_A [1] \big) \otimes_A M \\
		& \simeq \mathbf{Map}_{\mathbf{ModSys}_{p}(X)(A)} \big( \mathbf{triv}_A, \rho \otimes \rho^\vee [1] \big) \otimes_A M.
	\end{align*}
Furthermore, it follows by our hypothesis on $X$ that
	\[
		C \coloneqq \mathbf{Map}_{\mathbf{ModSys}_{p}(X)(A)} \big(\mathbf{triv}_A, \rho \otimes \rho^\vee[1] \big) \in \Mod_A,
	\]
is a perfect $A$-module. Assemblying the above considerations together we thus obtain a chain of natural equivalences
	\begin{align*}
		\Map_{\Mod_A} \big(C^\vee, M \big) & \simeq \Map_{\Mod_A} \big( A, C \otimes_A M \big) \\
		& \simeq \Map_{\bPerfSys(X)(A)} \big( \mathbf{triv}_A , \rho' \otimes (\rho')^\vee[1]) \otimes_A M,
	\end{align*}
and the result now follows by the explicit formula for the tangent complex provided in \cref{prop:tangent_complex_of_PerfSys}. For a general $\rho \in \PerfSys(X)(Z)$, \cref{homotopy2} implies that
	\[
		\mathbf{Map}_{\bPerfSys(X)(Z)} \big( \mathbf{triv}_{\rmR \Gamma(Z, \cO_Z)}, \rho \otimes \rho^\vee [1] \big)^\vee,
	\]
is a retract of some $\mathbf{Map}_{\bPerfSys(X)(Z)} \big( \mathbf{triv}_{\rmR \Gamma(Z, \cO_Z)}, \rho' \otimes (\rho ')^\vee [1] \big)^\vee$, for some liftable $\rho'$. The result now follows from the fact that perfect $\rmR \Gamma(Z, \cO_Z)$-modules are closed under retracts.
\end{proof}

\subsection{\'Etale cohomology of perfect local systems} As we saw in the previous section in order to construct the cotangent complex of $\PerfSys(X)$ one needs to require $X$ to be $p$-adically perfect. In this \S, we will show that such condition
is always verified in the case where $X$ is the \'etale homotopy type of a smooth and proper scheme.

Let us fix notations: from now on we shall denote by $X$ a geometrically connected proper and smooth scheme over an algebraically closed field $K$.
Let 
	\[
		\pi_X: X \to \mathrm{Spec} K
	\]
denote the structural morphism. For each integer $n \geq 1$, we have a canonical equivalence of $\infty$-categories 
	\[
		\mathrm{Shv}( \mathrm{Spec} K, \mathbb{Z}/ p^{n} 
		\mathbb{Z}) \simeq \mathrm{Mod}_{\mathbb{Z}/ p^n \mathbb{Z}}.
	\] 
We have a pullback functor 
	\[
		p^* : \mathrm{Shv}_{\text{\'et}}( \mathrm{Spec} K, \mathbb{Z} / p^n \mathbb{Z} ) \to  \mathrm{Shv}_{\text{\'et}}( X, \mathbb{Z} / p^n \mathbb{Z} ) ,
	\]
which associates to each $\mathbb{Z}/ p^n \mathbb{Z}$-module $M$ the \'etale
constant sheaf on $X$ with values in $M$.

\begin{proposition} \label{finiteness_et_coho}
	Let $X$ be a proper normal scheme over an algebraically closed field $K$. Then $\mathrm{R} \Gamma (X_{\acute{\mathrm{e}} \mathrm{t}}, \mathbb{Z}/ p^n \mathbb{Z})$ is a perfect complex of $\mathbb{Z} / p^n \mathbb{Z}$-modules.
\end{proposition}

\begin{proof}
	This is a direct consequence of the more general results \cite[Proposition 4.2.15]{gaitsgory2014weil} in the $\ell \neq p$ case and \cite[Theorem 5.1]{scholze2013p} in the $p$-adic case.
\end{proof}

\begin{definition}
	Let $A$ be a derived ring. We say that $A$ is \emph{Noetherian} if it satisfies the following conditions:
	\begin{enumerate}
		\item $\pi_0(A)$ is a Noetherian ring;
		\item For each $i \geq 0$, $\pi_i(A)$ is an $\pi_0(A)$-module of finite type. 
	\end{enumerate}
\end{definition}

\begin{definition}
	Let $A$ be a derived ring and $M \in \Mod_A$ an $A$-module. We say that $M$ has \emph{tor-amplitude $\leq n$} if, for every discrete $A$-module $N$, the homotopy groups 
		\[
			\pi_i(M \otimes_A N)  \in \Mod_A
		\] 
	vanish for every integer $i >n $.
\end{definition}

\begin{lemma} \label{tor_ampl_perf}
Let $A$ be a Noetherian simplicial ring and $M \in \mathrm{Mod}_A$ be a connective $A$-module. Then $M$ is a perfect $A$-module if and only if the following two conditions are satisfied:
\begin{enumerate}
\item For each $i$, $\pi_i(M)$ is of finite type over $\pi_0(A)$;
\item $M$ is of finite Tor-dimension.
\end{enumerate}
\end{lemma}
\begin{proof}
	This is \cite[Proposition 7.2.4.23]{lurie2012higher}.
\end{proof}

\begin{lemma}[Projection Formula] \label{lem:proj_formula}
Let $X$ be a scheme over an algebraically closed field $K$. Let $A$ be a simplicial $\bZ/ p^n \bZ$-algebra and $\mathcal{F} \in \mathrm{Shv}_{\emphet}( X, A )$. Then, for any $M \in \mathrm{Mod}_A$, we have a natural equivalence,
	\[
		\pi_*( \mathcal{F} \otimes_A \pi^*( M) ) \simeq \pi_*(\mathcal{F}) \otimes_A M,
	\]
in the derived $\infty$-category $\mathrm{Mod}_A$, where $\pi$ denotes the structural morphism $\pi \colon X \to \Spec K$.
\end{lemma}

\begin{proof}
Let $\mathcal{C} \subset \mathrm{Mod}_{A}$ be the full subcategory spanned by those $A$-modules $M$ such that there exists a canonical equivalence \[\mathrm{R} \Gamma( X_{\text{\'et}}, M) \simeq \mathrm{R} \Gamma(X_{\text{\'et}}, A) \otimes_{A} M.\] It is 
clear that $A \in \mathcal{C}$ and $\mathcal{C}$ is closed under small colimits, since both tensor product and the direct image functor $\pi_*$ commute with small colimits. Furthermore, the fact that the $\infty$-category $\mathrm{Mod}_A
$ is compactly generated by the unit object $A$ implies the result.
\end{proof}

\begin{definition} Let $A \in \CAlg_{\bZ/ p^n \bZ}$ be a simplicial $\bZ/ p^n \bZ$-algebra.
	Let $\cF \in \Shv(X_{\et}, A)$. We say that $\cF$ is a \emph{local system of perfect $A$-modules on} $X_\et$ if there exists an \'etale covering $f \colon Y \to X$ such that
		\[
			f^*(\cF) \simeq \pi_Y^* (N),
		\]
	where $N \in \Mod_A$ is a perfect $A$-module and $\pi_Y \colon Y \to \Spec K$ denotes the structural morphism.
\end{definition}

\begin{remark}
Let $A$ be a derived $\mathbb{Z}/ p^n \mathbb{Z}$-algebra and let $N \in \Shv(X_{\text{\'et}}, A)$ be a local system of perfect $A$-modules on $X_{\text{\'et}}$. Thanks to \cite[Proposition 4.2.2]{gaitsgory2014weil} it follows that $N$
can be written as a (finite sequence) of retracts of the form
	\[
		f_{!}(A) \in \Shv(X_\et, A ) ,
	\]
where
	\[
		f_! \colon \Shv( Y_{\text{\'et}}, \mathbb{Z}/ p^n \mathbb{Z}) \to \Shv(X_{\text{\'et}}, \mathbb{Z}/ p^n \mathbb{Z}),
	\]
denotes the exceptional direct image functor associated to a given \'etale map $f \colon Y \to X$, see \cite[Example 4.1.10]{gaitsgory2014weil} for a definition of the latter. We will not need this result in our study.
\end{remark}

\begin{remark}
	Thanks to \cite[Remark  4.1.6]{gaitsgory2014weil} the \infcat $\Shv_{\et}(X, \bZ / p^n \bZ)$ admits a canonical t-structure.
\end{remark}

\begin{lemma} \label{lem:t-exactness_of_etale_pullback}
	Let $f \colon Y \to X $ denote an \'etale morphism. Then the pullback functor \[f^* \colon \Shv_{\emph{\et}}(X, \bZ/ p^n \bZ) \to \Shv_{\emph{\et}}(Y, \bZ / p^n \bZ),\] is t-exact.
\end{lemma}

\begin{proof}
	Let $\cF \in \Shv_{\et}(X, \bZ/ p^n \bZ)^\heartsuit$ denote a discrete \'etale sheaf on $X$. Then $f^*(\cF) \in \Shv_{\et}(Y, \bZ/ p^n \bZ)$ is on objects by the formula
		\[
			(U \to Y) \in Y_{\et} \mapsto \cF(U) \in \Mod_{\bZ/ p^n \bZ}.
		\]
	It is then clear that $f^*(\cF)$ is a discrete \'etale sheaf on $Y$, as well. The claim now follows.
\end{proof}

\begin{proposition} \label{perfect_etale}
Let $A \in \CAlg_{\mathbb{Z}/ p^n \mathbb{Z}}$ be a Noetherian simplicial $\mathbb{Z}/ p^n \mathbb{Z}$-algebra.
Let $\cF$ be a local system of perfect $A$-modules on $X_{\acute{\mathrm{e}}\mathrm{t}}$. Then the (derived) global sections 
$ \mathrm{R} \Gamma ( X_\emphet, \cF) \in \Mod_A$, is a perfect $A$-module.
\end{proposition}

\begin{proof}
Let $\cF$ be a local system of perfect $A$-modules on $X_{\text{\'et}}$. Our goal is 
to show that 
	\[
		\mathrm{R} \Gamma(X_{\text{\'et}}, \cF) \in \Mod_A,
	\]
lies in the full subcategory $\Perf(A)$.
By \cref{tor_ampl_perf} it suffices to show that, for each $i \in \mathbb{Z}$, the cohomology groups
	\[
		H^{-i}(X_{\text{\'et}}, \cF) := \pi_i(\mathrm{R} \Gamma(X_{\text{\'et}}, \cF)) ,
	\]
are finite
$\pi_0(A)$-modules and 
	\[
		\mathrm{R} \Gamma(X_{\text{\'et}},\cF) \in \Mod_A,
	\]
is of finite Tor-amplitude over $A$. Up to performing a suitable shift of $\rmR \Gamma(X_\et, \cF)$,
we can assume that $\cF$ is a connective perfect $A$-module on $X_{\text{\'et}}$. In order to prove the claim,
we will directly verify the conditions of \cref{tor_ampl_perf}. We start by verifying condition i) in loc. cit.:

Let $f \colon Y \to X$ be an \'etale morphism such that
	\[
		f^*(\cF) \simeq \pi^*_Y(N),
	\]
for some $N \in \Perf(A)$ and let $\pi_Y \colon Y \to \Spec K$ denote the structural morphism. Thank to \cref{lem:t-exactness_of_etale_pullback}, for each $i \ge 0$, one has a chain of natural equivalences
	\begin{align*}
		f^*(\pi_i(\cF))	 & \simeq \pi_i(f^*(\cF)) \\
					 & \simeq \pi_i(\pi^*_Y(N)) \\
				 	 & \simeq \pi_Y^*(\pi_i(N)),
	\end{align*}
where the last equivalence follows from \cite[Warning 4.1.8]{gaitsgory2014weil} and the definition of
	\[
		\pi_Y^* \colon \Mod_{\bZ/ p^n \bZ} \to \Shv_{\et}(X, \bZ / p^n \bZ).
	\]
For this reason, $\pi_i(\cF)$ is still a local system on $X_\et$.
As $X$ 
is smooth it is in particular a normal scheme. Consequently, we can assume without loss of generality, that the \'etale map $f \colon Y 
\to X$ is a Galois covering. 
It follows, by Galois descent, that we have a natural equivalence
	\[
		\mathrm{R} 
		\Gamma(X_{\text{\'et}}, \pi_i(\cF)) \simeq \mathrm{R} \Gamma(G, 
		\mathrm{R} \Gamma_{\mathrm{grp}}( Y_{\text{\'et}}, \pi_Y^*(\pi_i(N)))) ,
	\]
of $A$-modules,
where $G$ denotes the finite 
group of automorphisms of the Galois covering $f \colon Y \to X$ and $\mathrm{R} \Gamma_{\mathrm{grp}}(G, -)$ the functor of group cohomology associated to $G$.

Assume 
first that $\mathrm{R} \Gamma(Y_{\text{\'et}},\pi_Y^*(\pi_i(N)))$ is an $A$-module whose homotopy groups are finitely generated over $\pi_0(A)$. Granting this assumption, it follows by the fact that
$G$ is finite, that the group cohomology of $G$ with $
\mathbb{Z}$-coefficients is finitely generated and of torsion. By analyzing the corresponding Hochschild-Serre spectral sequence
	\begin{equation} \label{eq:Hochschild_Serre_sp}
		\pi_j( \rmR \Gamma_{\mathrm{grp}}(G, \pi_i(\rmR \Gamma(Y_\et, \pi_Y^*(N)))) \Rightarrow \pi_{i+j}(\rmR \Gamma_{\mathrm{grp}}(G, \rmR \Gamma(Y_\et, \pi_Y^*(N)))),
	\end{equation}
we deduce that the homotopy groups 
	\[
		\pi_i(\mathrm{R} \Gamma(G, \mathrm{R} \Gamma( V_{\text{\'et}}, \pi_Y^*(N))) ,
	\]
are finitely generated over $\pi_0(A)$, as well. The above reasoning allow us to reduce to the case where $\cF $ is constant, i.e., 
	\[
		\cF \simeq \pi_X^*(N),
	\]
with $N \in \Perf(A)$. By the projection formula,
\cref{lem:proj_formula}, we have that
	\begin{equation} \label{eq:proj_formula_applied_to_constant_coefficients}
		\rmR \Gamma(X_\et, \pi_X^*(\pi_i(N))) \simeq \rmR \Gamma(X_\et, \pi_0(A)) \otimes_{\pi_0(A)} \pi_i(N).
	\end{equation}
Thus to prove finiteness of the left hand side of \eqref{eq:proj_formula_applied_to_constant_coefficients} it suffices to prove it in the case where $\pi_0(A) \simeq \pi_i(N)$.
By the same reasoning, we have that 
	\[
		\rmR \Gamma(X_\et, \pi_0(A)) \simeq \rmR \Gamma(X_\et, \bZ/ p^n \bZ) \otimes_{\bZ / p^n \bZ} \pi_0(A).
	\]
Thus, to prove the claim, we are reduced to know that $\rmR \Gamma(X_\et, \bZ/ p^n \bZ)$ has finitely generated homotopy groups. The latter claim is now a consequence of \cref{finiteness_et_coho}.

Consider now the fiber sequence associated to the $n$-th step of the Postnikov tower for $\cF$,
	\[
		 \tau_{\leq n+1} \cF \to \tau_{\leq n } \cF \to  \pi_{n+1} (\cF) [n+2].
	\]
By induction we conclude that, for every $n \geq 0$, both complexes
	\[
		\mathrm{R} \Gamma( X_{\text{\'et}}, \tau_{\leq n} \cF) \quad \text{and } \mathrm{R} \Gamma( X_{\text{\'et}}, \pi_{n+1} (\cF)) [n+2] \in \Mod_A
	\]
have finitely generated homotopy $\pi_0(A)$-modules. Since the functor $
\mathrm{R} \Gamma(X_{\text{\'et}}, -)$ is exact, it follows that also the complex 
	\[
		\mathrm R \Gamma(X_{\textrm{\'et}}, \tau_{\leq n+1} \cF ),
	\]
has finite type homotopy groups, as well.
Consider now the fiber sequence
	\[
		\tau_{> n } \cF \to \cF \to \tau_{ \leq n } \cF,
	\]
of \'etale sheaves.
The \'etale site of $X$ is of finite cohomological dimension, cf.\cite[Tag 095U, Lemma 88.2]{destacks}. It thus follows that, for any given integer $i$, there is a sufficiently large integer $n$ such that 
	\[
		\pi_i(\mathrm{R} \Gamma(X_{\textrm{\'et}}, \tau_{>n} (\cF)) \in \Mod_{\pi_0 (A)},
	\]
vanishes. Therefore, by exactness of $\rmR \Gamma(X_\et, - )$ we conclude that
\[\pi_i(\rmR \Gamma(X_\et, \cF) ) \simeq \pi_i(\rmR \Gamma(X_\et, \tau_{\le n}(\cF)),\] for some sufficiently large $n$. We conclude thus that $\pi_i(\rmR \Gamma(X_\et, \cF ))$ is a finitely generated $\pi_0(A)$-module, for every $i \in \bZ$.
 This establishes condition (i) of \cref{tor_ampl_perf}, in general. In order to establish the Lemma we are left to prove that
 	\[
		\rmR \Gamma( X_\et, \cF),
	\]
is a finite tor-amplitude $A$-module.
Let $M \in \Mod_A$ denote a discrete $A$-module, (which can be naturally considered as a $\pi_0(A)$-module).
By\cite[Tag 095U, Lemma 88.2]{destacks}, it follows that the $A$-module 
	\[
		\pi_{X, *}  \pi_X^*(M) \simeq \rmR \Gamma(X_\et, \pi_X^*(M)),
	\]
has non-zero homotopy groups lying in a fixed finite set of indices. Let $f \colon Y \to X$ be a Galois covering of $X$, of finite group $G$ of automorphisms, such that 
	\[
		f^*(\cF) \simeq \pi_Y^*(N),
	\]
for some $N \in \Perf(A)$. Moreover, we can always assume $Y$ to be connected (otherwise, we take a connected component of $U$ mapping surjectively onto $X$). Then, by means of the projection formula, we conclude that
	\begin{align*}
		\rmR \Gamma(Y_\et, f^*(\cF)) \otimes_A M & \simeq \rmR \Gamma(Y_\et, f^*(\cF) \otimes_A \pi_Y^*(M)) \\
										  & \simeq \rmR \Gamma(Y_\et, \pi_Y^*(N) \otimes_A \pi_Y^*(M)),
	\end{align*}
can be obtained by a finite sequence of retracts and finite colimits of $\pi_{X, *}  \pi_X^*(M)$, since $N$ belongs to the full subcategory $\Perf(A) \subseteq \Mod_A$ generated by $A$ under finite colimits and retracts. By the previous paragraph, we conclude that
the non-zero homotopy groups of the $A$-module
	\[
		\rmR \Gamma(Y_\et, f^*(\cF) \otimes_A \pi_Y^*(M)),
	\]
lie in a finite set of indices, whose amplitude is bounded by $2\mathrm{dim} (X)$ plus the finite tor-amplitude of $N \in \Perf(A)$. We now observe that we have a sequence of natural equivalences
	\begin{align*}
		\rmR \Gamma(X_\et, \cF) \otimes_A M & \simeq \rmR \Gamma(X_\et, \cF \otimes_A \pi_X^*(M)) \\
									    & \simeq \rmR \Gamma_{\mathrm{grp}}(G, \rmR \Gamma (Y_\et,f^*( \cF )\otimes_A  \pi_Y^*(M))),
	\end{align*}
of $A$-modules. Since group cohomology is a left t-exact functor it follows from the preceding paragraph that the above module is bounded on the left. We further deduce
that, for each $i $,
	\begin{align*}
		\rmR \Gamma_{\mathrm{grp}}(G, \pi_i( \rmR \Gamma (Y_\et,f^*( \cF )\otimes_A  \pi_Y^*(M)))) & \simeq \rmR \Gamma (X_\et, \cF \otimes_A \pi_X^*(M)) \\
																			       & \simeq \rmR \Gamma(X_\et, \cF ) \otimes_A M
	\end{align*}
is bounded on the right by the bound $2 \mathrm{dim}(X) + 1$, cf.\cite[Tag 095U, Lemma 88.2]{destacks}.  The result now follows by the previous discussion, in constant case, and by analyzing the associated Hochschild-Serre spectral sequence displayed in \eqref{eq:Hochschild_Serre_sp}.
\end{proof}

\begin{corollary}
	The profinite space $\Sh_\emph{\et}(X) \in \pro(\cS^\fc)$ is $p$-adically cohomological perfect.
\end{corollary}

\begin{proof}
	It is an immediate consequence of \cref{perfect_etale} together with \cref{rem:p-adic_cohomological_perfect_on_each_reduction}.
\end{proof}

\subsection{Main results}
Let $X$ be a geometrically connected proper and smooth scheme over an algebraically closed field.
To such $X$ we can associate it a profinite space, namely its \'etale homotopy
type
	\[
		\Sh^\et(X) \in \pro(\cS^\fc),
	\]
cf. \cite[\S 3.6]{2009derived}.
By construction, $\Sh^\et(X) \in \pro(\cS^\fc)$ classifies \'etale local systems on
$X$. Moreover, we have a canonical identification
	\[
		\pi_1 \big( \Sh^\et(X) \big) \simeq \pi_1^\et(X) ,
	\]
of profinite groups. Therefore, it is natural to consider the moduli stack
	\[
		\PerfSys \big(\Sh^\et(X) \big) \in \dSt \big(\dAfd_k, \tau_\et \big),
	\]
as a derived extension of the moduli $\LocSys(X) \in \St \big( \Afd_k, \tau_\et,
\rmP_\sm \big)$.

\begin{definition}
Let $\dLocSys(X) \subseteq \PerfSys \big( \Sh^\et(X) \big)
$ denote the substack spanned by continuous $p$-adic representations of
$\Sh^\et(X) \in \pro(\cS^\fc)$, with values in rank $n$ free modules.
\end{definition}

\begin{proposition} \label{t1}
We have a canonical equivalence of stacks
	\[
		\trun_{\leq 0} ( \dLocSys(X) ) \simeq \LocSys(X),
	\]
in the \infcat $\St( \Afd_k , \tau_\emphet )$.
\end{proposition}

\begin{proof}
Let $A \in \Afd_k$, then $\dLocSys(X)(\Sp (A))$ can be identified with the space
	\[
		\dLocSys(X)(\Sp(A)) \simeq \Map_{\Mon_{\bE_1}(\ind(\pro(\cS)))}
		\big( \Sh^\et(X), \rmB \cEnd(A) \big),
	\]
of continuous $\bE_1$-monoid like morphisms 
	\[
		\Sh^\et(X) \to \rmB \cEnd(A).
	\]
As $A \in \Afd_k$ is discrete, it follows that the underlying ind-pro-space of $\rmB \cEnd(A)$ is discrete. Therefore, by the universal property of $1$-st truncation we have
a chain of equivalences
	\begin{align*}
		\Map_{\Mon_{\bE_1}(\ind(\pro(\cS)))}
		\big(  \Sh^\et(X), \rmB \cEnd (A) \big) & \simeq \\
		& \simeq \Map_{\Mon_{\bE_1}(\ind(\pro(\cS)))}
		\big( \tau_{\leq 1} (\Sh^\et(X)), \rmB  \cEnd (A) \big) \\
		& \simeq \Map_{\Mon_{\bE_1}(\ind(\pro(\cS)))}
		\big( \rmB \pi_1^\et(X), \rmB \cEnd (A) \big) \\
		& \simeq \Map_{\Mon^{\mathrm{grp}}_{\bE_1}(\ind(\pro(\cS)))} \big( \rmB \pi_1^\et(X), 
		\rmB \GLn(A) \big) ,
	\end{align*}
where the last equivalence follows from the fact that $\pi_1^\et(X)$ is an actual group object in $\Mon_{\bE_1}(\ind(\pro(\cS)))$. Indeed,
every morphism $\pi_1^\et(X) \to \cEnd(A)$ factors through the
sub-group of units of $\cEnd(A)$. The latter coincides precisely with $\GLn(A)$, equipped with the induced $k$-analytic
topology. For this reason, we have a natural morphism
	\[
		\theta \colon \Hom_{\cont}(\pi_1^\et(X), \GLn(A)) \to \Map_{\Mon^{\mathrm{grp}}_{\bE_1}(\ind(\pro(\cS)))} \big( \rmB \pi_1^\et(X), 
		\rmB \GLn(A) \big),
	\]
obtained by sending each continuous morphism $f \colon \pi_1^\et(X) \to \GLn(A)$ to the corresponding morphism of $\bE_1$-monoid group like objects
	\[
		\rmB \pi_1^\et(X) \to \rmB \GLn(A),
	\]
in $\ind(\pro(\cS))$. Moreover, it is clear that the morphism $\theta$ is invariant under $\GLn(A)$-conjugation. Therefore, we obtain a well defined induced morphism
	\[
		\theta \colon \LocSys(X)(A) \to \Map_{\Mon^{\mathrm{grp}}_{\bE_1}(\ind(\pro(\cS)))} \big( \rmB \pi_1^\et(X), 
		\rmB \GLn(A) \big).
	\]
Let $\rho \in \Map_{\Mon^{\mathrm{grp}}_{\bE_1}(\ind(\pro(\cS)))} \big( \rmB \pi_1^\et(X), 
		\rmB \GLn(A) \big).$ Since $\GLn(A)$ is discrete, there exists a continuous group homomorphism $f \colon \pi_1^\et(X) \to \GLn(A)$ which realizes $\rho$, under $\theta$. Moreover, such choice only depends on the conjugacy class of $f$. The result follows.
\end{proof}

\begin{theorem} \label{t2}
The moduli stack $\dLocSys(X) \in \dSt(\dAfd_k, \tau_\emphet)$ admits a cotangent
complex. Given $\rho \in \dLocSys(X) ( Z)$, with $Z \in \dAfd_k$, we have a natural equivalence 
	\begin{align*}
		\bL^\an_{\LocSys(X), \rho} & \simeq C^*_\emphet(X, \Ad(\rho) ,\big)^\vee [-1],
	\end{align*}
in $\Mod_{\rmR \Gamma(Z, \cO_Z)}$, where $C^*_\emphet \big( X, \Ad(\rho) \big)$ denotes the \'etale cohomology of $X$
with coefficients in the adjoint representation
	\[
		\Ad (\rho) \coloneqq \rho \otimes \rho^\vee.
	\]
\end{theorem}

\begin{proof}
Since $X$ is smooth and proper over an algebraically closed field it follows that $\Sh^\et(X)$ is
$p$-cohomologically compact and $p$-cohomologically perfect. Therefore,
$\PerfSys(X)$ admits a cotangent complex and by restriction so does
$\LocSys(X)$. Moreover, the tangent complex of $\LocSys(X)$ at the morphism
	\[
		\rho \colon Z \to \LocSys(X)
	\]
can be identified with the mapping space
	\[
		\bT^\an_{\LocSys(X), \rho} \simeq \Map_{\bPerfSys(X)(Z)} \big(
		\mathbf{triv}_{\rmR \Gamma(Z, \cO_Z)}, \rho \otimes \rho \big) [1].
	\]
We are thus reduced to prove that 
	\[
	\Map_{\bPerfSys(X)(Z)} \big(
		\mathbf{triv}_{\rmR \Gamma(Z, \cO_Z)}, \rho \otimes \rho \big)[1] \simeq C^*_\et(X, \Ad(\rho) \big) [1].
	\]
But this follows by the universal property of $\Sh^\et(X)$ together with the fact
that global sections of local systems with torsion coefficients on $\Sh^\et(X)$
classify \'etale cohomology on $X$, with torsion coefficients, \cite[Theorem 4.3]{hoyois2018higher}. The result follows
now for liftable such $\rho$ and for general $\rho$ by \cref{homotopy1}.
\end{proof}

\begin{proposition} \label{t3}
The moduli stack $\dLocSys(X)$ is cohesive and nilcomplete.
\end{proposition}

\begin{proof}
This is a direct consequence of the analogous statement for $\PerfSys(X)$. 
\end{proof}

As a consequence we obtain our main result:

\begin{theorem} \label{thm:final_thm}
The moduli stack $\LocSys(X) \in \dSt \big(\dAfd_k, \tau_\emphet \big)$ is
representable by a derived $k$-analytic stack.
\end{theorem}

\begin{proof}
The proof follows by the Representability theorem, \cref{thm:representability}, together with \cref{t1}, \cref{pr:Hom_G},
\cref{t2} and \cref{t3}.
\end{proof}
\bibliography{Introduction}
\bibliographystyle{alpha}
\adress
\end{document}

%% file: newcommands.tex
\newcommand{\GLn}{\mathrm{GL}_n}

\newcommand{\anGLn}{\mathbf{GL}_n^\an}

\newcommand{\fib}{\mathrm{fib}}
\newcommand{\dLocSys}{\mathrm R \mathrm{LocSys}_{p, n}}

\newcommand{\LocSys}{\mathrm{LocSys}_{p ,n}}

\newcommand{\LocSysfr}{\LocSys^{\square}}

\newcommand{\fr}{\widehat{\rmF}_r}
\newcommand{\id}{\mathrm{id}}
\newcommand{\Idem}{\mathrm{Idem}}

\newcommand{\cofib}{\mathrm{cofib}}

\newcommand{\sm}{\mathrm{sm}}

\newcommand{\nil}{\mathrm{nil}}

\newcommand{\kc}{\cO_k}
\newcommand{\triv}{\mathbf{triv}}

\newcommand{\Map}{\mathrm{Map}}
\newcommand{\Hom}{\mathrm{Hom}}
\newcommand{\ad}{\mathrm{ad}}

\newcommand{\Ok}{\cO_k}

\DeclareMathOperator{\Spec}{Spec}
\DeclareMathOperator{\Spf}{Spf}
\DeclareMathOperator{\Sp}{Sp}
\newcommand{\st}{\mathrm{st}}

\newcommand{\trun}{\mathrm{t}}

\newcommand{\cEnd}{\mathbf{End}}

\newcommand{\Mat}{\mathrm{Mat}}

\newcommand{\Mon}{\mathrm{Mon}}

\newcommand{\dSt}{\mathrm{dSt}}

\newcommand{\fc}{\mathrm{fc}}
\newcommand{\Sh}{\mathrm{Sh}}
\newcommand{\Ad}{\mathrm{Ad}}

\newcommand{\cont}{\mathrm{cont}}

\newcommand{\CAlg}{\mathcal{C}\mathrm{Alg}}
\newcommand{\adCAlg}{\cC \mathrm{Alg}^{\ad}_{\Ok}}

\newcommand{\Cat}{\mathcal{C}\mathrm{at}_\infty}
\newcommand{\Coh}{\mathrm{Coh}^+}

\newcommand{\op}{\mathrm{op}}

\newcommand{\Shv}{\mathrm{Shv}}
\newcommand{\ind}{\mathrm{Ind}}
\newcommand{\pro}{\mathrm{Pro}}
\newcommand{\Perf}{\mathrm{Perf}}
\newcommand{\infcat}{$\infty$-category\xspace}
\newcommand{\infcats}{$\infty$-categories\xspace}

\newcommand{\PerfSys}{\mathrm{PerfSys}_{p}}
\newcommand{\bPerfSys}{\mathbf{PerfSys}_{p}}

\newcommand{\et}{\text{\'et}}
\newcommand{\emphet}{\emph{\text{\'et}}}

\newcommand{\functor}{(-)}
\newcommand{\rigg}{(-)^{\mathrm{rig}}}
\newcommand{\rig}{\mathrm{rig}}

\newcommand{\Fun}{\mathrm{Fun}}

\newcommand{\Mod}{\mathrm{Mod}}
\newcommand{\St}{\mathrm{St}}

\newcommand{\Afd}{\mathrm{Afd}}

\newcommand{\An}{\mathrm{An}}

\newcommand{\dAfd}{\mathrm{dAfd}}

\newcommand{\dAn}{\mathrm{dAn}}

\newcommand{\an}{\mathrm{an}}

\DeclareMathOperator*{\colim}{colim}

\newcommand{\rmB}{\mathrm{B}}

\newcommand{\rmF}{\mathrm F}

\newcommand{\rmH}{\mathrm H}

\newcommand{\rmM}{\mathrm M}
\newcommand{\rmP}{\mathrm P}
\newcommand{\rmR}{\mathrm R}

\newcommand{\bA}{\mathbb A}

\newcommand{\bE}{\mathbb E}

\newcommand{\bL}{\mathbb L}

\newcommand{\bQ}{\mathbb Q}
\newcommand{\bR}{\mathbb{R}}
\newcommand{\bT}{\mathbb T}
\newcommand{\bZ}{\mathbb Z}

\newcommand{\cC}{\mathcal C}
\newcommand{\cD}{\mathcal D}

\newcommand{\cF}{\mathcal{F}}

\newcommand{\cO}{\mathcal{O}}

\newcommand{\cT}{\mathcal{T}}
\newcommand{\cV}{\mathcal{V}}
\newcommand{\cX}{\mathcal X}
\newcommand{\cY}{\mathcal Y}

\newcommand{\cS}{\mathcal S}

